\documentclass[11pt,a4paper]{article}
\usepackage{amsmath,amsfonts,amssymb,amsthm,enumerate,graphicx,mathtools,tikz,titling,caption,tabularx}
\usepackage{hyperref}
\usepackage[margin=2.5cm]{geometry}

\usepackage{xypic} 
\usepackage{mathabx}
\usepackage{mathrsfs} 
\usepackage{xspace}
\usepackage{fixltx2e}

\usepackage{authblk}
\usepackage{tocloft}
\usepackage[normalem]{ulem}

\usepackage{pdflscape}
\usepackage[toc,page,header]{appendix}
\usepackage{listings}

\usepackage{color}

\newtheorem{thm}{Theorem}[section]
\newtheorem{lem}[thm]{Lemma}
\newtheorem{prop}[thm]{Proposition}
\newtheorem{cor}[thm]{Corollary}

\theoremstyle{definition}
\newtheorem{defn}[thm]{Definition}
\newtheorem{ex}[thm]{Example}

\newtheorem{rem}[thm]{Remark}
\newtheorem{hyp}[thm]{Hypothesis}

\newcommand{\ol}{\overline}

\newcommand{\defbold}{\textbf}

\newcommand{\inv}{^{-1}}

\newcommand{\N}{\mathrm{N}}

\newcommand{\Z}{\mathrm{Z}}

\newcommand{\triv}{\{1\}}

\newcommand{\Aut}{\mathrm{Aut}}

\newcommand{\Alt}{\mathrm{Alt}}
\newcommand{\Sym}{\mathrm{Sym}}

\newcommand{\SL}{\mathrm{SL}}
\newcommand{\PSL}{\mathrm{PSL}}
\newcommand{\GL}{\mathrm{GL}}
\newcommand{\PGL}{\mathrm{PGL}}
\newcommand{\AGL}{\mathrm{AGL}}
\newcommand{\GaL}{\Gamma\mathrm{L}}
\newcommand{\PGaL}{\mathrm{P}\Gamma\mathrm{L}}
\newcommand{\AGaL}{\mathrm{A}\Gamma\mathrm{L}}

\newcommand{\Fb}{\mathbb{F}}

\newcommand{\Qb}{\mathbb{Q}}

\newcommand{\Zb}{\mathbb{Z}}

\newcommand{\mc}[1]{\mathcal{#1}}

\newcommand{\grp}[1]{\langle #1 \rangle}

\begin{document}

\title{Multiple transitivity except for a system of imprimitivity}

\preauthor{\large}
\DeclareRobustCommand{\authoring}{
\renewcommand{\thefootnote}{\arabic{footnote}}
\begin{center}Colin D. Reid\textsuperscript{1}\footnotetext[1]{Research supported in part by ARC grant FL170100032.}\\ \bigskip
The University of Newcastle \\ School of Information and Physical Sciences \\ Callaghan, NSW 2308, Australia \\ Email: \href{mailto:colin@reidit.net}{colin@reidit.net} 
\end{center}
}
\author{\authoring}
\postauthor{\par}

\maketitle

\begin{abstract}
Let $\Omega$ be a set equipped with an equivalence relation $\sim$; we refer to the equivalence classes as blocks of $\Omega$.  A permutation group $G \le \Sym(\Omega)$ is \defbold{$k$-by-block-transitive} if $\sim$ is $G$-invariant, with at least $k$ blocks, and $G$ is transitive on the set of $k$-tuples of points such that no two entries lie in the same block.  The action is \defbold{block-faithful} if the action on the set of blocks is faithful.

In this article we classify the finite block-faithful $2$-by-block-transitive actions.  We also show that for $k \ge 3$, there are no finite block-faithful $k$-by-block-transitive actions with nontrivial blocks.
\end{abstract}

\section{Introduction}

Given a group $G$ acting on a set $X$ and $x_1,x_2,\dots, x_n \in X$, write $G(x_1)$ for the stabilizer of $x_1$ in $G$ and $G(x_1,\dots,x_n) = \bigcap^n_{i=1}G(x_i)$.  Let $\Omega$ be a set equipped with an equivalence relation $\sim$; we refer to the equivalence classes of $\sim$ as \defbold{blocks} of $\Omega$.  Given $\omega \in \Omega$, write $[\omega]$ for the $\sim$-class of $\omega$.  For $k \ge 1$, define the set $\Omega^{[k]}$ of \defbold{distant $k$-tuples} to consist of those $k$-tuples $(\omega_1,\omega_2,\dots,\omega_k)$ such that no two entries lie in the same block.  We then say $G \le \Sym(\Omega)$ is \defbold{$k$-by-block-transitive} if there are at least $k$ blocks, and $G$ acts transitively on $\Omega^{[k]}$.  Note that if $G$ is $k$-by-block-transitive, it is also $k'$-by-block-transitive for $k' < k$; in particular, $G$ is transitive on $\Omega$.  Also, by considering the orbits of a point stabilizer, it is readily seen (Lemma~\ref{lem:unique_blocks}) that if $G$ is $k$-by-block-transitive for some $k \ge 2$, then $\sim$ is the coarsest $G$-invariant equivalence relation other than the universal relation, so $\sim$ can be recovered from the action.

Given a $k$-by-block-transitive action on a set $\Omega$ with equivalence relation $\sim$, we can write $\Omega$ as $\Omega_0 \times B$ where the blocks are sets of the form $\{\omega'\} \times B$, and thus consider the action as an imprimitive extension of an action on $\Omega_0$.  Clearly, the action on $\Omega_0$ must be $k$-transitive, that is, transitive on ordered $k$-tuples of distinct elements.  Thus the most basic way to build a $k$-by-block-transitive action is to form a wreath product $H \wr_{\Omega_0} G_0$, where $G_0 \le \Sym(\Omega_0)$ is $k$-transitive and $H \le \Sym(B)$ is transitive, and let it act in the natural way on $\Omega_0 \times B$.  For the general case of a $k$-by-block-transitive $G \le \Sym(\Omega)$, then $G$ is equipped with an action $\pi_0: G \rightarrow \Sym(\Omega_0)$ on $\Omega_0$, where $G_0:= \pi_0(G)$ is $k$-transitive.  The point stabilizer $G(\omega)$ is a subgroup of the setwise stabilizer $G([\omega])$ of the block containing $\omega$, and we have $\pi_0(G(\omega)) \le \pi_0(G([\omega]))$.  The focus of the present article is to determine the possible values of $\pi_0(G(\omega))$ other than $\pi_0(G([\omega]))$ itself.

In classifying the possibilities for $\pi_0(G(\omega))$, we quickly reduce to the context of \defbold{block-faithful} $k$-by-block-transitive actions, meaning those such that $\pi_0$ is injective.  Namely, instead of considering the original action of $G$ on $\Omega$, it suffices to consider the action of $G_0$ on the cosets of $L = \pi_0(G(\omega))$.  One sees that the action of $G_0$ on $G_0/L$ is $k$-by-block-transitive, where the blocks correspond to left cosets of $G_1 = \pi_0(G([\omega]))$.  The block size of the resulting action is then the index $|G_1:L|$.

The purpose of this article is to classify the finite block-faithful $k$-by-block-transitive permutation groups $G$ for $k \ge 2$.  For all such groups $G$, it is enough to specify the group $G$, the block stabilizer $G_1 := G([\omega])$ and the point stabilizer $L:= G(\omega)$.  As the action of $G$ on $G/G_1$ is faithful and $k$-transitive, we can appeal to the known classification of finite $k$-transitive permutation groups for $k \ge 2$; then all that remains, for each possible pair $(G,G_1)$, is to classify the possible point stabilizers $L$ up to conjugacy, and we can ignore the case $G_1 = L$ as there is nothing new to say here.  One sees that in fact $G_1$ is the unique maximal subgroup of $G$ containing $L$, so we only need to specify the pair $(G,L)$.  The proofs of the results from the introduction will be given in Section~\ref{sec:proofs} at the end of the article.

First, let us note that we obtain no interesting examples for $k \ge 3$.

\begin{thm}\label{thmintro:3bbtrans}
Let $k \ge 3$ and let $G$ be a finite block-faithful $k$-by-block-transitive permutation group.  Then $G$ acts $k$-transitively, that is, the blocks are singletons.
\end{thm}

\begin{cor}\label{corintro:3bbtrans}
Let $\Omega$ be a set and let $G \le \Sym(\Omega)$ be $k$-by-block-transitive, such that $k \ge 3$ and the set $\Omega_0$ of blocks is finite.  Then for $\omega \in \Omega$ we have $G([\omega]) = KG(\omega)$ where $K$ is the kernel of the action of $G$ on $\Omega_0$.
\end{cor}

For finite block-faithful $2$-by-block-transitive actions with nontrivial blocks, the picture is more complicated, but the groups involved are still somewhat special compared to the class of all finite $2$-transitive permutation groups.  If $G$ has such an action, we find that its socle is one of
\[
\PSL_n(q), \; \mathrm{PSU}_3(q), \; {^2\mathrm{B}}_2(q), \; {^2\mathrm{G}}_2(q), \; \mathrm{M}_{11}.
\]
Let $\mu$ be a generator of $\Fb^*_q$.  Given $g \in \PGL_{n+1}(q)$, we write $\mathrm{Pdet}(g)$ for the set of determinants of matrices representing $g$: this is a coset of $\det(\Z(\GL_{n+1}(q))) = \grp{\mu^{n+1}}$, so we regard $\mathrm{Pdet}(g)$ as an element of the group $\Fb^*_q/\grp{\mu^{n+1}}$.  Given $G \le \PGaL_{n+1}(q)$, define
\[
\mathrm{Pdet}(G) := \{\mathrm{Pdet}(g) \mid g \in G \cap \PGL_{n+1}(q))\}.
\]

\begin{thm}\label{thmintro:2bbtrans}
Let $G$ be a finite group with a faithful $2$-transitive action on the set $\Omega_0$, extending to a $2$-by-block-transitive action of $G$ on the set $\Omega = \Omega_0 \times B$, with block size $|B| \ge 2$; let $\omega \in \Omega$.  Then $G$ has a nonabelian simple socle $S$ and the stabilizer $G([\omega])$ of the block $[\omega]$ containing $\omega$ is the unique maximal subgroup of $G$ that contains $G(\omega)$.  If $S$ is of Lie type and naturally represented as a group of (projective) matrices of dimension $n+1$ over $\Fb_q$, then we can regard $G$ as a subgroup of $\PGaL_{n+1}(q)$, and for $H \le G$ we write $e_H:=|H:H\cap \PGL_{n+1}(q)|$.  If $S = \PSL_{n+1}(q)$ we take $G([\omega])$ to be a point stabilizer of the usual action of $G$ on the projective $n$-space $P_n(q)$, and write $W$ for the socle of $G([\omega])$.  Up to isomorphism of permutation groups, exactly one of the following is satisfied.
\begin{enumerate}[(a)]
\item $\PSL_{n+1}(q) \le G \le \PGaL_{n+1}(q)$, with $n \ge 2$, $q > 2$.  In this case $G(\omega)$ contains the soluble residual of $G([\omega])$; if $(n,q)=(2,3)$, then $G(\omega)$ is of the form $W \rtimes \SL_2(3)$.  In addition, $e_{G(\omega)} = e_G$, and the block size $|B|$ divides $q-1$ and is coprime to $|\mathrm{Pdet}(G)|$. 
\item $\PSL_3(q) \le G \le \PGaL_3(q)$ and $G(\omega)$ is contained in a group of the form
\[
L^{\GaL_1} = G \cap (W \rtimes \GaL_1(q^2)), \text{ such that } |L^{\GaL_1}:G(\omega)| \le 2;
\]
 in this case there are up to three possibilities for the group $G(\omega)$, which belong to different $G$-conjugacy classes, and $|B| = |L^{\GaL_1}:G(\omega)|q(q-1)/2$.
\item $S$ is of rank $1$ simple Lie type and the action of $G$ on $\Omega_0$ is the standard $2$-transitive action.  Moreover, $G(\omega) = N\grp{z}$, where $N$ is normal in $G([\omega])$ and contains $G([\omega]) \cap S$, and where $G([\omega])/N$ takes the form $\grp{xN} \rtimes \grp{zN}$ such that $|\grp{xN}| = |\grp{zN}|= |B|$.  In this case, $|B|$ divides $e_G$ and also divides the order of the multiplicative group of the field.
\item The action is one of eighteen exceptional $2$-by-block-transitive actions, listed in Table~\ref{table:exceptional} below.
\end{enumerate}
The permutation groups described in (a) and (d) are $2$-by-block-transitive.  In case (b) there are additional constraints on the possible groups $G(\omega)$: see Proposition~\ref{prop:psl:plane-field}. In case (c), the possible block sizes, and the number of conjugacy classes of stabilizers of $2$-by-block-transitive actions of the given block size, are both calculated by modular arithmetic: see Proposition~\ref{prop:LDC:rank_one}.
\end{thm}

For all but the first of the exceptional $2$-by-block-transitive actions in Table~\ref{table:exceptional} below, we have 
\[
\PSL_{n+1}(q) \le G \le \PGaL_{n+1}(q)
\]
and $G([\omega])$ is a point stabilizer of the standard action of $G$ on the projective $n$-space $P_n(q)$; we again write $W$ for the socle of $G([\omega])$.  Blank spaces in the table indicate a repeated entry from the line above.  The third column indicates the structure of the stabilizer of a distant pair.  Note that in many cases $\PSL_3(q) = \PGL_3(q)$ (namely when $q - 1$ is not a multiple of $3$) or $\PGL_3(q) = \PGaL_3(q)$ (whenever $q$ is prime).

\begin{table}[h!]
\begin{center}

\begin{tabular}{ c | c | c | c | c }

$G$ & $G(\omega)$ & $G(\omega,\omega')$ & $|\Omega_0|$ & $|B|$  \\ \hline
$\mathrm{M}_{11}$ & $\Alt(6)$ & $C^2_3 \rtimes C_2$ & $11$ & $2$ \\ \hline 
$\PSL_5(2)$ & $W \rtimes \Alt(7)$ & $\PSL_3(2)$ & $31$ & $8$  \\ \hline
$\PSL_3(5)$ & $W \rtimes (\SL_2(3) \rtimes C_4)$ & $C^2_4$ & $31$ & $5$  \\
 & $W \rtimes (\SL_2(3) \rtimes C_2)$ &  $C^2_2$ & & $10$ \\
 & $W \rtimes \SL_2(3)$ & $\triv$ & & $20$   \\ \hline
$\PSL_3(7)$ & $W \rtimes (\SL_2(3).C_2)$ & $C_3$ & $57$ & $14$ \\
$\PGaL_3(7)$ & $W \rtimes (\SL_2(3).C_2 \times C_3)$ & $C^2_3$ & & \\ \hline
$\PSL_3(9)$ & $W \rtimes (\SL_2(5).C_4)$ & $\Sym(3)^2$ & $91$ & $12$   \\
$\PGaL_3(9)$ & $W \rtimes (\SL_2(5).D_8)$ & $\Sym(3)^2 \times C_2$ & &  \\ \hline
$\PSL_3(11)$ & $W \rtimes (\SL_2(5) \times C_5)$ & $C^2_5$ & $133$ & $22$   \\
 & $W \rtimes \SL_2(5)$ & $\triv$ & & $110$  \\ 
 & $W \rtimes (\GL_2(3) \times C_5)$ & $C^2_2$ & & $55$   \\
 & $W \rtimes (\SL_2(3) \times C_5)$ & $\triv$ & & $110$   \\ \hline
$\PGaL_3(19)$ & $W \rtimes (\SL_2(5) \times C_9)$ & $C^2_3$ & $381$ & $114$  \\ \hline
$\PSL_3(23)$ & $W \rtimes (\SL_2(3).C_2 \times C_{11})$ & $\triv$ & $553$ & $506$  \\ \hline
$\PSL_3(29)$ & $W \rtimes ((\SL_2(5) \rtimes C_2) \times C_7)$ & $C^2_2$ & $871$ & $406$   \\
 & $W \rtimes (\SL_2(5) \times C_7)$ & $\triv$ & & $812$   \\ \hline
$\PSL_3(59)$ & $W \rtimes (\SL_2(5) \times C_{29})$ & $\triv$ & $3541$ & $3422$ \\ \hline
\end{tabular}
\captionof{table}{Exceptional $2$-by-block-transitive actions\label{table:exceptional}}
\end{center}
\end{table}

The following emerges as an observation on the classification.

\begin{cor}\label{corintro:2bbtrans:socle}
Let $G$ be a finite block-faithful $2$-by-block-transitive permutation group.  Then the socle of a block stabilizer acts trivially on that block.
\end{cor}

We can also classify the finite sharply $2$-by-block-transitive permutation groups, that is, actions preserving an equivalence relation on a finite set such that for any two distant pairs, there is exactly one element mapping the first distant pair to the second distant pair.  With six exceptions, these are sharply $2$-transitive or arise from case (b) of Theorem~\ref{thmintro:2bbtrans}.

\begin{thm}\label{thmintro:2bbtrans:sharp}
Let $G$ be a group acting on the finite set $\Omega$ and acting faithfully on $\Omega^{[2]}$.  Then $G$ acts regularly on $\Omega^{[2]}$ if and only if one of the following holds.
\begin{enumerate}[(a)]
\item $G$ is sharply $2$-transitive, in other words, the blocks are singletons.
\item We are in case (b) of Theorem~\ref{thmintro:2bbtrans} with $|L^{\GaL_1}:G(\omega)| = 2$ and $|G| = |\PGL_3(q)|$, and the block size is $q(q-1)$.
\item We have one of the six exceptional actions from Table~\ref{table:exceptional} for which the stabilizer of a distant pair is trivial; in particular, $G = \PSL_3(q) = \PGaL_3(q)$ with $q \in \{5,11,23,29,59\}$, and the block size is $q(q-1)$.
\end{enumerate}
\end{thm}

We conclude this introduction with some remarks on the main theorems and their context in the literature.

\begin{rem} \
\begin{enumerate}[(1)]
\item
The author's original motivation for classifying finite block-faithful $2$-by-block-transitive actions is an application to groups acting on infinite locally finite trees.  Specifically, if $T$ is a locally finite tree and the closed subgroup $G \le \Aut(T)$ acts $2$-transitively on the space of ends $\partial T$ of $T$, it is not hard to see that each vertex stabilizer $G(v)$ acts $2$-by-block-transitively on $\partial T$, where the blocks correspond to the neighbours of $v$ in $T$; since there are finitely many blocks, one obtains a quotient $2$-by-block-transitive action on a finite set.  This application to groups acting on trees is developed in a separate article (\cite{Reid}).

\item
By taking $G$ up to isomorphism as a permutation group, in the cases that $G$ has socle $\PSL_{n+1}(q)$ for $n \ge 2$, we are effectively ignoring the distinction between the set $P_n(q)$ of lines in $\Fb^{n+1}_q$ and the set $P^*_n(q)$ of $n$-dimensional subspaces of $\Fb^{n+1}_q$, which yield isomorphic permutation groups (via the inverse transpose automorphism of $G$) but not equivalent $G$-sets (in other words, the corresponding point stabilizers are not conjugate in $G$).

\item
A previous classification theorem that served as an inspiration for Theorem~\ref{thmintro:2bbtrans} is the classification by Devillers, Giudici, Li, Pearce and Praeger (\cite[Theorem~1.2]{Rank3}) of the finite quasiprimitive imprimitive rank $3$ permutation groups (where the \defbold{rank} is the number of orbits of a point stabilizer).  Any such group is easily seen to be $2$-by-block-transitive and block-faithful; the three orbits of $G(\omega)$ are $\{\omega\}$, $[\omega] \setminus \{\omega\}$ and $\Omega \setminus [\omega]$, so naturally, the classification in \cite{Rank3} describes a special case of the actions given in Theorem~\ref{thmintro:2bbtrans}, namely the case when $G([\omega])$ acts $2$-transitively on $G([\omega])/G(\omega)$ (which is automatically the case if $|B|=2$, but not otherwise; indeed, $G(\omega)$ need not even be maximal in $G([\omega])$).  Specifically, the groups in \cite[Table 1]{Rank3} relate to Theorem~\ref{thmintro:2bbtrans} as follows:
\begin{enumerate}[(a)]
\item Case (a) of Theorem~\ref{thmintro:2bbtrans} includes line 3 of \cite[Table 1]{Rank3}, but extra conditions are required to ensure the action has rank $3$.
\item Most actions in case (b) of Theorem~\ref{thmintro:2bbtrans} do not have rank $3$, but there are five exceptions, all with $G(\omega) = L^{\GaL_1}$.  Here $G$ is one of $\PGL_3(4), \PGaL_3(4), \PGaL_3(8), \PSL_3(2),\PSL_3(3)$, listed in lines 4, 5, 8, 9, 10 of \cite[Table 1]{Rank3}.
\item Case (c) of Theorem~\ref{thmintro:2bbtrans} includes line 2 of \cite[Table 1]{Rank3}, but a rank $3$ action only arises if $|B|=2$, which is limited to actions with socle $\PSL_2(q)$.  By contrast, the other types of socle give examples of $2$-by-block-transitive actions with odd block size $n$ and rank $n+1$, see Example~\ref{ex:ldc} below.
\item The first three lines of Table~\ref{table:exceptional} have rank $3$, and are listed in lines 1, 7 and 6 of \cite[Table 1]{Rank3}.
\end{enumerate}

\item
The first line of Table~\ref{table:exceptional} is a notable `near miss' for an example of a finite block-faithful $3$-by-block-transitive action with nontrivial blocks: as well as being $2$-by-block-transitive and $4$-transitive on blocks, it has only two orbits on distant triples (see Lemma~\ref{lem:no_ldc:Mathieu}).
\end{enumerate}
\end{rem}

\paragraph{Structure of article}

The remainder of the article is divided into two sections.  In Section 2 we show some basic properties of $k$-by-block-transitive actions and recall the necessary information about the classification of finite $2$-transitive permutation groups.  The main section is Section 3, where we work through the classification of finite block-faithful $2$-by-block-transitive permutation groups on a case-by-case basis.

\paragraph{Acknowledgement} I thank Tom De Medts, Michael Giudici, Bernhard M\"{u}hlherr and Hendrik Van Maldeghem for helpful comments related to this article.  I also thank the anonymous referee who suggested a number of useful references and improvements.

\section{Preliminaries}

\subsection{Generalities on $2$-by-block-transitive groups}

Here we note some general properties of $2$-by-block-transitive groups.

The first thing to note is the following double coset formula for $2$-by-block-transitive action.  Throughout, we write $[\omega]$ for the $\sim$-class of $\omega \in \Omega$.

\begin{lem}\label{lem:double_coset}
Let $\Omega$ be a set, let $G \le \Sym(\Omega)$ be a transitive group preserving the nonuniversal equivalence relation $\sim$ and let $\omega \in \Omega$.  Then $G$ is $2$-by-block-transitive if and only if the following equation is satisfied for some (or equivalently for all) $g \in G \setminus G([\omega])$:
\[
G = G(\omega)gG(\omega) \sqcup G([\omega]).
\]
\end{lem}

\begin{proof}
Since $G$ acts transitively on $\Omega$, there is a bijection $\theta$ from the set of $(G(\omega),G(\omega))$-double cosets in $G$ to the set of $G(\omega)$-orbits on $\Omega$, where given $x \in G$ we set 
\[
\theta(G(\omega)xG(\omega)) = \{h(x\omega) \mid h \in G(\omega)\}.
\]
Note that $G([\omega])$ is itself a union of $(G(\omega),G(\omega))$-double cosets, and the $G([\omega])$-orbit of $\omega$ is exactly $[\omega]$; thus $\bigcup\{\theta(G(\omega)xG(\omega)) \mid x \in G([\omega])\}$ is the union of all $G(\omega)$-orbits on $[\omega]$.

In particular, the double coset equation in the statement is equivalent to the assertion that $G(\omega)$ acts transitively on $\Omega \setminus [\omega]$.  In turn, we see that $G(\omega)$ acts transitively on $\Omega \setminus [\omega]$ if and only if $G$ acts transitively on distant pairs.
\end{proof}

If $G$ is finite, we obtain a formula for the order of the stabilizer of a distant pair.

\begin{cor}\label{cor:double_coset}
Let $\Omega$ be a set and let $G$ be a finite group acting transitively on $\Omega$ and preserving the equivalence relation $\sim$.  Let $(\omega,\omega') \in \Omega^{[2]}$.  Then $G$ acts $2$-by-block-transitively on $\Omega$ if and only if
\[
|G(\omega,\omega')| = \frac{|G(\omega)|^2}{|G| - |G([\omega])|}.
\]
In particular, if $G$ acts $2$-by-block-transitively on $\Omega$ then the right-hand side of the above equation is an integer.
\end{cor}

\begin{proof}
Since $G$ is transitive we can write $\omega' = g\omega$ for some $g \in G$; note that $g \not\in G([\omega])$.  We can calculate the size of the double coset $G(\omega)gG(\omega)$ as
\[
|G(\omega)gG(\omega)| = |G(\omega)gG(\omega)g\inv| =  |G(\omega)G(\omega')| =  \frac{|G(\omega)||G(\omega')|}{|G(\omega) \cap G(\omega')|} =  \frac{|G(\omega)|^2}{|G(\omega,\omega')|},
\]
so the equation in the statement is equivalent to the equation
\[
|G(\omega)gG(\omega)| = |G| - |G([\omega])|.
\]
By Lemma~\ref{lem:double_coset}, the latter equation is equivalent to $G$ acting $2$-by-block-transitively on $\Omega$.
\end{proof}

Using the standard correspondence between systems of imprimitivity of a transitive permutation group and subgroups containing a point stabilizer, we deduce the following.

\begin{lem}\label{lem:unique_blocks}
Let $\Omega$ be a set and let $G \le \Sym(\Omega)$ be $k$-by-block-transitive on $\Omega$ relative to the equivalence relation $\sim$, for $k \ge 2$.  Then for each $\omega \in \Omega$, the block stabilizer $G([\omega])$ is the largest proper subgroup of $G$ containing $G(\omega)$.  Equivalently, $\sim$ is the coarsest nonuniversal $G$-invariant equivalence relation.  In particular, we can recover $\sim$ from the action of $G$ on $\Omega$.
\end{lem}

\begin{proof}
Since $G$ is $k$-by-block-transitive, certainly $G$ is transitive and there are $k \ge 2$ blocks, so $\sim$ cannot be the universal relation.  We see that $\sim$ is $G$-invariant because if it were not, say $x \sim y$ but $gx \not\sim gy$, then we would have a distant $k$-tuple $(gx,gy,x_3,\dots,x_k)$ in the same $G$-orbit as the nondistant $k$-tuple $(x,y,g\inv x_3,\dots,g\inv x_k)$.  In particular, we have $G(\omega) \le G([\omega]) < G$.

Suppose now that $H$ is some subgroup of $G$ containing $G(\omega)$, and suppose that $H \nleq G([\omega])$, say $g \in H \setminus G([\omega])$; let $O$ be the $H$-orbit of $g\omega$.  Then $g\omega \in \Omega \setminus [\omega]$; since $H \ge G(\omega)$ and $G(\omega)$ acts transitively on $\Omega \setminus [\omega]$ it follows that $\Omega \setminus [\omega] \subseteq O$, in particular $[g\omega] \subseteq O$.  But now also $g\inv[g\omega] = [\omega] \subseteq O$, so in fact $O = \Omega$, that is, $H$ is transitive on $\Omega$.  Since $G(\omega) \le H$, we conclude that $H = G$.

From the previous paragraph, we deduce that $\sim$ is the coarsest $G$-invariant equivalence relation other than the universal relation.
\end{proof}

Here is a special case, which rules out an affine type action on the set of blocks as soon as the blocks are nontrivial.

\begin{cor}\label{cor:LDC_normal}
Let $\Omega$ be a set, let $G \le \Sym(\Omega)$ be block-faithful and $k$-by-block-transitive on $\Omega$ relative to the nontrivial equivalence relation $\sim$, for $k \ge 2$.  Let $N$ be a nontrivial normal subgroup of $G$.  Then $G = NG(\omega)$.  In particular, $N$ is not abelian.
\end{cor}

\begin{proof}
Since $G$ acts $k$-transitively on the set of blocks, we see that $N$ acts transitively on the set of blocks, so $N \nleq G([\omega])$.  It then follows by Lemma~\ref{lem:unique_blocks} that $G = NG(\omega)$.  Since $\sim$ is nontrivial we have $G(\omega) < G([\omega])$, so $N \cap G([\omega])$ is nontrivial.  Since $N$ acts faithfully on the set of blocks we deduce that $N$ is not abelian.
\end{proof}

We observe the following property of $2$-by-block-transitive actions, which will be used without further comment.

\begin{cor}\label{cor:LDC_2pt:maximal}
Let $G$ be a group and let $L < K < G$ be proper subgroups.  If the action of $G$ on $G/L$ is $2$-by-block-transitive with block stabilizer $M \ge L$, then $M \ge K$ and the action on $G/K$ is $2$-by-block-transitive with block stabilizer $M$.
\end{cor}

\begin{proof}
Suppose $G$ has $2$-by-block-transitive action on $\Omega = G/L$ and write $\omega$ for the trivial coset.  Then $M = G([\omega])$; by Lemma~\ref{lem:unique_blocks}, since $L < K < G$ we have $K \le M$.  Let $g \in G \setminus M$.  By Lemma~\ref{lem:double_coset} we have $G = LgL \cup M$; it then follows that $G = KgK \cup M$, so the action of $G$ on $G/K$ is $2$-by-block-transitive by Lemma~\ref{lem:double_coset}.
\end{proof}

Given a group $G$, we say the $G$-set $\Omega$ is an \defbold{extension} of the $G$-set $\Omega_0$ if there is a $G$-equivariant surjection $\pi$ from $\Omega$ to $\Omega_0$.  The fibres of $\pi$ then form a system of imprimitivity for $G$ acting on $\Omega$.

Given a $k$-by-block-transitive permutation group $G \le \Sym(\Omega)$, we have an associated $k$-transitive action (possibly also $k'$-transitive for some $k' > k$) $\pi:G \rightarrow \Sym(\Omega/\sim)$.  We say $G$ is \defbold{block-faithful} if $G$ acts faithfully on $\Omega/\sim$.  Thus every block-faithful $k$-by-block-transitive action is an extension of a faithful $k$-transitive action of the same group.

The next lemma gives a necessary and sufficient condition for a block-faithful extension of a $2$-transitive action to be $2$-by-block-transitive.

\begin{lem}\label{lem:LDC_2pt}
Let $\Omega$ be a set and let $G$ be a group acting on $\Omega$ that preserves the nonuniversal equivalence relation $\sim$, acts faithfully and $2$-transitively on $\Omega/\sim$ and acts transitively on $\Omega$.
\begin{enumerate}[(i)]
\item Given $\omega,\omega' \in \Omega$ with $\omega \not\sim \omega'$, then there is an involution $s \in G$ such that $s[\omega] = [\omega']$, and such that
\[
G(\{[\omega],[\omega']\}) = G([\omega],[\omega']) \rtimes \grp{s}.
\]
\item Let $\omega \in \Omega$, let $L = G(\omega)$ and let $s \in G$ be an involution such that $s\omega \not\sim \omega$. Then $G$ has $2$-by-block-transitive action on $\Omega$ if and only if 
\begin{equation}\label{eq:lem:LDC_2pt_a}
G([\omega]) = LsL([s\omega])s.
\end{equation}
In particular, if $G$ has $2$-by-block-transitive action on $\Omega$ then
\begin{equation}\label{eq:lem:LDC_2pt_b}
G([\omega],[s\omega]) = L([s\omega])sL([s\omega])s;
\end{equation}
if in addition $G$ is finite, then
\begin{equation}\label{eq:lem:LDC_2pt_c}
|G([\omega],[s\omega]):L([s\omega])| = |G([\omega]):L|.
\end{equation}
\end{enumerate}
\end{lem}

\begin{proof}
(i)
Since $G$ acts $2$-transitively on $\Omega/\sim$, there is an element swapping two blocks, so $G$ has an element of even order, and hence an element $s$ of order $2$.  Since $G$ acts faithfully on $\Omega/\sim$, there is some block not fixed by $s$, say $s[\omega_0] = s[\omega'_0]$ where $\omega_0,\omega'_0 \in \Omega$ such that $\omega_0 \not\sim \omega'_0$.  Since $G$ is $2$-transitive on $\Omega/\sim$, after conjugating $s$ by a suitable element of $G$, we can in fact take $[\omega_0]=[\omega]$ and $[\omega'_0] = [\omega']$.  The form taken by $G(\{[\omega],[\omega']\})$ is now clear.

(ii)
We have $G([\omega]) = LsL([s\omega])s$ if and only if $L$ acts transitively on the coset space $X = G([\omega])/sL([s\omega])s$.  As a $G([\omega])$-space, $X$ is equivalent to the $G([\omega])$-orbit $Y$ of $s\omega$, since 
\[
sL([s\omega])s = sLs \cap G([\omega]).
\]
Since $G$ acts $2$-transitively on blocks, we see that $Y$ intersects every block other than $[\omega]$.

If $G$ is $2$-by-block-transitive, then $Y = \Omega \setminus [\omega]$ and $L$ acts transitively on $\Omega \setminus [\omega]$, so we deduce (\ref{eq:lem:LDC_2pt_a}) from the previous paragraph, and then (\ref{eq:lem:LDC_2pt_b}) follows from (\ref{eq:lem:LDC_2pt_a}) by intersecting both sides with the group $G([s\omega])$.  If $G$ is finite then (\ref{eq:lem:LDC_2pt_c}) follows from (\ref{eq:lem:LDC_2pt_b}) and the fact that $L$ acts transitively on $(\Omega/\sim) \setminus \{[\omega]\}$, so that $|L:L([s\omega])| = |G([\omega]):G([\omega],[s\omega])|$.

Conversely, assume (\ref{eq:lem:LDC_2pt_a}).  Then $L$ acts transitively on $Y$, so $L$ acts transitively on $(\Omega/\sim) \setminus \{[\omega]\}$.  We can write the last statement as a double coset equation, using the fact that $[s\omega] \neq [\omega]$:
\[
G = LsG([\omega]) \sqcup G([\omega]);
\]
equivalently,
\[
G = G([\omega])sL \sqcup G([\omega]),
\]
which means that $G([\omega])$ has only two orbits on $\Omega$, one of which is $\Omega \setminus [\omega]$.  Thus by (\ref{eq:lem:LDC_2pt_a}), in fact $L$ acts transitively on $\Omega \setminus [\omega]$; hence $G$ acts transitively on distant pairs.
\end{proof}

Suppose now we have a $2$-transitive action of $G$ on the set $\Omega$, with $2$-by-block transitive action on $G/L$ for $L \le G(\omega)$.  In classifying the $2$-by-block-transitive extensions of the action on $\Omega$, it is sometimes convenient to count candidates for $L$ up to $G(\omega)$-conjugacy, rather than up to $G$-conjugacy.  The next lemma shows that these two ways of counting extensions are equivalent, if we exclude the degenerate case when $|\Omega| \le 2$.

\begin{lem}
Let $\Omega$ be a set with at least $3$ elements, let $G$ be a group acting $2$-transitively on $\Omega$, and let $L \le G(\omega)$ be such that $G$ acts $2$-by-block-transitively on $G/L$.  Then for all $g \in G$ we have $gLg\inv \le G(\omega)$ if and only if $g \in G(\omega)$.
\end{lem}

\begin{proof}
If $g \in G(\omega)$, then clearly $gLg\inv \le G(\omega)$.

Conversely, suppose $g \in G$ is such that $gLg\inv \le G(\omega)$.  Then by Lemma~\ref{lem:unique_blocks}, $gG(\omega)g\inv$ is the largest proper subgroup of $G$ that contains $gLg\inv$, so $G(\omega) \le gG(\omega)g\inv$.  Since $G$ acts $2$-transitively on $\Omega$, the subgroup $G(\omega)$ is already maximal among proper subgroups, so we must have $G(\omega) = gG(\omega)g\inv$, that is, $g \in \N_G(G(\omega))$.  There are two possibilities for the normalizer: either $\N_G(G(\omega)) = G(\omega)$ or $\N_G(G(\omega)) = G$.  In the former case we are done.  In the latter case, $G(\omega)$ fixes $\Omega$ pointwise, which is incompatible with $2$-transitivity given the assumption that $|\Omega| \ge 3$.
\end{proof}

\begin{rem}
Based on Lemma~\ref{lem:double_coset} one can classify the $2$-by-block-transitive actions of a given finite group $G$ by the following simple (but not particularly efficient) procedure:
\begin{enumerate}[(1)]
\item Determine a set $\mc{M}$ of representatives of the conjugacy classes of subgroups $H$ of $G$ such that $G$ acts $2$-transitively on $G/H$.
\item For each $H \in \mc{M}$, fix $g_H \in G \setminus H$.
\item We produce a sequence $\mc{L}^H_i$ of sets of subgroups of $H$, starting with $\mc{L}^H_0 = \{H\}$.  For each $L \in \mc{L}^H_i$ we take representatives of the $H$-conjugacy classes occurring as maximal subgroups $M$ of $L$; then for each representative $M$, we check if $G = Mg_HM \cup H$ (for instance by calculating whether $|Mg_HM|+|H|=|G|$).  We then form a set $\mc{L}^H_{i+1}$ consisting of those $M$ in the previous sentence such that $G = Mg_HM \cup H$.
\item The previous step eventually terminates; write $\mc{L}^H = \bigcup_{i \ge 0}\mc{L}^H_i$.
\item Write $\mc{L} = \bigcup_{H \in \mc{M}}\mc{L}^H$.  Up to equivalence, the $2$-by-block-transitive actions of $G$ are given by the action on $G/L$ for $L \in \mc{L}$.
\end{enumerate}
\end{rem}

\subsection{Finite $2$-transitive permutation groups of almost simple type}

The finite $2$-transitive permutation groups are all known; an overview can be found for example in \cite[\S7.7]{DixonMortimer}.  Every such group is either of affine type or has nonabelian simple socle, and given Corollary~\ref{cor:LDC_normal} we can focus on the latter.  These groups are displayed in the next table, with the rows corresponding to the action of the socle $S$, where two actions are identified if their stabilizers belong to the same $\Aut(S)$-class.  We also indicate the degree $t$ of transitivity, the largest overgroup $N := \N_{\Sym(\Omega)}(S)$ of the socle compatible with the action and the structure of a point stabilizer of $N$; where applicable, $P$ denotes a normal subgroup of $S$ that acts regularly on $\Omega \setminus \{\omega\}$.  In all cases the socle itself acts $2$-transitively, except for the action of $\PGaL_2(8) = {^2 \mathrm{G}_2(3)}$ on $28$ points.

\begin{minipage}{1.0\textwidth}
\begin{small}
\begin{center}

\begin{tabular}{ c | c | c | c | c}

Degree & $t$ & $S$ & $N$ & $N(\omega)$ \\ \hline
$d \ge 5$ & $d-2$ or $d$ & $\Alt(d)$ & $\Sym(d)$ & $\Sym(d-1)$ \\ 
\footnote{$(n,q) \not\in \{(1,2),(1,3),(1,4)\}$}$\frac{q^{n+1}-1}{q-1}$ & $2$ or $3$ & $\PSL_{n+1}(q)$ & $\PGaL_{n+1}(q)$ & $(C^{en}_p \rtimes \GL_n(q)) \rtimes \grp{\phi}$ \\
$q^3+1,q \ge 3$ & $2$ & $\mathrm{PSU}_3(q)$ & $\mathrm{P}\Gamma\mathrm{U}_3(q)$ & $P \rtimes \grp{x,\phi}$ \\ 
$q^2+1,q = 2^{2n+1}>2$ & $2$ & $^2 \mathrm{B}_2(q)$ & $^2 \mathrm{B}_2(q) \rtimes \grp{\phi}$ & $P \rtimes \grp{x,\phi}$ \\ 
$q^3+1,q = 3^{2n+1}>3$ & $2$ & $^2 \mathrm{G}_2(q)$ & $^2 \mathrm{G}_2(q) \rtimes \grp{\phi}$ & $P \rtimes \grp{x,\phi}$  \\
$2^{2n+1} -2^{n}$ & $2$ & $\mathrm{Sp}_{2n+2}(2)$ & $\mathrm{Sp}_{2n+2}(2)$ & $\Omega^-_{2n+2}(2).C_2$ \\
$2^{2n+1} +2^{n}$ & $2$ & $\mathrm{Sp}_{2n+2}(2)$ & $\mathrm{Sp}_{2n+2}(2)$ & $\Omega^+_{2n+2}(2).C_2$ \\ \hline
$11$ & $2$ & $\PSL_2(11)$ & $\PSL_2(11)$ & $\Alt(5)$ \\
$11$ & $4$ & $\mathrm{M}_{11}$ & $\mathrm{M}_{11}$ & $\Alt(6).C_2$ \\ 
$12$ & $3$ & $\mathrm{M}_{11}$ & $\mathrm{M}_{11}$ & $\PSL_2(11)$ \\ 
$12$ & $5$ & $\mathrm{M}_{12}$ & $\mathrm{M}_{12}$ & $\mathrm{M}_{11}$ \\ 
$15$ & $2$ & $\Alt(7)$ & $\Alt(7)$ & $\PSL_3(2)$ \\ 
$22$ & $3$ & $\mathrm{M}_{22}$ & $\mathrm{M}_{22} \rtimes C_2$ & $\PSL_3(4) \rtimes C_2$ \\ 
$23$ & $4$ & $\mathrm{M}_{23}$ & $\mathrm{M}_{23}$ & $\mathrm{M}_{22}$ \\ 
$24$ & $5$ & $\mathrm{M}_{24}$ & $\mathrm{M}_{24}$ & $\mathrm{M}_{23}$ \\
$28$ & $2$ & $\PSL_2(8)$ & $\PGaL_2(8)$ & $C_9 \rtimes C_6$ \\ 
$176$ & $2$ & $\mathrm{HS}$ & $\mathrm{HS}$ & $\mathrm{PSU}_3(5)  \rtimes C_2$ \\ 
$276$ & $2$ & $\mathrm{Co}_3$ & $\mathrm{Co}_3$ & $\mathrm{McL}  \rtimes C_2$\\
\end{tabular}
\end{center}
\end{small}
\end{minipage}

\subsection{Transitive semilinear groups}

For the classification of finite block-faithful $2$-by-block-transitive actions, we will also need some aspects of the classification of finite $2$-transitive affine groups, which are summarized in the following lemma.

\begin{lem}[{Hering \cite{Hering}; see also \cite[Appendix 1]{LiebeckAffine}}]\label{lem:affine_2trans}
Let $H = V \rtimes G$ be a finite $2$-transitive affine permutation group, where $V$ is regarded as the additive group of some vector space on which $G$ acts by semilinear maps.  Then $G$ and $V$ can be taken as follows:
\begin{enumerate}[(i)]
\item $V$ is the field of order $q$, and $G \le \GaL_1(q)$;
\item $V$ is the vector space of dimension $n$ over the field of order $q$, and $\SL_n(q) \unlhd G$;
\item $V$ is the vector space of dimension $2m$ over the field of order $q$, and $\mathrm{Sp}_{2m}(q) \unlhd G$;
\item $V$ is the vector space of dimension $6$ over the field of order $q = 2^e$, and $\mathrm{G}_2(q) \unlhd G$;
\item $H$ is one of finitely many exceptional $2$-transitive groups of affine type, with $V$ being a vector space of dimension $2$, $4$ or $6$.  In each case, either $H$ is soluble or $H$ has a unique nonabelian composition factor, which is one of: $\Alt(5),\Alt(6),\Alt(7),\PSL_2(13)$.  If $V$ has dimension $2$, then $q \in \{5,7,9,11,19,23,29,59\}$.
\end{enumerate}
\end{lem}

\section{$2$-by-block-transitive groups}\label{sec:LDC}

In this section we will obtain a classification of the block-faithful $k$-by-block-transitive actions of finite groups, as promised in the introduction.  For the most part, the proof will be via case analysis of the finite $2$-transitive permutation groups.

Let us set some notation for linear and projective spaces that will be used throughout this section.

\begin{defn}
Let $p$ be a prime and $q = p^e$ for some positive integer $e$.  We give the vector space $V= \Fb^{n+1}_q$ a standard basis $\{v_0,\dots,v_n\}$.  Given a subset $X$ of $V$, write $\grp{X}_q$ for the $\Fb_q$-subspace of $V$ generated by $X$ and write $\alpha_i = \grp{v_i}_q$.  We write $P_n(q)$ for the set of lines in $\Fb^{n+1}_q$.  Write $\GaL_{n+1}(q)$ for the group of semilinear maps from $\Fb^{n+1}_q$ to itself.  In this context there is an element $\phi \in \GaL_{n+1}(q)$ of order $e$ that acts by sending $\sum^n_{i=0}\lambda_i v_i$ to $\sum^n_{i=0}\lambda^p_i v_i$; we will refer to $\GaL_{n+1}(q)$-conjugates of powers of $\phi$ as \defbold{field automorphisms}.  One can also consider $\phi$ to act as an automorphism of $\GL_{n+1}(q)$ that transforms matrices by sending every entry to its $p$-th power; however, it is important to note that the action of $\GL_{n+1}(q)$ on $V$ also naturally extends to an action of $\GaL_{n+1}(q) = \GL_{n+1}(q) \rtimes \grp{\phi}$.  (By contrast, for $n \ge 2$ the natural action of $\GL_{n+1}(q)$ on $V$ cannot be extended to incorporate the inverse transpose automorphism of $\GL_{n+1}(q)$.)

There is an induced action of $\GaL_{n+1}(q)$ on $P_n(q)$, with kernel the scalar matrices; we write $\PGaL_{n+1}(q)$ for $\GaL_{n+1}(q)$ modulo the scalar matrices.  Note that $\PGaL_{n+1}(q)$ acts (at least) $2$-transitively on $P_n(q)$.
\end{defn}

In this section, some calculations on individual finite groups, namely to determine subgroups of given indices and enumerate double cosets (the latter to determine if an action is $2$-by-block-transitive) were performed using the computer algebra package GAP \cite{GAP}.  We omit the details of these routine computations.  The author also used the online ATLAS of Finite Group Representations \cite{ATLAS} as an indicative reference for some properties of finite groups, however it is not required for the proofs.

\subsection{Some specific $2$-transitive groups}

We start with some groups that are convenient to deal with individually.

\begin{lem}\label{lem:no_ldc:specific}
Let $G$ be one of 
\[
\PSL_2(11), \Alt(7), \PGaL_2(8),\mathrm{HS},\mathrm{Co}_3
\]
acting $2$-transitively on a set $X$ of $d$ points, where $d=11, 15, 28, 176, 276$ respectively.  Then $X$ does not extend properly to a $2$-by-block-transitive action of $G$.
\end{lem}

\begin{proof}
Let $x$ and $y$ be distinct elements of $X$ and let $s$ be an involution of $G$ such that $sx=y$ and $sy=x$, as in Lemma~\ref{lem:LDC_2pt}(i).  Let $L \le G(x)$ be a point stabilizer of a $2$-by-block-transitive action of $G$.

For $G = \PSL_2(11)$ acting on $11$ points, we see that $|G(x,y)|=6$; for $G = \PGaL_2(8)$ acting on $28$ points then $|G(x,y)|=2$.  For $G = \mathrm{HS}$ acting on $176$ points, we have $G(x) = \mathrm{PSU}_3(5) \rtimes C_2$ and $G(x,y) = \Aut(\Alt(6))$, so every automorphism of $G(x,y)$ is inner.  In all of these cases, we see that there is no proper subgroup $H$ of $G(x,y)$ such that $G(x,y) = HsHs$.  The conclusion now follows by Lemma~\ref{lem:LDC_2pt}.

For $G = \Alt(7)$ acting on $15$ points, we have $|G| - |G(x)| =  2^4 \cdot 3 \cdot 7^2$, so by Corollary~\ref{cor:double_coset} we would need $|L|$ to be a multiple of $2^2 \cdot 3 \cdot 7$, or in other words, index at most $2$ in $G(x)$.  However, $G(x) \cong \PSL_3(2)$ is simple and therefore has no subgroup of index $2$.  Thus there is no proper extension of $X$ to a $2$-by-block-transitive action in this case.

For $G = \mathrm{Co}_3$ acting on $276$ points, we have $G(x) = \mathrm{McL} \rtimes C_2$ and $G(x,y) = \mathrm{PSU}_4(3) \rtimes C_2$.  In order to achieve $G(x,y) = L(y)sL(y)s$, as in Lemma~\ref{lem:LDC_2pt}, we would need $G(x,y) = AsAs$ for some maximal subgroup $A$ of $G(x,y)$ containing $L(y)$.  Given that $G(x,y)$ is almost simple, one sees from \cite[Theorem~1.1]{Baumeister} that no such $A$ exists.
\end{proof}

\begin{lem}\label{lem:no_ldc:Mathieu}
Among the Mathieu groups (including $\mathrm{M}_{22} \rtimes C_2$) there is only one proper $2$-by-block-transitive action, namely the imprimitive rank $3$ action of $\mathrm{M}_{11}$, which is given in line 1 of Table~\ref{table:exceptional}.  The latter action has exactly two orbits on distant triples, each of size $3960$.
\end{lem}

\begin{proof}
We suppose that $G$ is a Mathieu group equipped with one of its multiply transitive actions on the set $X$, take $x \in X$, and suppose $L < G(x)$ is such that the action on $G/L$ is $2$-by-block-transitive.  We again use the fact given by Corollary~\ref{cor:double_coset} that $|L|^2$ must be a multiple of $|G|-|G(x)|$.

First consider $(G,G(x)) = (\mathrm{M}_{11},\mathrm{M}_{10})$.  In this case $|G|-|G(x)| = 2^5 \cdot 3^2 \cdot 5^2$, so we need $|L|$ to be a multiple of $2^3 \cdot 3 \cdot 5 = 120$.  We find that the only subgroup of $G(x) = \mathrm{M}_{10}$ of sufficient order is the subgroup $\Alt(6)$ of order $2$.  In this case we indeed obtain a $2$-by-block-transitive action of $\mathrm{M}_{11}$ (indeed an imprimitive rank $3$ action, see \cite{Rank3}) on the $22$-point set $\Omega = \mathrm{M}_{11}/\Alt(6)$, where the blocks correspond to cosets of $\mathrm{M}_{10}$.  The stabilizer of a distant pair $(\omega_1,\omega_2)$ is a group $H = C^2_3 \rtimes C_2$ of order $18$.  A calculation in GAP finds that the action of $H$ on $\Omega$ has four fixed points (which must be the points in $[\omega_1] \cup [\omega_2]$) plus two orbits of size $9$.  Thus $\mathrm{M}_{11}$ has two orbits on distant triples, each of size $9 \cdot 22 \cdot 20 = 3960$.

Next suppose $|X|=12$, so $G$ is either $\mathrm{M}_{11}$ with point stabilizer $\PSL_2(11)$, or $\mathrm{M}_{12}$ in its natural action.  In either case, $|G| - |G(x)|$ is a multiple of $22$, so $|L|$ is also a multiple of $22$; however, the only proper subgroups of $\mathrm{M}_{11}$ of order a multiple of $22$ are the point stabilizers of the $3$-transitive action of $\mathrm{M}_{11}$ on $12$ points.  This rules out any proper $2$-by-block-transitive actions when $G = \mathrm{M}_{11}$, and for $G = \mathrm{M}_{12}$ one can check by double coset enumeration that $\mathrm{M}_{12}$ does not in fact act $2$-by-block-transitively on the cosets of $\PSL_2(11)$.

Finally we suppose $G$ is one of the large Mathieu groups, including $\mathrm{M}_{22} \rtimes C_2$, acting on $22 \le d \le 24$ points.  In each case $|G|-|G(x)|$ is a multiple of $2^6 \cdot 5 \cdot p$ where $p \in  \{7,11,23\}$ is the largest prime dividing $d-1$, so $|L|$ must be divisible by $2^3 \cdot 5 \cdot p$.  After checking the orders of maximal subgroups of $G(x)$, we are only left with $G = \mathrm{M}_{22} \rtimes C_2$, $G(x) = \PSL_3(4) \rtimes C_2$ and $L \le \PSL_3(4)$.  This remaining case is ruled out by Corollary~\ref{cor:LDC_normal}, since we would have $L$ contained in the proper normal subgroup $\mathrm{M}_{22}$ of $G$.\end{proof}

\begin{lem}\label{lem:pgl52}
Let $G = \PSL_5(2)$.  Then there is exactly one $\Aut(G)$-conjugacy class of subgroups $L$ of $G$ such that $G$ has proper $2$-by-block-transitive action on $G/L$, which is the one indicated in line 2 of Table~\ref{table:exceptional}.
\end{lem}

\begin{proof}
We first note that $G$ has only one $2$-transitive action up to conjugacy in $\Aut(G)$, namely the standard action of $G$ on the projective space $P_4(2)$, so we only need to consider subgroups $L \le G(v)$ for some fixed $v \in P_4(2)$.  Since $|G|-|G(v)|$ is a multiple of $7$, by Corollary~\ref{cor:double_coset} we may assume $|L|$ is a multiple of $7$.

Let $\mc{L}$ be the set of proper subgroups of $L$ of $G(v)$ that have both of the following properties: $|L|$ is a multiple of $7$, and $L$ acts transitively on $P_4(2) \setminus \{v\}$.  Calculations in GAP reveal that $\mc{L}$ is a single $G(v)$-conjugacy class; given $L_1 \in \mc{L}$ we can write $L_1 = W \rtimes \Alt(7)$, where $W = C^4_2$ is the socle of $G(v)$.

One can check (also it is shown in \cite{Rank3}) that $L_1$ has only three double cosets in $G$, namely, $L_1$ itself: the nontrivial double coset inside $G(v)$; and the remaining double coset is $G \setminus G(v)$.  In particular, $G$ has $2$-by-block-transitive action on $G/L_1$ by Lemma~\ref{lem:double_coset}.  By the previous two paragraphs, up to $\Aut(G)$-conjugacy this is the only proper $2$-by-block-transitive action of $G$.  The stabilizer of a distant pair in this action is a group of the form $\PSL_3(2)$.  
\end{proof}

\subsection{Families of $2$-transitive actions admitting no proper $2$-by-block-transitive extensions}

For some infinite families of $2$-transitive actions, we can deduce from known results that there are no proper $2$-by-block-transitive extensions.

\begin{lem}\label{lem:no_ldc:symmetric}
Let $\Alt(X) \le G \le \Sym(X)$, where $X = \{1,2,\dots,d\}$, $2 \le d < \infty$, and let $s$ be an involution in $\Sym(X)$.  Then there is no proper subgroup $H$ of $G$ such that $G = HsHs$.
\end{lem}

\begin{proof}
Let $H < G$.  If $G = \Sym(X)$ or $d \le 4$, then $sHs$ is conjugate to $H$ in $G$ and the conclusion is clear, so we may assume $G = \Alt(X)$ and $d \ge 5$.  It also suffices to consider the case that $H$ is a maximal subgroup of $G$.  Suppose $G = HsHs$.  Then $sHs$ is also maximal in $G$ and isomorphic to $H$ as a permutation group.  However, one sees from the two cases of \cite[\S1, Corollary 5]{LPS} that it is not possible to write $G = AB$ where $A$ and $B$ are permutationally isomorphic maximal subgroups of $G$.  This contradiction proves the lemma.
\end{proof}

\begin{rem}
If instead of conjugating by a transposition, we applied an automorphism $\theta$ of $G = \Sym(d)$ or $G = \Aut(d)$, we would have the following examples of $H < G$ such that $G = H\theta(H)$: if $\theta$ represents the exotic outer automorphism of $G$ where $G$ is $\Sym(6)$ or $\Alt(6)$, then a point stabilizer $H$ would satisfy $G = H\theta(H)$, since the action of $H$ on $G/\theta(H)$ corresponds to the action of $\Sym(5) \cong \PGL_2(5)$ or $\Alt(5) \cong \PSL_2(5)$ on the projective line.
\end{rem}

\begin{cor}\label{cor:no_ldc:symmetric}
The natural actions of symmetric and alternating groups do not extend properly to $2$-by-block-transitive actions.
\end{cor}

\begin{proof}
Let $G$ be the symmetric or alternating group of degree $d \ge 2$, acting on $X = \{1,\dots,d\}$.  By Lemma~\ref{lem:LDC_2pt}, in order for $G$ to have $2$-by-block-transitive action on $G/L$, a proper subgroup $L < G(1)$ must satisfy
\[
G(1,2) = L(2)sL(2)s,
\]
with $H = L(2)$ being a proper subgroup of $G(1,2)$ and $s \in G$ an involution such that $s(1) = 2$.  However $G(1,2)$ is a symmetric or alternating group, so by Lemma~\ref{lem:no_ldc:symmetric}, no suitable subgroup $H$ exists.
\end{proof}

\begin{lem}\label{lem:LDC_symplectic}
Let $G = \mathrm{Sp}_{2m}(2)$ in one of its $2$-transitive actions, where $m \ge 3$.  Then the setwise stabilizer of an unordered pair of points $\{\omega,\omega'\}$ splits as a direct product
\[
G(\{\omega,\omega'\}) = G(\omega,\omega') \times \grp{s},
\]
where $s$ is an involution swapping $\omega$ and $\omega'$.  Consequently there is no proper extension of the action to a $2$-by-block-transitive action of $G$.
\end{lem}

\begin{proof}
We follow the description of the $2$-transitive actions of $G$ given in \cite[\S7.7]{DixonMortimer}.

Take a vector space $V$ over $\Fb_2$ with basis $\{v_1,\dots,v_m,w_1,\dots,w_m\}$; let $\psi$ be the bilinear form such that
\[
\psi(v_i,w_j) = \delta_{ij}, \; \psi(v_i,v_j) = \psi(w_i,w_j) = \psi(w_i,v_j) = 0,
\]
and let $G$ be the subgroup of $\GL(V)$ preserving the nondegenerate symplectic form
\[
\varphi: (u,v) \mapsto \psi(u,v)-\psi(v,u).
\]
Write $\Omega$ for the set of functions $\theta: V \rightarrow \Fb_2$ such that
\[
\forall u,v \in V: \varphi(u,v) = \theta(u+v) - \theta(u) - \theta(v). 
\]
Equivalently, $\Omega$ consists of quadratic forms on $V$ that can be written as $\theta_a: u \mapsto \psi(u,u)+\varphi(u,a)$ for $a \in V$.

We now have an action of $G$ on $\Omega$ given by $g.\theta_a = \theta_{ga}$, or equivalently $g.\theta_a(u) = \theta_a(g\inv u)$.  We find that $\Omega$ splits into two $G$-orbits
\[
\Omega_+ := \{\theta_a \mid \psi(a,a) = 0\} \text{ and } \Omega_- := \{\theta_a \mid \psi(a,a)=1\},
\]
The actions of $G$ on $\Omega_+$ and $\Omega_-$ are both faithful and represent the two standard $2$-transitive actions of $G$, see \cite[Theorem~7.7A]{DixonMortimer}.  Thus as a permutation group we can take $G$ to be given by its action on $\Omega_{\epsilon}$ with $\epsilon \in \{+,-\}$.

Given $a \in V$, write $a^\perp = \{b \in V \mid \varphi(b,a)=0\}$.  Then $a^\perp$ is a subspace of $V$; moreover, since $\varphi$ is nondegenerate, if $a \neq 0$ then $a^\perp$ has codimension $1$, and if $a \neq b$ then $a^\perp \neq b^\perp$.  Given $a \in V \setminus \{0\}$ we claim that the pointwise fixator $K_a$ of $a^\perp$ in $G$ is cyclic of order $2$, namely $K_a = \grp{t_a}$ where $t_a: u \mapsto u+\varphi(u,a)a$.  On the one hand it is clear that $t_a$ fixes $a^\perp$ pointwise, and it is easy to check that $t_a$ is an involution in $G$ (\cite[Exercise~7.7.5]{DixonMortimer}).  On the other hand, given $g \in K_a$ then $gu-u$ is constant as $u$ ranges over the nontrivial coset of $a^\perp$, so we can write $g: u \mapsto u+\varphi(u,a)c$ for some $c \in V$. Since $g \in G$,
 \[
 \forall u,v \in V: \varphi(u+\varphi(u,a)c,v+\varphi(v,a)c) = \varphi(u,v),
 \]
 in other words 
 \[
 \forall u,v \in V: \varphi(u,a)\varphi(v,c) =  \varphi(v,a)\varphi(u,c),
 \]
 so $c \in \{0,a\}$ and $g \in \grp{t_a}$.
 
We can take our point stabilizer of the action of $G$ on $\Omega_\epsilon$ to be the stabilizer $G(\theta_a)$ of the quadratic form $\theta_a:u \mapsto \psi(u,u)+\varphi(u,a)$, such that $\theta_a \in \Omega_\epsilon$.  Let $\theta_b$ be some other point in $\Omega_\epsilon$. Given $g \in G(\theta_a,\theta_b)$, then $g$ fixes $\theta_a + \theta_b$, from which we obtain the equation
\[
\forall u \in V: \varphi(u,a+b) = \varphi(g\inv u,a+b),
\]
so $G(\theta_a,\theta_b)$ preserves the subspace $W = (a+b)^\perp$.  By \cite[Lemma~7.7A]{DixonMortimer} we know that $s = t_{a+b}$ swaps $\theta_a$ and $\theta_b$, so it does not belong to $G(\theta_a,\theta_b)$.  In particular, $G(\theta_a,\theta_b) \cap K_{a+b} = \triv$, so $G(\theta_a,\theta_b)$ acts faithfully on $W$ and hence commutes with $s$.  Thus the setwise stabilizer $G(\{\theta_a,\theta_b\})$ takes the form $G(\theta_a,\theta_b) \times \grp{s}$.
 
Given $L \le G(\theta_a)$ such that $G$ has $2$-by-block-transitive action on $G/L$, we see by Lemma~\ref{lem:LDC_2pt} that
\[
G(\theta_a) = LsL(\theta_b)s;
\]
however, in the present situation, $LsL(\theta_b)s = L$.  Thus there is no proper extension of $\Omega_{\epsilon}$ to a $2$-by-block-transitive action of $G$.
\end{proof}

\subsection{Projective space actions}

The next case we need to consider is faithful $2$-by-block-transitive actions of groups $G$ with socle $\PSL_{n+1}(q)$ for $n \ge 2$ and point stabilizer $L$, such that the block stabilizer $G_1 \ge L$ is a point stabilizer of the standard action of $G$ on projective $n$-space $P_n(q)$.  (We will deal with the case that $G$ has socle $\PSL_2(q)$ separately; socle $\PSL_3(q)$ also brings up some complications that we will deal with later.)  Since there is no loss of generality in doing so, we will in fact work with groups $Z\SL_{n+1}(q) \le G \le \GaL_{n+1}(q)$, where $Z$ is the group of scalar matrices in $\GL_{n+1}(q)$.

Let us set some hypotheses and notation for this subsection, which will also be reused later.

\begin{hyp}\label{not:psl}
Let $n \ge 2$, let $p$ be a prime, let $q = p^e$ for some $e \ge 1$, let $Z\SL_{n+1}(q) \le G \le \GaL_{n+1}(q)$, where $Z$ is the group of scalar matrices in $\GL_{n+1}(q)$.  Let $\mu$ be a generator of $\Fb^*_q$.  Given $H \le \GaL_{n+1}(q)$ write $H_{\GL} := H \cap \GL_{n+1}(q)$, and write $e_H:= |H:H_{\GL}|$.  Then we note that that $e_H$ is the largest order of field automorphism in $H\GL_{n+1}(q)$.

Let $G$ act on the standard $(n+1)$-dimensional space $V = \Fb^{n+1}_q$, with $P_n(q)$ the corresponding projective $n$-space, and let $W$ be the group of elements of $\SL_{n+1}(q)$ that fix pointwise the spaces $V/\alpha_0$ and $\alpha_0$.  Write $\GL_n(q)$ for the subgroup of elements of $\GL_{n+1}(q)$ that fix the subspace $\alpha_0 = \grp{v_0}_q$ pointwise and preserve the subspace $\grp{v_1,\dots,v_n}_q$, and write $\GL_1(q)$ for the subgroup of elements of $\GL_{n+1}(q)$ that fix the subspace $\grp{v_1,\dots,v_n}_q$ pointwise and preserve $\alpha_0$.  Take the subgroup $\SL_n(q)$ of elements of $\GL_n(q)$ of determinant $1$, and write $M = W \rtimes Z\SL_n(q)$.  We see that the linear part of $G(\alpha_0)$ satisfies
\[
M \unlhd G(\alpha_0)_{\GL} \le W \rtimes (\GL_n(q) \times \GL_1(q)).
\]
Let $s$ be the element of $\SL_{n+1}(q) \le G$ that swaps $v_0$ and $v_1$, sends $v_2$ to $-v_2$ and fixes $v_3,\dots,v_n$.  Thus we have
\[
G(\{\alpha_0,\alpha_1\}) = G(\alpha_0,\alpha_1) \rtimes \grp{s}.
\]
Given $g \in \GL_{n+1}(q)$, we define $\mathrm{Pdet}(g) = \det(gZ)$, regarded as an element of $\Fb^*_q/\grp{\mu^{n+1}}$.

Let $Z \le L \le G(\alpha_0)$; we will be considering for which $L$ the action of $G$ on $G/L$ is $2$-by-block-transitive.  In anticipation of appealing to Lemma~\ref{lem:LDC_2pt} we number the following equations for reference, which may or may not be satisfied:
\begin{equation}\label{eq:linear_mgml}
G(\alpha_0) = LsL(\alpha_1)s;
\end{equation}
\begin{equation}\label{eq:linear_2pt}
G(\alpha_0,\alpha_1) = L(\alpha_1)sL(\alpha_1)s.
\end{equation}
\end{hyp}

\begin{defn}\label{def:psl}
Under Hypothesis~\ref{not:psl}, we will say the action of $G$ on $G/L$ is a \defbold{projective-determinant (PD)} action if it is $2$-by-block-transitive and $M \le L \le G(\alpha_0)$.

If $n=2$, we will say the action of $G$ on $G/L$ is \defbold{quadratic-extended projective plane type (QP)} if it is $2$-by-block-transitive, we have $Z \le L \le G(\alpha_0)$, and writing $G(\alpha_0) \le W \rtimes \GL_2(q) \times \GL_1(q)$ as before, then $LW/W$ normalizes a Singer cycle in $\GL_2(q)W/W$.  (See for instance \cite[\S2]{Hestenes} for some basic information about Singer cycles in linear groups.)  Equivalently,
\[
Z \le L \le WZ\GaL_1(q^2) \cap G(\alpha_0),
\]
where $\GaL_1(q^2)$ acts in the natural way on $\grp{v_1,v_2}_q$ (with respect to some $\Fb_{q^2}$-field structure compatible with the $\Fb_q$-vector space structure).  The terminology will be motivated later, see Example~\ref{ex:qepp}.

Note that the group $WZ\GaL_1(q^2) \cap G(\alpha_0)$ is specified up to $G(\alpha_0)$-conjugacy.  Specifically, $WZ$ is normal in $G(\alpha_0)$ (since $Z$ is central and $WZ/Z$ is the socle of $G(\alpha_0)/Z$), whereas $\GaL_1(q^2)WZ/WZ$ is the normalizer in $\GaL_3(q)(\alpha_0)/WZ \cong \GaL_2(q)$ of a Singer cycle; the Singer cycle is unique up to $\GL_2(q)$-conjugacy, and the conjugation action of $\N_{G(\alpha_0)}(\GL_2(q))$ accounts for all inner automorphisms of $\GL_2(q)$.

If the action of $G$ on $G/L$ is $2$-by-block-transitive, but neither projective-determinant nor QP, we will say it is \defbold{exceptional}.

We carry over the terminology of PD, QP and exceptional actions to the quotient $G/Z$ in the obvious way.
\end{defn}

We first note that PD actions of $G/Z$ are never sharply $2$-by-block-transitive.  For some later results it will also be useful to record the structure of the stabilizer of a pair of lines and of the group $M \cap sMs$.

\begin{lem}\label{lem:psl_2pt}
Assume Hypothesis~\ref{not:psl} and write $\ol{G} = \grp{G,\GL_{n+1}(q)}$.
\begin{enumerate}[(i)]
\item Let $W_i$ be the kernel of the action of $H:=\ol{G}(\alpha_0,\alpha_1)$ on $(V/\alpha_i) \oplus \alpha_i$.  Then $W_i$ is a normal subgroup of $H$ of order $q^{n-1}$, with $W_0 = W(\alpha_1)$.  We then have an $\grp{s}$-invariant normal subgroup $Z^* = W_0 \times W_1 \times Z$ of $H_\GL$, and we have
\[
H = (Z^* \rtimes (K \rtimes \grp{h})) \rtimes \grp{y'_0},
\]
where $K$ is a copy of $\GL_{n-1}(q)$ acting on $\grp{v_2,\dots,v_n}_q$ and fixing $v_0$ and $v_1$; $h$ is the element of $H_\GL$ that sends $v_1$ to $\mu v_1$ and $v_2$ to $\mu\inv v_2$, and fixes $v_0$ and $v_3,\dots,v_n$; and $y'_0$ is a field automorphism of order $e_G$.
\item We have
\[
M \cap sMs = Z^* \rtimes (\SL_{n-1}(q) \rtimes \grp{h^{a_0}}),
\]
where $a_0 = (q-1)/\mathrm{gcd}(q-1,n+1)$.  In particular,
\[
|M \cap sMs:Z| = q^{2(n-1)}\mathrm{gcd}(q-1,n+1)|\SL_{n-1}(q)|.
\]
\item If $G$ has a PD action on $G/L$, then the action of $G/Z$ on $G/L$ is not sharply $2$-by-block-transitive.
\end{enumerate}
\end{lem}

\begin{proof}
(i)
Let $y'_0 = \phi^{e/e_G}$ where $\phi$ is the standard Frobenius automorphism.  We see that $\ol{G} = \GL_{n+1}(q) \rtimes \grp{y'_0}$ and hence $H = H_\GL \rtimes \grp{y'_0}$.  From now on, let us focus on $H_\GL$.

It is clear from the definitions that $W_0$ and $W_1$ are normal in $H$, that $W_0 = W(\alpha_1)$, and that $W_1 = sW_0s$.  Since the $W$-orbit of $\alpha_1$ has size $q$, we have $|W_i| = |W|/q = q^{n-1}$.  From the definitions we see that $W_0 \cap W_1$ acts trivially on $V$ and hence is trivial; moreover, $Z \cap W_0W_1 = \triv$, so we obtain an $\grp{s}$-invariant normal subgroup $Z^* = W_0 \times W_1 \times Z$ of $H$ contained in $H_\GL$.  We can split $H_\GL$ as a semidirect product $Z^* \rtimes C$ by taking $C$ to be the group of matrices that fix $v_0$ and stabilize the spaces $\alpha_1$ and $\grp{v_2,\dots,v_n}_q$.  It is then clear that $C$ can be decomposed as $\GL_{n-1}(q) \rtimes \grp{h}$ as described.  (The choice of $h$ here is made in order to obtain a cyclic complement to $K$ in $C$ that acts with determinant $1$ on $V/\alpha_0$.)

(ii)
Since $Z^*$ is $\grp{s}$-invariant we have $Z^* \le M \cap sMs$, and it is also clear that $M \cap sMs$ contains all elements of $K$ of determinant $1$.  It remains to describe $K' := M \cap sMs \cap (K \times \grp{h})$.

Let $g \in K'$.  Then $g=kh^a$ for some $k \in K$ and $a \in \Zb$, so $g$ fixes $\grp{v_0}_q$ and acts as multiplication by $\mu^{a}$ on $\grp{v_1}_q$.  Since $g \in M \cap sMs$, we also have $g = z_0g_0 = z_1g_1$, where $z_0,z_1 \in Z$ and $g_i$ acts with determinant $1$ on $\alpha_i$ and $V/\alpha_i$.  In particular, we see that $z_0=1$ and $z_1$ is scalar multiplication by $\mu^a$.  Since $g$ and $h^a$ both act with determinant $1$ on $V/\alpha_0$, we see that $\det(k)=1$.  At the same time, calculating the determinant of $g$ on $V/\alpha_1$ in two ways gives
\[
\mu^{-a} = \mu^{an},
\]
so $\mu^{a(n+1)}=1$, in other words, $\mu^a$ is a power of $\mu^{a_0}$.  Thus $g \in \SL_{n-1}(q) \rtimes \grp{h^{a_0}}$.

On the other hand, we have $h \in M$, since $h$ acts with determinant $1$ on $\alpha_0$ and on $V/\alpha_0$.  At the same time, the action of $h^{a_0}$ on $\alpha_1$ has determinant $\mu^{a_0(n+1)} = 1$, and similarly on $V/\alpha_1$; thus $h^{a_0} \in M \cap sMs$.  We deduce that $K' = \SL_{n-1}(q) \rtimes \grp{h^{a_0}}$.  The value of $|M \cap sMs:Z|$ is now clear.

(iii) is clear from the fact that $|M \cap sMs:Z| > 1$.
\end{proof}

Our next aim is to show that if $n \ge 3$, all $2$-by-block-transitive actions are PD, other than when $G = \PSL_5(2)$, a special case that has already been dealt with.

\begin{prop}\label{prop:psl:nonstandard}
Assume Hypothesis~\ref{not:psl}.  Let $Z \le L \le G(\alpha_0)$ be such that $G$ has $2$-by-block-transitive action on $G/L$, with block stabilizer $G(\alpha_0) \ge L$.  Then $W \le L$ and one of the following holds:
\begin{enumerate}[(a)]
\item We have $M \le L$, in other words, the action is PD.
\item $n=2$ and the action of $L$ on $(V/\alpha_0) - \{0\}$ is a transitive subgroup of $\GaL_2(q)$ that does not contain $\SL_2(q)$.
\item $n=4$, $q=2$ and $L$ is the subgroup $C^4_2 \rtimes \Alt(7)$ of $G=\PSL_5(2)$ given in Lemma~\ref{lem:pgl52}.
\end{enumerate}
\end{prop}

The next lemma will be used in the proof of Proposition~\ref{prop:psl:nonstandard} to show that $L$ contains $\SL_n(q)$, except possibly when $n=2$ or $G = \PSL_5(2)$.

\begin{lem}\label{lem:affine_gl}
Let $n \ge 3$, let $p$ be a prime, let $q = p^e$ for some $e \ge 1$, and let $G \le \GaL_n(q)$ be a group acting on $V = \Fb^n_q$.  Suppose that $G$ acts transitively on $V \setminus \{0\}$ and that for all $v \in V$, the action induced by $G_{\grp{v}_q}$ on $V/\grp{v}_q$ contains $\GL_{n-1}(q)$.  Then either $G \ge \SL_n(q)$, or we have $q=2$, $n=4$ and $G = \Alt(7)$.
\end{lem}

\begin{proof}
Since $G$ acts transitively on $V \setminus \{0\}$, the associated affine group $V \rtimes G$ is $2$-transitive, and we can appeal to Lemma~\ref{lem:affine_2trans}.  We may also assume for a contradiction that $G$ does not contain $\SL_n(q)$.  Let $H$ be the action induced by $G_{\grp{v}_q}$ on $V/\grp{v}_q$.

Let us first deal with the case $n=3$.  We see that up to conjugacy, every transitive subgroup of $\GL_3(q)$ that does not contain $\SL_3(q)$ is contained in $\GaL_1(q^3)$.  However, we see that in order to have $H \ge \GL_2(q)$, we would still need $|G:G_{\grp{v}_q}|=q^2+q+1$ and $|G_{\grp{v}_{q}}|$ a multiple of $(q^2-1)(q^2-q)$, and then $|G|$ is too large to be a subgroup of $\GaL_1(q^3)$.  So from now on we may assume $n \ge 4$.  In particular, this means $G$ must involve the nonabelian simple group $\PSL_3(q)$ as a quotient of a subgroup. 

We next consider the possibility that $G$ can be interpreted as a semilinear group of smaller dimension over a larger field, say $G \le \GaL_m(q^d)$ where $dm=n$ and $d > 1$, and we identify $V$ with $W = \Fb^m_{q^d}$ as an additive group.  In this case $G_{\grp{v}_q}$ also stabilizes the $\Fb_{q^d}$-linear span $\grp{v}_{q^d}$ of $v$.  Since $G$ acts irreducibly on $V/\grp{v}_q$ as an $\Fb_q$-vector space, we deduce that $\grp{v}_{q^d} = V$, that is, $d=n$ and $m=1$.  In other words, $W$ is a finite field and each element of $G$ acts as a combination of a multiplication and a field automorphism.    But then $G$ is soluble, a contradiction.

Next, consider the case that $n=2m$ and $p=2$, and $\mathrm{Sp}_{2m}(q) \unlhd G$.  Then $m \ge 2$ and $G/\mathrm{Sp}_{2m}(q)$ is soluble, and we would need $\mathrm{Sp}_{2m}(q)_{\grp{v}_{q}}$ to have a quotient $\SL_{n-1}(q) \le Q \le \GaL_{n-1}(q)$.  For $m \ge 3$ this is easily ruled out by considering the powers of $p$ dividing $|\mathrm{Sp}_{2m}(q)|$ and $|\SL_{2m-1}(q)|$, so we can take $m=2$.  In that case the $p'$-part of the order of $\mathrm{Sp}_{2m}(q)_{\grp{v}_{q}}$ is
\[
\frac{(q^2-1)(q^4-1)}{|P_3(q)|} = (q^2-1)(q-1),
\]
whereas the $p'$-part of $|\SL_3(q)|$ is $(q^2-1)(q^3-1)$, so this case is also ruled out.

Next consider the case that $n=6$ and $\mathrm{G}_2(q) \unlhd G$.  Similar to the last paragraph, we would need $\mathrm{G}_2(q)_{\grp{v}_{q}}$ to have a quotient $\SL_5(q) \le Q \le \GaL_5(q)$, which is easily ruled out by considering the powers of $p$ dividing $|\mathrm{G}_2(q)|$ and $|\SL_5(q)|$.

For $G = \Alt(7)$, $n=4$, $q=2$, we have $G_{\grp{v}_q} = \GL_3(2)$ acting faithfully on $V/\grp{v}_q$.

Finally, suppose $G$ is one the remaining exceptional $2$-transitive affine groups of dimension $n$, where $n \in \{4,6\}$.  Then there is only one insoluble composition factor $S$ of $G$ and it is small: we have $S \le \Alt(6)$ or $S = \PSL_2(13)$.  This leaves $\Alt(5)$, $\Alt(6)$ and $\PSL_2(13)$ as the only nonabelian simple groups that can occur as a quotient of a subgroup of $G$.  One sees that this list excludes $\PSL_3(q)$, so we have no further examples.
\end{proof}

\begin{proof}[Proof of Proposition~\ref{prop:psl:nonstandard}]
By Lemma~\ref{lem:LDC_2pt}, $G$ has $2$-by-block-transitive action on $G/L$ with block stabilizer $G(\alpha_0)$ if and only if (\ref{eq:linear_mgml}) is satisfied.  We observe that  (\ref{eq:linear_mgml}) implies (\ref{eq:linear_2pt}).

Since the linear part of $G$ acts $2$-transitively on the lines in $V$, we observe that $e_{G(\alpha_0,\alpha_1)} = e_G$.  We also deduce from (\ref{eq:linear_2pt}) that $G(\alpha_0,\alpha_1) = L(\alpha_1)G_{\GL}(\alpha_0,\alpha_1)$, hence $e_{L(\alpha_1)} = e_{G(\alpha_0,\alpha_1)}$.  
Let $e_p$ be the largest power of $p$ dividing $e_G$, and note that $e_p < q$.

Since $G$ has $2$-by-block-transitive action on $G/L$, with block stabilizer $G(\alpha_0)$, we see that $L$ acts transitively on the nontrivial cosets of $G(\alpha_0)$, in other words $L$ acts transitively on $P_n(q) \setminus \{\alpha_0\}$.  In particular, given $v,v' \in V \setminus \alpha_0$, there is $g \in L$ such that $gv \in \grp{v'}_q$, and then since $L$ contains the scalar matrices, in fact we can ensure $gv=v'$.  Thus if we write $\theta: G(\alpha_0) \rightarrow \GaL_n(q)$ for the action of $G(\alpha_0)$ on $V/\alpha_0$, we see that $A:= \theta(L)$ is transitive on nonzero vectors.

Write $\beta = \grp{v_0,v_1}_q$ and let $\theta_{V/\beta}$ be the action of $G(\beta)$ on $V/\beta$.  Since (\ref{eq:linear_2pt}) is satisfied and $s$ acts trivially on $V/\beta$, and since $G$ contains $Z\SL_{n+1}(q)$, we see that 
\[
\theta_{V/\beta}(L(\alpha_1)) = \theta_{V/\beta}(G(\alpha_0,\alpha_1)) \ge \GL(V/\beta).
\]
Given Lemma~\ref{lem:affine_gl}, we are therefore in one of the following situations:
\begin{enumerate}[(A)]
\item $A \ge \SL_n(q)$;
\item $n = 2$ and $A \ngeq \SL_n(q)$;
\item $n=4$, $q=2$ and $A = \Alt(7)$.
\end{enumerate}

Observe that $W$ can be regarded as a dual $\Fb_q$-vector space to $V/\alpha_0$, interpreting the action of $W$ on $V$ as an inner product as follows: given $v \in V$ and $w \in W$, then $w(v)-v = \lambda_{w,v} v_0$, where the coefficient $\lambda_{w,v} \in \Fb_q$ only depends on $v$ modulo $\alpha_0$.  We can then define $\langle w,v+\alpha_0 \rangle := \lambda_{w,v}$.  In particular, we can identify $W$ with the dual space of $V/\alpha_0$, and then the conjugation action of $l \in L_{\GL}$ on $W$ is given by $\gamma(l)\rho(\theta(l))$, where $\rho$ is the dual (in other words, inverse transpose) representation of the action of $A_{\GL}$ on $V/\alpha_0$, and $\gamma(l)$ is multiplication by the scalar induced by $l$ on $\alpha_0$.  (In particular, note that $\gamma(l)\rho(\theta(l))$ is trivial if and only if $l$ acts on both $V/\alpha_0$ and $\alpha_0$ as multiplication by the same scalar, which happens if and only if $l \in WZ$.)  In case (A), assuming $(n,q) \not\in \{(2,2),(2,3)\}$, we see that the conjugation action of $L_{\GL}$ on $W$ contains the perfect self-dual group $\SL(W)$, so the action is transitive on $W \setminus \{1\}$.

Notice that (C) implies (c) by Lemma~\ref{lem:pgl52}.  If in addition, we have $L \ge W$, then (A) implies (a) and (B) implies (b).  So all that remains is to prove $L \ge W$ in cases (A) and (B).  We finish the proof with a series of claims.

\emph{Claim 1: Suppose $n=2$; then $L \ge W$.}

Assume for a contradiction that $L \cap W = \triv$.   Then we see that the largest power of $p$ dividing $|L|$ is at most $qe_p$, whereas $|G(\alpha_0)|$ is a multiple of $q^3e_p$.  In this context the equation $G(\alpha_0) = sL(\alpha_1)sL$ can only be satisfied if $q^3e_p \le q^2e^2_p$, in other words, $q \le e_p$, which is impossible.  Thus $L \cap W \neq \triv$.  Now consider $W$ as the dual $\Fb_q$-vector space to $V/\alpha_0$, and let $Y$ represent a line (by which we mean a $1$-dimensional subspace) in $W$.  Since $L$ acts transitively on $V/\alpha_0$, we see that $L$ acts transitively on codimension $1$ subspaces, in other words, lines, in $W$.  Thus the intersection $L \cap Y$ has some order $p^{e'}$, where $1 \le e' \le e$, such that $e'$ does not depend on the choice of $Y$.  Write $e'' = e - e'$ and note that $p^{e''} < q$.  As there are $q+1$ lines, and distinct lines intersect only at the origin, we have
\[
|L \cap W| = (q+1)(p^{e'}-1) +1 = qp^{e'}+p^{e'}-q = p^{e'}(q+1-p^{e''}).
\]
At the same time, $L \cap W$ is a subgroup of $W$, so $|L \cap W|$ is a power of $p$, and hence $(q+1-p^{e''})$ is a power of $p$.  This can only happen if $p^{e''}=1$, in other words, $L \ge Y$.  Then since $L$ is transitive on lines, in fact $L \ge W$ as claimed.

\emph{Claim 2: Suppose $n=3$ and $q \in \{2,3\}$; then $L \ge W$.}

If $q=2$ then $|G \setminus G(\alpha_0)| = 2^7 \cdot 3 \cdot 7^2$, so in order to have $G \setminus G(\alpha_0) = LgL$ for some $g \in G$, as in Lemma~\ref{lem:double_coset}, the order of $L$ must be a multiple of $2^4 \cdot 3 \cdot 7 = 336$.  A calculation shows that $G(\alpha_0)$ has no proper subgroups of suitable order.  If $q=3$ then $G$ is $\PSL_4(3)$ or $\PGL_4(3)$, and $|G \setminus G(\alpha_0)|$ is a multiple of $2^4 \cdot 3^7 \cdot 13^2$, so $|L|$ must be a multiple of $2^2 \cdot 3^4 \cdot 13 = 4212$.  A calculation of maximal subgroups reveals that there is only one case where $G(\alpha_0)$ has a proper subgroup of suitable order, namely when $G = \PGL_4(3)$ and $L = G(\alpha_0) \cap \PSL_4(3)$.  However, the latter case is clearly ruled out by Corollary~\ref{cor:LDC_normal}, proving the claim.

\emph{Claim 3: Suppose (A) holds and $\SL_{n-1}(q)$ is perfect; then $L \ge W$.}

Since $L$ acts transitively on $W \setminus \{1\}$, we may suppose for a contradiction that $L \cap W = \triv$.

Recall that $\beta = \grp{v_0,v_1}_q$ and $\theta_{V/\beta}$ is the action of $G(\beta)$ on $V/\beta$; let $\theta_\beta$ be the action of $G(\beta)$ on $\beta$.  Let $G^* = G(\alpha_0,\beta)$ and $L^* = L \cap G^*$.  Since $L$ acts transitively on $P_n(q) \setminus \{\alpha_0\}$, we see that $L^* = L(\beta)$ acts transitively on the set $X$ of lines in $\beta$ other than $\alpha_0$; note that $|X|=q$.  Moreover, since $L^*_\GL$ is normal in $L^*$ of index dividing $e_G$, we see that $X$ is partitioned into $L^*_\GL$-orbits of equal size, with at most $e_p$ orbits in total.  Thus the index $|L^*_\GL:L_\GL(\alpha_1)|$ is a multiple of $q' := q/e_p$.  Since $e_p < q$, we see that $q'$ is a positive power of $p$.

As noted earlier, we have $\theta_{V/\beta}(L(\alpha_1)) \ge \GL(V/\beta)$.  Since $\SL(V/\beta)$ is perfect, it follows that $\theta_{V/\beta}(L_\GL(\alpha_1)) \ge \SL(V/\beta)$.  At the same time, we have 
\[
\theta_{V/\beta}(L_\GL(\alpha_1)) \le \theta_{V/\beta}(L^*_\GL) \le \GL(V/\beta).
\]
Thus the index of $\theta_{V/\beta}(L_\GL(\alpha_1))$ in $\theta_{V/\beta}(L^*_\GL)$ divides $|\GL(V/\beta):\SL(V/\beta)|$ and hence is coprime to $p$.  In particular, the index
\[
|(\ker\theta_{V/\beta} \cap L^*_\GL):(\ker\theta_{V/\beta} \cap L_\GL(\alpha_1))|
\]
is still a multiple of $q'$.  There is therefore an element $h \in \ker\theta_{V/\beta} \cap L^*_\GL$ of $p$-power order that does not stabilize $\alpha_1$.  We see that $h$ also acts trivially on $\beta/\alpha_0$ and $\alpha_0$.

Let $N$ be the group of elements of $G_\GL$ that act trivially on $V/\beta$, $\beta/\alpha_0$ and $\alpha_0$.  Then we see that $N \le M \cap G^*$, so $N \le WL^*$; on the other hand, $N$ is a normal $p$-subgroup of $G^*$ of order $q^{2n-1}$ that contains $W$.  Thus $N$ is a semidirect product $W \rtimes R$ where $R = L^* \cap N$ is a normal subgroup of $L^*$ of order $q^{n-1}$.  We see that the kernel of the action of $R$ on $V/\alpha_0$ is contained in $W$, hence trivial; that is, $R$ acts faithfully on $V/\alpha_0$.  Let $R_2 = R \cap \ker\theta_\beta$.  By the previous paragraph, $R_2$ is properly contained in $R$; we also see that $|R:R_2|\le q$, so $R_2$ is nontrivial.  Now consider the image $\theta(N)$ of $N$ in $\GL(V/\alpha_0)$.  Since $\theta_{V/\beta}(L^*)$ contains $\SL(V/\beta)$, we see that $\theta(L^*)$ acts transitively by conjugation on $\theta(N) \setminus \{1\}$: the proof is similar to the proof that $L$ acts transitively by conjugation on $W \setminus \{1\}$.  But at the same time, $\theta(L^*)$ clearly preserves $\theta(R_2)$ by conjugation, and we have
\[
\triv < \theta(R_2) < \theta(R) \le \theta(N).
\]
This is a contradiction, so we conclude that $L \ge W$.  This completes the proof of Claim 3 and hence the proposition.
\end{proof}

Given $G$ satisfying Hypothesis~\ref{not:psl}, we now characterize the point stabilizers of the PD actions of $G$.

\begin{prop}\label{prop:psl:standard}
Assume Hypothesis~\ref{not:psl} and take $M \le L \le G(\alpha_0)$. Then the following are equivalent:
\begin{enumerate}[(i)]
\item $G$ has $2$-by-block-transitive action on $G/L$, with block stabilizer $G(\alpha_0)$.
\item We have $e_L = e_G$ and $\det(L_\GL) = \det(G_\GL)$.
\item We have $e_L = e_G$ and the index $|G(\alpha_0):L|$ is coprime to $|\mathrm{Pdet}(G_\GL)|$.
\end{enumerate}
\end{prop}

\begin{proof}
As in the proof of Proposition~\ref{prop:psl:nonstandard}, $G$ has $2$-by-block-transitive action on $G/L$ with block stabilizer $G(\alpha_0)$ if and only if the equation (\ref{eq:linear_mgml}) is satisfied, and we have $e_{G(\alpha_0,\alpha_1)} = e_G$.

Write $\beta = \grp{v_0,v_1}_q$, $G^*= G(\alpha_0,\beta)$ and $L^* = L(\beta)$, and consider the possible equation
\begin{equation}\label{eq:linear_mgml:mod}
G^* = L^*sL(\alpha_1)s.
\end{equation}
We claim that the equations (\ref{eq:linear_mgml}), (\ref{eq:linear_2pt}) and (\ref{eq:linear_mgml:mod}) are equivalent in the present context.  We know in general that (\ref{eq:linear_mgml}) implies (\ref{eq:linear_2pt}).  Note that $G(\alpha_0,\alpha_1) \le G^*$.  Using the action of $L^* \cap W = W(\beta)$, we see that $L^*$ acts transitively on the $1$-dimensional subspaces of $\beta$ other than $\alpha_0$, so that $G^* = L^*G(\alpha_0,\alpha_1)$.  Thus if $G(\alpha_0,\alpha_1) = L(\alpha_1)sL(\alpha_1)s$, then
\[
L^*sL(\alpha_1)s = L^*L(\alpha_1)sL(\alpha_1)s = L^*G(\alpha_0,\alpha_1) = G^*,
\]
showing that (\ref{eq:linear_2pt}) implies (\ref{eq:linear_mgml:mod}).  In turn, since the action of $L$ on $V/\alpha_0$ contains a copy of the special linear group, we see that the $L$-orbit of $\beta$ consists of all $1$-dimensional subspaces of $V/\alpha_0$; in other words, $LG^* = G(\alpha_0)$.  Thus if (\ref{eq:linear_mgml:mod}) holds then
\[
LsL(\alpha_1)s = LG^* = G(\alpha_0),
\]
showing that (\ref{eq:linear_mgml:mod}) implies (\ref{eq:linear_mgml}).

We now observe that the subgroup $M$ of $L_\GL$ acts transitively on $1$-dimensional subspaces of $V$ other than $\alpha_0$, with the result that $e_{L(\alpha_1)} = e_L$.  In turn, in order to satisfy (\ref{eq:linear_2pt}), we need 
\[
e_{L(\alpha_1)} = e_{G(\alpha_0,\alpha_1)} = e_G.
\]
So from now on we may assume $e_L = e_G$.  Note moreover that $M(\alpha_1)$ and $sM(\alpha_1)s$ are both normal in $G(\alpha_0,\alpha_1)$, so if we write $M^* = sM(\alpha_1)sM(\alpha_1)$, then (\ref{eq:linear_2pt}) reduces to
\begin{equation}\label{eq:linear_2pt:reduced}
\frac{G(\alpha_0,\alpha_1)}{M^*} = \frac{L(\alpha_1)sL(\alpha_1)s}{M^*}.
\end{equation}

We have a homomorphism
\[
\phi: G(\alpha_0,\alpha_1) \rightarrow \GaL(V/\beta) \times \GaL(\alpha_1) \times \GaL(\alpha_0),
\]
given by the action of $G(\alpha_0,\alpha_1)$ on the space $V/\beta \oplus \alpha_1 \oplus \alpha_0$.  The assumption $n\ge 2$ ensures that $V/\beta$ is nontrivial.  We see that $M^*$ contains the kernel of this action.

In order to make stabilizers of the general linear group more legible, we will write $\Lambda = \GL_{n+1}(q)$.  Given $g \in \Lambda(\alpha_0,\alpha_1)$, let $\phi_2(g), \phi_1(g), \phi_0(g)$ be the determinants of $g$ acting on $V/\beta, \alpha_1,\alpha_0$ respectively, and define
\[
\delta: \Lambda(\alpha_0,\alpha_1) \rightarrow \Fb^*_q;\quad g \mapsto \phi_2(g)\phi_1(g)\phi_0(g)^{-n}.
\]
Then $M(\alpha_1)$ is the kernel of $\delta$, so $\Lambda(\alpha_0,\alpha_1)/M(\alpha_1)$ is cyclic of order $q-1$.  Given $g \in \Lambda(\alpha_0,\alpha_1)$, we see that $\delta(g) = \det(g)\phi_0(g)^{-(n+1)}$.  If $g \in M(\alpha_1)$ is such that $\phi_i(g) = \mu^{a_i}$ and $\phi_i(sgs) = \mu^{a'_i}$, then we see that $a_2 = na_0 -a_1$; $a'_2 = a_2$; $a'_1 = a_0$; and $a'_0 = a_1$.  Hence
\[
\delta(sgs) = \mu^{na_0-a_1}\mu^{a_0}\mu^{-na_1} = \mu^{(n+1)(a_0-a_1)}.
\]
Since we can choose $a_0$ and $a_1$ freely, we conclude that $M^* = K\ker\delta$ where $K$ is the group of elements of $\Lambda(\alpha_0,\alpha_1)$ of determinant $1$.  In particular, the quotient map from $\Lambda(\alpha_0,\alpha_1)$ to $\Lambda(\alpha_0,\alpha_1)/M^*$ is equivalent to the map
\[
\Lambda(\alpha_0,\alpha_1) \rightarrow \Fb^*_q/\grp{\mu^{n+1}}; \quad g \mapsto \det(g)\grp{\mu^{n+1}},
\]
which is a restriction of the map $\mathrm{Pdet}$.  Note that $\mathrm{Pdet}(s)$ is trivial.  We now see that the action of $s$ on $G(\alpha_0,\alpha_1)/M^*$ is trivial, so (\ref{eq:linear_2pt:reduced}) becomes
\begin{equation}\label{eq:linear_2pt:reduced:bis}
\frac{G(\alpha_0,\alpha_1)}{M^*} = \frac{L(\alpha_1)M^*}{M^*},
\end{equation}
and then since $e_{G(\alpha_0,\alpha_1)} = e_{L(\alpha_1)}$, (\ref{eq:linear_2pt:reduced:bis}) is equivalent to
\begin{equation}\label{eq:linear_2pt:reduced:ter}
\frac{G_\GL(\alpha_0,\alpha_1)}{M^*} = \frac{L_\GL(\alpha_1)M^*}{M^*}.
\end{equation}
We see that (\ref{eq:linear_2pt:reduced:ter}) is satisfied if and only if
\[
\mathrm{Pdet}(L_\GL(\alpha_1)) = \mathrm{Pdet}(G_\GL(\alpha_0,\alpha_1)).
\]
Since $\mathrm{Pdet}(M)$ is trivial and $M$ acts transitively on lines of $V$ other than $\alpha_0$, we see that $\mathrm{Pdet}(L_\GL(\alpha_1)) = \mathrm{Pdet}(L_\GL)$.  Similarly, $\mathrm{Pdet}(G_\GL(\alpha_0,\alpha_1)) = \mathrm{Pdet}(G_\GL)$.  Note also that for any $Z \le H \le G$, then $\det(H_\GL)$ contains $\grp{\mu^{n+1}}$, so $\mathrm{Pdet}(H_\GL) = \mathrm{Pdet}(G_\GL)$ if and only if $\det(H_\GL) = \det(G_\GL)$.  Thus (\ref{eq:linear_2pt:reduced:ter}) is satisfied if and only if $\det(L_\GL) = \det(G_\GL)$.  This completes the proof that (i) and (ii) are equivalent.

Finally, we claim that (ii) and (iii) are equivalent; we may suppose $e_G = e_L$, so that $|G_\GL(\alpha_0):L_\GL| = |G(\alpha_0):L|$.  We have seen that $\det(L_\GL) = \det(G_\GL)$ if and only if $\mathrm{Pdet}(L_\GL) = \mathrm{Pdet}(G_\GL)$, and moreover $\mathrm{Pdet}(G_\GL) = \mathrm{Pdet}(G_\GL(\alpha_0))$.  Thus (ii) is satisfied if and only if $|G(\alpha_0):L|$ is coprime to $|\mathrm{Pdet}(G_\GL)|$, showing that (ii) and (iii) are equivalent.
\end{proof}

We note that if $e_L = e_G$, then to satisfy condition (iii) of Proposition~\ref{prop:psl:standard} it is sufficient but not necessary to have $|G(\alpha_0):L|$ coprime to $\mathrm{gcd}(n+1,q-1)$.  For example, the group $\PSL_3(19)$ has a PD action with block size $3$, with point stabilizer $W \rtimes (\SL_2(19) \rtimes C_2)$, that does not extend to an action of $\PGL_3(19)$; since $|\mathrm{Pdet}(\GL_3(19))| = 3$, there is no PD action of $\PGL_3(19)$ with block size $3$.

\subsection{Some calculations on abelian-by-cyclic groups}

To complete the classification of $2$-by-block-transitive actions, we need to consider certain subgroups of finite abelian-by-cyclic groups.  Specifically, we will be considering the following situation, motivated by Lemma~\ref{lem:LDC_2pt}: $G$ is a finite abelian-by-cyclic group admitting an automorphism $s$ of order $2$, and $H$ is a proper subgroup of $G$ such that $G = Hs(H)$.

We first establish some restrictions on quotients of $G$.

\begin{lem}\label{lem:abelian_intersection}
Let $G$ be a finite group, let $\alpha$ be an automorphism of $G$ and let $H \le G$ be such that $H\alpha(H) = G$.
\begin{enumerate}[(i)]
\item Let $M$ be a normal $\alpha$-invariant subgroup of $G$ such that $G/M$ is cyclic.  Then $G = MH$.
\item Suppose that $H \cap D$ is $\alpha$-invariant, where $D$ is the derived group of $G$.  Then $N:=H \cap \alpha(H)$ is normal in $G$ and $G/N$ is a group of order $|H:N|^2$.
\end{enumerate}
\end{lem}

\begin{proof}
(i)
In the quotient $G/M$, we see that the images of $H$ and $\alpha(H)$ have the same order, so $MH = M\alpha(H)$.  It is then clear that $G = MH\alpha(H) = MH$.

(ii)
Let $R = H \cap D$.  Since $D$ is characteristic, we have
\[
H \cap D = \alpha(H \cap D) = \alpha(H) \cap D,
\]
so in fact $R = N \cap D$.  We see that $H/R$ is abelian, so $N$ is normal in $H$.  Similarly, $N$ is normal in $\alpha(H)$.  Since $G = H\alpha(H)$ it follows that $N$ is normal in $G$.  The quotient $G/N$ is then the product of two subgroups $H/N$ and $\alpha(H)/N$ with trivial intersection; thus $G/N$ has order $|H:N||\alpha(H):N| = |H:N|^2$.
\end{proof}

After dividing out by $H \cap sHs$, we will typically be interested in the situation where $H$ is cyclic.  Finite groups $G = HK$ such that $H$ and $K$ are cyclic subgroups with trivial intersection were studied by Douglas in a series of articles \cite{Douglas}.  We are in effect considering a special case, where in addition some automorphism of $G$ of order $2$ swaps $H$ and $K$.  Our focus here is a little different than in \cite{Douglas} however, as we will be classifying factorizations $\ol{G} = Hs(H)$ of quotients $\ol{G}$ of a given group $G$ of a more special form, rather than constructing all finite groups that admit such factorizations.

If $G = A \rtimes H$ for some abelian $s$-invariant normal subgroup $A$, then $A$ is also very close to being cyclic.

\begin{lem}\label{lem:metabelian_2pt}
Let $G$ be a finite group with an abelian normal subgroup $A$, such that $G/A$ is cyclic and such that $G$ admits an automorphism $s$ of order $2$ that normalizes $A$.  Suppose that there is $H \le G$ such that $H \cap A = \triv$ and $G = Hs(H)$.  Then $A$ has a cyclic subgroup of index at most $2$.
\end{lem}

\begin{proof}
We consider $G$ as embedded in a semidirect product $G \rtimes \grp{s}$ in the obvious way.  By Lemma~\ref{lem:abelian_intersection}(i) we can write $G = A \rtimes H$; note that $H$ is cyclic, say $H = \grp{h}$.

We now suppose that $(G,A,H,s)$ is a counterexample with $|G|$ minimal.  By Lemma~\ref{lem:abelian_intersection}(ii), the intersection $N = H \cap sHs$ is normal in $G$; clearly also $N$ is $s$-invariant, so we can pass from $G \rtimes \grp{s}$ to the quotient $G/N \rtimes \grp{s}$.  By the minimality of $G$, we deduce that $N = \triv$.  Then $G$ has order $n^2$ where $n = |H| = |A|$.

We observe next that $A$ is a $p$-group for some prime $p$.  Otherwise, we could write $A = A_1 \times A_2$ where $A_1$ and $A_2$ are nontrivial and have coprime order.  We would then get a smaller counterexample as either $A_1 \rtimes H$ or $A_2 \rtimes H$.  So by minimality, $A$ must be a $p$-group, and hence $G$ is a $p$-group.  Thus $n = p^e$ for some $e \ge 1$.

Write $shs = ah$ for some $a \in A$.  Since $G$ is finite and $G = HsHs = sHsH$, for all $a' \in A$ there is some $k > 0$ such that $(ah)^k \in a'H$, so $a' = (ah)^k h^{-k}$.  We can rearrange $(ah)^k h^{-k}$ as
\begin{equation}\label{eq:abelian_product}
(ah)^k h^{-k} = \prod^{k-1}_{i=0} h^iah^{-i}.
\end{equation}
Thus as $k$ ranges over the natural numbers, every element of $A$ must be expressible as a product of conjugates of $A$ as in the right hand side of (\ref{eq:abelian_product}).

Since $A$ has less than $n$ nontrivial elements, the $H$-conjugacy classes of $A$ all have size at most $q:= n/p$.  We note also that for all $x \in A$, if $m$ and $qm/2$ are integers then $x^{qm/2}=1$.  If $p$ is odd this follows from the fact that $A$ is not cyclic; if $p=2$ and $x^{qm/2} \neq 1$, then $x$ would have order exactly $q$.  But in the latter case, $A = \grp{x} \times C_2$, which does not yield a counterexample to the lemma.

Let $Z = \Z(G) \cap A$.  Since $A$ is not cyclic, we see from (\ref{eq:abelian_product}) that $a \not\in Z$, so $Z$ is a proper subgroup of $A$.  Since $h^{q} \in \Z(G)$, we can rewrite the product in (\ref{eq:abelian_product}) as follows: writing $k = dq+r$ for $d \ge 0$ and $0 \le r < q$, then
\begin{equation}\label{eq:abelian_product:bis}
(ah)^k h^{-k} = b^d\prod^{r-1}_{i=0} h^iah^{-i},
\end{equation}
where $b = (ah)^qh^{-q} = \prod^{q-1}_{i=0}h^iah^{-i}$; notice that $b \in Z$.  If $b$ is trivial, then the right hand side of (\ref{eq:abelian_product}) can only take at most $q$ values, which is a contradiction.  So $b$ must be some nontrivial element of $Z$.

\emph{Claim 1: $A$ has exponent dividing $4$.  If $h^{q/r}ah^{-q/r} \in \grp{a}Z$ for some divisor $r$ of $q$, then $a^r \not\in Z$.}

We see that $G/A_0$ is a counterexample to the lemma, where $A_0$ is the group of fourth powers of $A$ if $p=2$ and the group of $p$-th powers of $A$ otherwise.  Thus $A_0=\triv$ by minimality of $G$.

Let $r$ be a divisor of $q$ and let $q' = q/r$.  Then we can write $b$ as a product of $q'$ conjugates of $b'$, where
\[
b' = \prod^{r-1}_{i=0} h^{q'i}ah^{-q'i}.
\]
Suppose $h^{q'}ah^{-q'} = az$ for some $z \in Z$.  Then $b' = a^rz^{r(r-1)/2}$.  If we also have $a^r \in Z$ then we would have $b = z^{q(r-1)/2} = 1$, a contradiction.  So if $h^{q'}ah^{-q'} = az$, then $a^r \not\in Z$.

By the minimality of $G$, the quotient $G/Z$ is not a counterexample to the lemma, so either $p=2$ or $A/Z$ is cyclic of odd order $p$.  However, since $h$ has $p$-power order, the latter would imply $hah\inv \in aZ$, which has been ruled out.  Thus $p=2$ and $A$ has exponent dividing $4$.

Now suppose $aZ \neq a\inv Z$ and that $h^{q'}ah^{-q'} = a\inv z$.  Then $r$ is even and $b' = z^{r(r-1)/2}$, so $b = z^{q(r-1)/2}=1$, a contradiction.  Since the order of $a$ divides $4$, we conclude that if $h^{q'}ah^{-q'} \in \grp{a}Z$ then $a^r \not\in Z$, proving the claim.

\

\emph{Claim 2: $A/Z$ has exponent $4$.}

Given Claim 1, we may suppose for a contradiction that $A/Z$ has exponent $2$.  Then $A/Z$ is not cyclic by Claim 1, but also $A/Z \rtimes \grp{h}$ satisfies the lemma by the minimality of $G$, so $A/Z = C_2 \times C_2$ and hence $q \ge 4$.  We then have $hah\inv = ac$ for some $c \in A \setminus Z$ and $hch\inv = cz$ for some $z \in Z$, so
\[
h^2ah^{-2} = hach\inv = ac^2z = az.
\]
By Claim 1, we see that $a^{q/2} \not\in Z$, in particular $a^2 \not\in Z$, which contradicts the structure of $A/Z$.  This contradiction completes the proof of the claim.

\

For the final contradiction, observe that by Claim 2 we have $\Phi(A) \neq \triv$; since $G/\Phi(A)$ is not a counterexample to the lemma, it follows that $A$ is generated by at most $2$ elements.  Given Claim 1, in fact $A = C_4 \times C_4$ and $q=8$.  By Claims 1 and 2, $A/Z$ is not cyclic and has exponent $4$, ensuring that $|Z|=2$.  Write $a' = hah\inv$.  We see from (\ref{eq:abelian_product:bis}) that $a' \not\in a\Phi(A)$, so we can write $A = \grp{a} \times \grp{a'}$.  We then see that $b = a^2(a')^2$, and then examining the structure of automorphisms of $A$ of $2$-power order, we see that $h^2ah^{-2}$ takes the form $a^\epsilon b$ for $\epsilon \in \{1,3\}$.  By Claim 1, we have $a^4 \not\in Z$, which is impossible.  This contradiction completes the proof.
\end{proof}

\begin{ex}Let
\[
G = \grp{a_1,a_2,h \mid a^2_1, \; a^2_2, \; h^4, \; a_1a_2a\inv_1a\inv_2, \; ha_1h\inv a\inv_2, \; ha_2h\inv a\inv_1},
\]
let $A = \grp{a_1,a_2}$ and let $H = \grp{h}$.  Then $G/A$ is cyclic and $G$ admits an automorphism $s$ of order $2$ such that 
\[
s(a_1)=a_1, \; s(a_2) = a_2, \; s(h) = a_1h.
\]
Using (\ref{eq:abelian_product}) from the proof of Lemma~\ref{lem:metabelian_2pt}, one sees that $G = s(H)H = Hs(H)$.  However, $A$ is not cyclic.  Thus the conclusion of Lemma~\ref{lem:metabelian_2pt}, that $A$ has a cyclic subgroup of index at most $2$, is sharp.
\end{ex}

We now specialize to a case that is relevant for applications to rank $1$ groups of Lie type, where we can count the number of conjugacy classes of subgroups $H < G$ of a given index such that $G = sHsH$.  To give a succinct expression for the relevant numbers, we define a modified totient function
\[
\varphi_k(t) := t \prod_{p \text{ prime}, \; p \mid t, \; p \not\mid k}\frac{p-1}{p};
\]
this reduces to Euler's totient function when $k=1$.

\begin{lem}\label{lem:metacyclic_2pt:conjugacy}
Let $G = \grp{x,y}$ be a finite metacyclic group with cyclic normal subgroup $\grp{x}$, and suppose $G$ embeds in a semidirect product $G \rtimes \grp{s}$.  Suppose that $s$ has order $2$, normalizes $\grp{x}$ and centralizes $G/\grp{x}$.  Write
\[
yxy\inv = x^a; \; sxs = x^{k+1}; \; sys = x^ly
\]
for $a,k,l \in \Zb$ with $a > 0$.  Let $k_0 = \mathrm{gcd}(k,l)$; let $d_0$ be the largest natural number coprime to $k_0$ that divides $|G:\grp{x}|$ and $|\grp{x}:\grp{y^{|G:\grp{x}|}}|$.   Write
\[
d_0 = 2^{e_0}p^{e_1}_1 \dots p^{e_r}_r,
\]
where $p_1,\dots,p_r$ are distinct odd primes, $e_0 \ge 0$ and $e_1,\dots,e_r > 0$.  Now set
\[
d = 2^{e'_0}p^{e'_1}_1 \dots p^{e'_r}_r
\]
where 
\[
e'_0 = 
\begin{cases}
e_0 &\mbox{if} \;  a \equiv 1 \mod 4\\
1 &\mbox{if} \;  e_0>0 \text{ and } a \equiv 3 \mod 4\\
0 &\mbox{otherwise}
\end{cases}; \quad
\forall 1 \le i \le r: e'_i = 
\begin{cases}
e_i &\mbox{if} \;  a \equiv 1 \mod p_i\\
0 &\mbox{otherwise}
\end{cases}.
\]
Given a natural number $n$, write $\mc{H}_n := \{H < G \mid |G:H|=n, G = HsHs\}$.  Then $\mc{H}_n$ is a union of $G$-conjugacy classes, which is nonempty if and only if $n$ divides $d$.  Writing $[\mc{H}_n]$ for the set of $G$-conjugacy classes in $\mc{H}_n$, if $\mc{H}_n$ is nonempty we have
\[
|\mc{H}_n| = \varphi_k(n); \quad |[\mc{H}_n]| =  \varphi_k(\mathrm{gcd}(a-1,n)).
\]
\end{lem}

\begin{proof}
Fix a natural number $n$.  Given $m \in \Zb$, write $z_m = x^my$ and $b_m = km+l$.  Let $\alpha(0) = 0$ and thereafter $\alpha(t+1) = a\alpha(t)+1$, so $\alpha(t) = \sum^{t}_{i=1}a^{i-1}$.

We see that $z_mxz\inv_m = x^a$ and that
\[
\quad sz_ms = x^{(k+1)m+l}y = x^{b_m}z_m.
\]
We now observe that $G = \grp{x}\grp{z_m}$ and
\begin{equation}\label{eq:metacyclic}
\forall t > 0: (sz_ms)^t = x^{\alpha(t)b_m}z^t_m.
\end{equation}

Note that $\alpha(t) \mod n$ is eventually periodic, with some period $n_\alpha \le n$; indeed, the period is the least $n'$ such that $\alpha(t) \equiv \alpha(t+n') \mod n$ for some $t \ge 0$.  In particular, we see that $n_\alpha = n$ if and only if the following holds: the values $\alpha(0),\dots,\alpha(n-1)$ form a complete set of congruence classes modulo $n$, and then $\alpha(n) \equiv \alpha(0) \equiv 0 \mod n$.  Define $U_n$ to be the set of integers $u$ such that $u \equiv 1 \mod q$ for all $q$ dividing $n$ such that $q$ is a prime or $q=4$.  By the Hull--Dobell Theorem (see \cite[\S 3.2.1.2 Theorem A]{Knuth}), we have $n_\alpha=n$ if and only if $a \in U_n$.

Write $N = \grp{x^n,y^n}$; clearly $N$ is normalized by $y$.  If $n_\alpha = n$, then $\alpha(n)$ is a multiple of $n$.  We then have
\[
xy^nx\inv = (x^{1-a}y)^n = x^{(1-a)\alpha(n)}y^n,
\]
so $x$ normalizes $N$ and hence $N \unlhd G$.

We note the following conditions on $m$, $n$ and a subgroup $H$ of $G$ for future reference.
\begin{enumerate}[(a)]
\item $b_m$ is a unit modulo $n$;
\item $n_\alpha=n$;
\item $|G:H|=n$;
\item $H = \grp{x^n}\grp{z_m}$.
\end{enumerate}

Given $H \in \mc{H}_n$, then $H$ must satisfy (c).  Since $s$ centralizes $G/\grp{x}$, in order to have $G = HsHs$ we must have $G = \grp{x}H$, and hence $H$ takes the form $(H \cap \grp{x})\grp{z_m}$ for some $m \in \Zb$; moreover, $|G:H| = |\grp{x}:\grp{x} \cap H|$, so $H \cap \grp{x} = \grp{x^n}$.  Thus $H$ also satisfies (d) for some $m \in \Zb$.  On the other hand, given a subgroup $H$ of $G$ satisfying (c) and (d), then using (\ref{eq:metacyclic}), we see that 
\begin{equation}\label{eq:metacyclic:bis}
sHsH = \bigcup_{t > 0}\grp{x^n}x^{\alpha(t)b_m}\grp{z_m} =  \bigcup_{t > 0}x^{\alpha(t)b_m}H;
\end{equation}
thus $G = sHsH$, or equivalently $G = HsHs$, if and only if the parameters $m$ and $n$ satisfy (a) and (b).

Suppose $H \le G$ is such that conditions (a)--(d) are satisfied.  Then $N \unlhd G$ and
\[
z^n_m = (x^{m}y)^n = x^{\alpha(n)m}y^n \in \grp{x^n}y^n,
\]
from which it follows that $N = \grp{x^n,z^n_m}$; in particular, $N \le H$.  Similarly $N \le sHs$, so $N \le H \cap sHs$.  The quotient $G/N$ then satisfies $G/N = \grp{\ol{x}}\grp{\ol{y}}$ where $\ol{x} = xN$ and $\ol{y} = yN$ both have order dividing $n$.   Since $|G:H \cap sHs| = n^2$, in fact $N = H \cap sHs$ and $|G:N|=n^2$, so $G/N$ splits as $\grp{\ol{x}} \rtimes \grp{\ol{y}}$ with $|\grp{\ol{x}}| = |\grp{\ol{y}}| = n$.  In order for this to occur, we see that $n$ must divide $|G:\grp{x}|$ and $|\grp{x}:\grp{y^{|G:\grp{x}|}}|$.

The conclusions we have so far put the following restriction on $n$ in the case that $\mc{H}_n$ is nonempty.  In order for some $m \in \Zb$ to satisfy (a), we need $n$ to be coprime to $k_0$.  The semidirect decomposition of $G/N$ from the last paragraph then ensures that $n$ divides $d_0$.  To additionally satisfy (b), we reduce to the case that $n$ divides $d$.

\

For the rest of the proof we suppose that $n$ is a divisor of $d$; in particular, $n = n_\alpha$, so $N \unlhd G$.  Given that every $H \in \mc{H}_n$ satisfies condition (d), we see that $\mc{H}_n \subseteq \{H_m \mid m \in \Zb\}$ where $H_m := N\grp{z_m}$, so it is enough to consider subgroups $H_m$ of this form.  Given $m \in \Zb$, for $0 < k < n$ we see that $z^k_m$ is not in the normal subgroup $N\grp{x}$ of $G$; however, $z^n_m = x^{\alpha(n)m}y^n \in N$, so $|H_m:N|=n$ and hence $|G:H_m|=n$.  We also see that $H_m \cap \grp{x} = \grp{x^n}$.  If $m \equiv m' \mod n$ then clearly $H_m = H_{m'}$; conversely, if $H_m = H_{m'}$ then 
\[
z_{m'}z^{-1}_m = x^{m'-m} \in (H_m \cap \grp{x}) = \grp{x^n},
\]
so $m \equiv m' \mod n$.  We have $G = H_msH_ms$ if and only if $b_m=km+l$ is coprime to $n$.  Our hypotheses ensure that $k$ and $l$ are coprime modulo $n$, so there exists $m_0$ such that $u_0:=km_0+l$ is a unit modulo $n$ (for instance by Dirichlet's theorem on primes in arithmetic progressions).  The set of possible values of $km+l$ modulo $n$ is then provided by the set $\{u_0+kt \mid t \in \Zb\}$.  Write $n_0$ for the largest factor of $n$ coprime to $k$; the proportion of values of $t$ (modulo $n$) for which $u_0+kt$ is a unit modulo $n$ is
\[
\prod_{p \mid n_0}\frac{p-1}{p} = \frac{\varphi_k(n)}{n}.
\]
Thus $|\mc{H}_n| = \varphi_k(n)$.

Now suppose $H \in \mc{H}_n$ and consider the conjugates of $H$ in $G$.  Since $G = \grp{x}H$, it is enough to consider conjugation by $\grp{x}$.  Here we find that 
\[
xH_mx\inv =  N\grp{x^{m+1}yx\inv} = N\grp{x^{m+(1-a)}y} = H_{m+(1-a)},
\]
and by our assumption on $n$, the number $(1-a)$ is divisible by every prime that divides $n$.  In particular, $km+l$ is a unit modulo $n$ if and only if $k(m+(1-a))+l$ is a unit modulo $n$, so $\mc{H}_n$ is invariant under conjugation by $\grp{x}$; hence $\mc{H}_n$ is a union of conjugacy classes of $G$.  Write $a' = \mathrm{gcd}(a-1,n)$; then $n$ and $a'$ have the same prime divisors, so $\varphi_k(a')/a' = \varphi_k(n)/n$.  The orbits of the action of $\grp{x}$ on $\mc{H}_n$ have size $n/a'$, so each $H \in \mc{H}_n$ has normalizer in $G$ of index $n/a'$.  Thus the number of $G$-conjugacy classes in $\mc{H}_n$ is
\[
|[\mc{H}_n]| = \frac{a'|\mc{H}_n|}{n} = \frac{a'\varphi_k(n)}{n} = \varphi_k(a'). \qedhere
\]
\end{proof}

\begin{rem}
In the situation of Lemma~\ref{lem:metacyclic_2pt:conjugacy}, with $n > 1$, then the set $\mc{H}_n$ admits an action of $s$ that does not preserve any $G$-conjugacy class, and hence $|[\mc{H}_n]|$ must be even.  We can check this from the formula for $|[\mc{H}_n]|$ in the case that $\mc{H}_n$ is nonempty, as follows.  Write $a' = \mathrm{gcd}(a-1,n)$.  If $n$ is even then $k$ is even (since $k+1$ is a unit modulo $n$) and $a'$ is even; thus $\varphi_k(a')$ is even.  If some odd prime $p$ divides $n$ but not $k$, then $p-1$ divides $\varphi_k(a')$, so again $\varphi_k(a')$ is even.  Thus we may assume that $n$ is odd and that every prime dividing $n$ also divides $k$, implying in particular that $k+2$ is coprime to $n$.  However, given that $s^2=1$, we see that 
\[
y = s(sys)s = x^{(k+2)l}y,
\]
so $(k+2)l$ is a multiple of $n$.  Then $l$ is a multiple of $n$, so every prime dividing $n$ divides $\mathrm{gcd}(k,l)$, which is impossible for $n>1$.
\end{rem}

\subsection{Groups of rank $1$ Lie type}

We will say a $2$-transitive permutation group $G$ is of \defbold{small rank Lie type} if its socle is a nonabelian simple group of one of the following forms, in its natural $2$-transitive action:
\[
\PSL_2(q) \; (q \ge 4), \; \mathrm{PSU}_3(q) \; (q \ge 3), \; {^2\mathrm{B}}_2(2^e) \; (e \ge 3 \text{ odd}), \; {^2\mathrm{G}}_2(3^e) \; (e \ge 3 \text{ odd}), \; \PSL_3(q).
\]
The first four types are of Lie rank $1$, while $\PSL_3(q)$ has Lie rank $2$.  (The group ${^2 \mathrm{G}_2(3)} \cong \PGaL_2(8)$ in its $2$-transitive action on $28$ points is also of rank $1$ Lie type, however we exclude it from this discussion as it has already been dealt with in Lemma~\ref{lem:no_ldc:specific}.)  To clarify, the underlying field $\Fb$ of the group is $\Fb_{q^2}$ in the unitary case and $\Fb_q$ otherwise, where $q = p^e$ is a power of the prime $p$; we take the $2$-transitive action of $\PSL_{n+1}(q)$ to be on the set $\Omega_0 = P_n(q)$ of lines in $V = \Fb^{n+1}_q$, which extends to an action of $\PGaL_{n+1}(q)$.

We set another hypothesis to establish some notation.  Note that this hypothesis is compatible with Hypothesis~\ref{not:psl} in the case of socle $\PSL_3(q)$.
\begin{hyp}\label{not:small_rank}
We suppose that $G$ is a group acting on the set $\Omega_0$.  Write $Z$ for the kernel of the action of $G$ on $\Omega_0$.  We suppose one of the following holds:
\begin{enumerate}[(i)]
\item $Z = \triv$ and $G$ has socle $S$ of rank $1$ Lie type, with $\Omega_0$ being the natural $2$-transitive $S$-set;
\item We have $G \le \GaL_3(q)$ and $\Omega_0$ is the set of lines in $V = \Fb^3_q$; $Z$ is the group of scalar matrices in $\GL_3(q)$; and $G \ge S$ where $S = Z\SL_3(q)$.
\end{enumerate}

Enlarge $G$ to the $2$-transitive group $\overline{G}$ and set the parameter $t_0$, as follows:
\begin{enumerate}[(a)]
\item If $S =  \PSL_2(q)$ and $q$ is odd, set $t_0=2$ and $\overline{G} = \grp{\PGL_2(q),G}$.
\item If $S = \mathrm{PSU}_3(q)$ and $q + 1$ is a multiple of $3$, set $t_0=3$ and $\overline{G} = \grp{\mathrm{PGU}_3(q),G}$.
\item If $S = Z\SL_3(q)$ and $q-1$ is a multiple of $3$, set $t_0=3$ and $\overline{G} = \grp{\GL_3(q),G}$.
\item Otherwise, set $t_0=1$ and $\overline{G} = G$.
\end{enumerate}
Let $\Fb$ be the underlying field of $S$. If $S \in \{\PSL_2(q),\mathrm{PSU}_3(q)\}$, let $\tilde{G}$ be the central extension of $\overline{G}$ obtained by lifting $S$ to the corresponding special linear/unitary group, otherwise let $\tilde{G} = \overline{G}$.  Then $\tilde{G}$ has a standard action by semilinear maps on $V = \Fb^{r(S)}$ where 
\[
r(\PSL_{n+1}(q))=n+1, \; r(\mathrm{PSU}_3(q))=3, \; r({^2\mathrm{B}}_2(q))=4, \; r({^2\mathrm{G}}_2(q)) = 7.
\]
We let $\tilde{G}_\GL$ be the subgroup of $\Fb$-linear elements of $\tilde{G}$, let $\overline{G}_\GL$ be the image of $\tilde{G}_\GL$ in $\overline{G}$ and let $G_\GL = \overline{G}_\GL \cap G$.  Let $e_G = |G:G_{\GL}|$ and let $f_G = 2e/e_G$ in the unitary case and $f_G=e/e_G$ otherwise.

\

If $G$ is of rank $1$ Lie type, we take $x_0 \in \overline{G}_\GL$ to be a diagonal element generating a maximal torus. We also have $\phi^{f_G} \in \tilde{G}$ where $\phi$ is the standard Frobenius map on $V$; we write $y_0$ for the image of $\phi^{f_G}$ in $\overline{G}$ and note that $y_0$ normalizes $\grp{x_0}$.\footnote{There are some complications in attaching a field automorphism to a unitary group, but there is no ambiguity with socle $\mathrm{PSU}_3(q)$, see \cite{BHRD}.}  We take points $\omega,\omega' \in \Omega_0$ such that $\ol{G}(\omega,\omega') = \grp{x_0} \rtimes \grp{y_0}$ (see \cite[\S7.7]{DixonMortimer}), and then take $s \in S$ of order $2$ such that $s$ swaps $\omega$ and $\omega'$ and commutes with $y_0$, and such that $sx_0s = x^{k+1}_0$, where if $S =  \mathrm{PSU}_3(q)$ then $k = -(q+1)$, otherwise $k=-2$.

\

For $S = Z\SL_3(q)$, we choose $\omega$ and $\omega'$ to be the standard lines $\alpha_0$ and $\alpha_1$ respectively in $V$; we take $s$ as in Hypothesis~\ref{not:psl}. We consider the decomposition
\[
\overline{G}_\GL(\alpha_0) = W \rtimes Z\Lambda
\]
where $\Lambda \cong \GL_2(q)$ fixes $v_0$ and acts in the natural manner on $\grp{v_1,v_2}_q$, take a Singer cycle $\hat{x}_0 \in \Lambda$, and then set $x_0 = \hat{x}^{q+1}_0$.  Notice that $x_0$ acts as a scalar on $\grp{v_1,v_2}_q$, so $x_0\alpha_1=\alpha_1$; on the other hand, $W\grp{\hat{x}_0}$ acts transitively on $P_2(q) \setminus \{\alpha_0\}$.  Both $x_0$ and $sx_0s$ are diagonal matrices in the standard basis, so they commute.  Writing $R = WZ\grp{\hat{x}_0}$, we see that $\N_{\overline{G}(\alpha_0)}(R)$ can be put in a form
\[
\N_{\overline{G}(\alpha_0)}(R) = WZ \rtimes \grp{\hat{x}_0,y_0},
\]
where we can regard $\grp{\hat{x}_0}$ as a copy of $\GL_1(q^2)$, and where $y_0$ has order $2e_G$, with 
\[
y_0\hat{x}_0y\inv_0 = \hat{x}^{p^{f_G}}_0.
\]
Since $R$ acts transitively on $P_2(q) \setminus \{\alpha_0\}$, we can choose $y_0 \in \overline{G}(\alpha_0,\alpha_1)$.  Let $Z^*$ be as defined in Lemma~\ref{lem:psl_2pt} and let $x_* = sx_0sx\inv_0$.  One finds that $sy_0sy\inv_0 \in Z^*\grp{x_*}$; for a later argument we will want to ensure that in fact $sy_0sy\inv_0 \in Z^*\grp{x^3_*}$, which we can achieve by replacing $y_0$ with $x^a_*y_0$ for a suitable $a \in \Zb$.  After multiplying by an element of $Z$ we can ensure that $y_0$ fixes $v_0$.  The element $y^{e_G}_0$ then fixes $v_0$ and belongs to $\GL_3(q)$, so $y^{e_G}_0 \in \Lambda$.

\

In all cases,
\[
\overline{G} = (S\grp{x_0})\grp{y_0},
\]
where $S\grp{x_0} = \overline{G}_{\GL}$, and we have $|(S\grp{x_0}) \cap \grp{y_0}|=2$ if $S = Z\SL_3(q)$ and $|(S\grp{x_0}) \cap \grp{y_0}|=1$ otherwise.

Write $x = x^{t_G}_0$ for the smallest positive power of $x_0$ contained in $G$ (so $t_G \in \{1,t_0\}$).  For $S = Z\SL_3(q)$, we also have the smallest power $\hat{x} := \hat{x}^{t_G}_0$ of $\hat{x}_0$ contained in $G$, so $\grp{x} = \grp{\hat{x}^{q+1}}$.  Similarly, the element $y_0 \in \ol{G}$ is not always an element of $G$; the most we can ensure is that there is an element $y := x^{r_G}_0y_0$ of $G$, where $0 \le r_G < t_G$.  (The details of how we have chosen $y_0$ will become relevant when $t_G > 1$, as in this case we will find some differences in $2$-by-block-transitive actions between the case $r_G=0$ and the case $r_G \neq 0$.)\end{hyp}

With the almost simple $2$-transitive groups of rank $1$ Lie type, the stabilizer of a pair of points in $\Omega_0$ is cyclic or metacyclic.  We can thus apply the results of the previous subsection.

\begin{prop}\label{prop:LDC:rank_one}
Let $G$ be a group satisfying Hypothesis~\ref{not:small_rank}, of rank $1$ Lie type acting $2$-transitively on the set $\Omega_0$.  Write $\mc{H}_n$ for the class of subgroups $L \le G(\omega)$ such that $|G(\omega):L|=n$ and $G$ has $2$-by-block-transitive action on $G/L$.
\begin{enumerate}[(i)]
\item We can write $G(\omega) = P \rtimes \grp{x,y}$, where $P$ is a $p$-group acting regularly on $\Omega_0 \setminus \{\omega\}$.
\item Let $L \in \mc{H}_n$ for some $n \ge 1$.  Then $L = PN\grp{z}$, where 
\[
z \in \grp{x}y; \; N = \grp{x^n,y^n} \unlhd \grp{x,y}.
\]
The quotient $G(\omega)/PN$ is of the form $\grp{\ol{x}} \rtimes \grp{\ol{z}}$ where $\grp{\ol{x}}$ and $\grp{\ol{z}}$ both have order $n$.
\item Set $o_G = |\grp{x}:\grp{y^{e_G}}|$ and write $\mathrm{gcd}(e_G,o_G) = 2^{e_0}p^{e_1}_1 \dots p^{e_r}_r$, where $p_1,\dots,p_r$ are distinct odd primes.  For each of the odd primes $p_i$ define a condition

$(\mathrm{U}_i)$ $S = \mathrm{PSU}_3(q)$; $p_i$ divides $q+1$; for $p_i=3$ we also require that $9$ divides $q+1$ or $r_G=0$.

We set 
\[
e'_0 = 
\begin{cases}
e_0 &\mbox{if} \;  t_G=2, \; e_G \text{ is even}, \; r_G=1,\; p^{f_G} \equiv 1 \mod 4\\
1 &\mbox{if} \;  t_G=2, \; e_G \text{ is even}, \; r_G=1,\; p^{f_G} \equiv 3 \mod 4\\
0 &\mbox{otherwise}
\end{cases}.
\]
For the exponents of the odd primes $p_i$, we set
\[
e'_i = 
\begin{cases}
e_i &\mbox{if} \;  p^{f_G} \equiv 1 \mod p_i \text{ and $(\mathrm{U}_i)$ is false}\\
0 &\mbox{otherwise}
\end{cases}.
\]
Then the values of $n$ for which $\mc{H}_n$ is nonempty are the divisors of
\[
d_G = 2^{e'_0}p^{e'_1}_1 \dots p^{e'_r}_r.
\]
\item Let $n > 1$ be a divisor of $d_G$ and let $h_n$ be the number of $G(\omega)$-conjugacy classes in $\mc{H}_n$.  Then
\[
h_n = \varphi_{t_G}(\mathrm{gcd}(p^{f_G}-1,n)).
\]
\end{enumerate}
\end{prop}

\begin{proof}
We have $\overline{G}(\omega) = P \rtimes \overline{G}(\omega,\omega')$, where $P$ is a $p$-group contained in $S$ that acts regularly on $\Omega \setminus \{\omega\}$ (see \cite[\S7.7]{DixonMortimer}), and $\overline{G}(\omega,\omega') = \grp{x_0} \rtimes \grp{y_0}$.  The description of $G(\omega)$ in (i) follows easily.  Note that $\grp{S,x} = G_\GL$ has index $e_G$ in $G$.

Let $L \in \mc{H}_n$ for some $n \ge 2$.  Then $L(\omega')sL(\omega')s= \grp{x,y}$ by Lemma~\ref{lem:LDC_2pt}.  The group $\grp{x,y}$ is metacyclic and the normal subgroup $\grp{x}$ is normalized by $s$.  Note that $k$ is a multiple of $t$, hence a multiple of $t_G$.  We see that 
\[
sys = sx^{r_G}_0sy_0 = x^{kr_G}_0y = x^{l_G}y \in \grp{x}y,
\]
where $l_G = kr_G/t_G$; note that $l_G$ can be nonzero only if $t_G > 1$.  In particular, $s$ centralizes the quotient $\grp{x,y}/\grp{x}$.  We can thus apply Lemma~\ref{lem:metacyclic_2pt:conjugacy} to limit the possibilities for $L(\omega')$.

Since $x$ has order coprime to $p$, Lemma~\ref{lem:metacyclic_2pt:conjugacy} ensures that $L(\omega')$ has index in $\grp{x,y}$ coprime to $p$.  We deduce that $|G(\omega):L|$ is likewise coprime to $p$, and thus $L = PL(\omega')$.  Using Lemma~\ref{lem:LDC_2pt} and given $L \le G(\omega)$, we now see that $G$ has $2$-by-block-transitive action on $G/L$ if and only if $P \le L$ and $G(\omega,\omega') = L(\omega')sL(\omega')s$; thus the groups $L \in \mc{H}_n$ are exactly the products $PM$ such that $|\grp{x,y}:M| = n$ and $\grp{x,y} = MsMs$.

The statement (ii) now follows directly from Lemma~\ref{lem:metacyclic_2pt:conjugacy}.  Note also that $L \cap sLs = N$, and the $G(\omega)$-conjugacy classes of $\mc{H}_n$ naturally correspond to conjugacy classes of subgroups $M$ of $\grp{x,y}$ of index $n$ satisfying $\grp{x,y} = MsMs$.

We have
\[
yx_0y\inv = x^{a_G}, \text{ where } a_G = p^{f_G}.
\]
We also have
\[
y^{e_G} = x^{r'_G}_0, \text{ where } r'_G = r_G\sum^{e_G-1}_{i=0} a^i_G.
\]
In particular, for $t_G > 1$ the value of $r_G$ is subject to the additional constraint that $y^{e_G} \in \grp{x}$, so $r'_G$ will be a multiple of $t_G$.  Set
\[
o_G := |\grp{x}:\grp{y^{e_G}}| = \mathrm{gcd}(|\Fb^*|/t_G,r'_G/t_G).
\]

As in Lemma~\ref{lem:metacyclic_2pt:conjugacy} we write $k_0 = \mathrm{gcd}(k,l_G)$.  We have the following cases:
\begin{enumerate}[(A)]
\item If $S = \PSL_2(q)$, $t_G=2$ and $r_G =1$, then $k_0=1$;
\item If $S = \mathrm{PSU}_3(q)$, then either $k_0 = q+1$ or $k_0 = (q+1)/3$, with the latter occurring if $t_G=3$, $r_G > 0$ and $q+1$ is a multiple of $3$;
\item Otherwise, $k_0=2$.
\end{enumerate}

Starting from
\[
\mathrm{gcd}(e_G,o_G) = 2^{e_0}p^{e_1}_1 \dots p^{e_r}_r,
\]
we take a divisor $d_G = 2^{e'_0}p^{e'_1}_1 \dots p^{e'_r}_r$, in such a way that the divisors of $d_G$ satisfy the conditions set out in Lemma~\ref{lem:metacyclic_2pt:conjugacy}.  Specifically, the exponents $e'_i$ are taken as follows.

For the exponent of $2$, in case (B) we see that $e'_0=0$ if $p=2$ (since $o_G$ is odd) and also if $p > 2$ (since $k_0$ is even).  In case (C), $k_0$ is even, so again we have $e'_0=0$.  So let us assume we are in case (A).  We may also assume that $e_G$ and $q-1$ are even; in particular, $q$ is an even power of an odd prime, so $q \equiv 1 \mod 4$.  Under these assumptions, 
\[
\mathrm{gcd}(e_G,o_G) = \mathrm{gcd}\left(e_G, \frac{q-1}{2},\frac{1}{2}\sum^{e_G-1}_{i=0} a^{i}_G \right).
\]
If $a_G \equiv 3 \mod 4$ we see that $\mathrm{gcd}(e_G,o_G)$ is even, and we take $e'_0=1$.  If instead $a_G \equiv 1 \mod 4$, we take $e'_0 = e_0$, that is, the exponent of the largest power of $2$ dividing $\mathrm{gcd}(e_G,o_G)$.

For odd primes, we need to exclude the prime divisors of $k_0$.  In cases (A) and (C) there are no odd prime divisors of $k_0$.  In case (B), the odd prime divisors of $k_0$ are the same as those of $q+1$, with the following exception: if $r_G > 0$ (implying that $t_G = 3$) and $q+1$ is not divisible by $9$, then $3$ divides $q+1$ but not $k_0$.  Thus the prime $p_i$ divides $k_0$ if and only if $(\mathrm{U}_i)$ holds.

We now take
\[
e'_i = 
\begin{cases}
e_i &\mbox{if} \;  p^{f_G} \equiv 1 \mod p_i \text{ and $(\mathrm{U}_i)$ is false}\\
0 &\mbox{otherwise}
\end{cases}.
\]

Part (iii) now follows from Lemmas~\ref{lem:LDC_2pt} and~\ref{lem:metacyclic_2pt:conjugacy}. 

It remains to count the conjugacy classes, using the formula from Lemma~\ref{lem:metacyclic_2pt:conjugacy}.  If $n$ is even, then $e'_0 > 0$ and we see that we are in case (A), so $k=-2$ and $t_G=2$.  Lemma~\ref{lem:metacyclic_2pt:conjugacy} then yields
\[
h_n =  \varphi_{-2}(\mathrm{gcd}(a_G-1,n)) = \varphi_{t_G}(\mathrm{gcd}(p^{f_G}-1,n)).
\]
From now on we may assume $n$ is odd.  If the socle is $\mathrm{PSU}_3(q)$ then we have $k = -(q+1)$, so
\[
h_n =  \varphi_{q+1}(\mathrm{gcd}(p^{f_G}-1,n)).
\]
However, we know that $n$ is coprime to all prime divisors of $q+1$, except possibly the prime $3$; if $3$ divides both $q+1$ and $n$, then we are in the case where $t_G = 3$.  So in fact
\[
h_n =  \varphi_{t_G}(\mathrm{gcd}(p^{f_G}-1,n)).
\]
In the remaining case,
\[
h_n =  \varphi_2(\mathrm{gcd}(p^{f_G}-1,n)) = \varphi(\mathrm{gcd}(p^{f_G}-1,n)) = \varphi_{t_G}(\mathrm{gcd}(p^{f_G}-1,n)).
\]
This completes the proof of part (iv).
\end{proof}

We deduce that the resulting action is never sharply $2$-by-block-transitive.

\begin{cor}\label{cor:LDC:rank_one:sharp}
Let $G$ and $\Omega_0$ be as in Proposition~\ref{prop:LDC:rank_one}.  Then the action of $G$ on $\Omega_0$ does not extend to a sharply $2$-by-block-transitive action of $G$.
\end{cor}

\begin{proof}
We retain the notation of Proposition~\ref{prop:LDC:rank_one}; it is enough to consider the action of $G$ on $G/L$ for $L \in \mc{H}_n$.  From the structure of the point stabilizer, we see that $N$ is the stabilizer of a distant pair in $G/L$ and $N$ is a subgroup of $\grp{x,y}$ of index $n^2$.  So in order to have a sharply $2$-by-block-transitive action with block size $n$, in the notation of Proposition~\ref{prop:LDC:rank_one} we would need 
\[
|\grp{x}| = o_G = e_G = d_G = n.
\]
We note that the equation $|\grp{x}| = |e_G|$ is rarely satisfied.  In the non-unitary case, $|\grp{x}| = (p^e-1)/t_G$ and $e_G$ divides $e$, so in order to have $|\grp{x}| = |e_G|$ we would need
\[
p^e \le t_Ge + 1.
\]
If $t_G=1$, the above inequality is only satisfied when $p=2$ and $e=1$.  If $t_G=2$ then $p \ge 3$, and the only solution is $(p,e)=(3,1)$.  So we are left with the groups $\PSL_2(2) \cong \Sym(3)$ and $\PSL_2(3) \cong \Alt(4)$, both of which are soluble.

In the unitary case, $|\grp{x}| = (p^{2e}-1)/t_G$ and $e_G$ divides $2e$, so we would need
\[
p^{2e} \le 2t_Ge + 1.
\]
If $t_G=1$ there are no solutions.  If $t_G=3$ then $p^e \equiv 2 \mod 3$, and to satisfy the inequality we would need $p=2$ and $e=1$.  Thus $G = \mathrm{PSU}_3(2) \cong C^2_3 \rtimes Q_8$, acting on $9$ points.  However, in this case $G$ is again a soluble group.
\end{proof}

\begin{rem}
The three groups $\PSL_2(2), \PSL_2(3), \mathrm{PSU}_3(2)$ appearing in the proof of Corollary~\ref{cor:LDC:rank_one:sharp} are all sharply $2$-transitive in their natural action, thus sharply $2$-by-block-transitive with trivial blocks.  These examples are excluded from the context of Proposition~\ref{prop:LDC:rank_one}, since we assume $G$ has nonabelian simple socle.
\end{rem}

We also rule out the possibility of proper $3$-by-block-transitive actions arising from Proposition~\ref{prop:LDC:rank_one}.  Since the action on blocks must be $3$-transitive, we only need to consider the case where $G$ has socle $\PSL_2(q)$ and $\Omega_0$ is the projective line.

\begin{prop}\label{prop:PSL:triples}
Let $\PSL_2(q) \le G \le \PGaL_2(q)$ be such that the action of $G$ on the projective line $\Omega_0$ is $3$-transitive, and suppose $\Omega_0$ extends to a $2$-by-block-transitive action on $\Omega = \Omega_0 \times B$, with block size $|B|=n>1$.  Then $G$ has $cn^2$ equally-sized orbits on $\Omega^{[3]}$, where $c=2$ if $n$ is even and $c=1$ otherwise.  In particular, $G$ is not $3$-by-block-transitive.
\end{prop}

\begin{proof}
We retain the notation of Proposition~\ref{prop:LDC:rank_one}.  We can take a block stabilizer of the form $G([\omega_1]) = P \rtimes \grp{x,y}$.  The action of $P \rtimes \grp{x,y}$ is transitive on the points of $\Omega \setminus [\omega_1]$, with $A = \grp{x,y}$ as the stabilizer in $G([\omega_1])$ of a block $[\omega_2]$ and a point stabilizer $H = G([\omega_1],\omega_2)$ of index $n$ in $A$.  Note that $A$ acts transitively on $[\omega_2]$, so $H$ only depends on the choice of representative $\omega_2$ up to conjugation in $A$.  Let $\Omega_1$ be the set of blocks of $\Omega$ other than $[\omega_1]$ and $[\omega_2]$; thus $|\Omega_1|=q-1$.  Considering the action of $\grp{x_0}$ on $\Omega_0$, we see that action of $\grp{x}$ on $\Omega_1$ is free, so it has $t_G$ orbits.  Meanwhile, given that $y = x^{r_G}_0y_0$ where $y_0$ acts on $\Omega_0$ as a field automorphism, we see that if $r_G=0$, then $\grp{y}$ stabilizes some $[\omega_3] \in \Omega_1$, whereas if $r_G=1$ then $y$ swaps the two orbits of $\grp{x}$ on $\Omega_1$.  Overall, $A$ must act transitively on $\Omega_1$ for $G$ to be $3$-transitive, so either $t_G=1$ or $t_G=2$ and $r_G=1$.  We then see that $y^{t_G}_0$ is the smallest power of $y_0$ contained in $G$, and we have
\[
\grp{y^{t_G}_0} = G([\omega_1],[\omega_2],[\omega_3]); \quad K:= G([\omega_1],\omega_2,[\omega_3]) = \grp{y^{t_G}_0} \cap H.
\]
Thus $\grp{y^{t_G}_0} \cap H^*$ fixes $[\omega_2]$ pointwise, where $H^*$ is the intersection of all $A$-conjugates of $H$.  From Proposition~\ref{prop:LDC:rank_one} we see that $H^* \ge N := \grp{x^n,y^n}$.  Thus $\grp{y^{t_Gn'}_0}$ fixes $[\omega_2]$ pointwise, where $n'$ is the least exponent such that $y^{t_Gn'}_0 \in N$.  If $n$ is odd we see that $n'=n$; if $n$ is even, then $t_G=2$ and $y^2_0 = x^{2c'}y^2$ for some $c' \in \Zb$, so $n' = n/2$.  Since $G$ is $3$-transitive on $\Omega_0$, we find that for each $\grp{y^{t_G}_0}$-invariant block $[\omega']$, there is some $g \in \N_G(\grp{y^{t_G}_0})$ such that $[\omega'] = [g\omega_2]$.  Thus $\grp{y^{t_Gn'}_0}$ fixes pointwise every $\grp{y^{t_G}_0}$-invariant block.

The stabilizer of the pair $(\omega_1,\omega_2)$ is $N$, so we must count the number of $N$-orbits on $\Omega \setminus ([\omega_1] \cup [\omega_2])$.  First consider the action of $N$ on $\Omega_1$: here the orbits are equally-sized since $N$ is normal in $A$, and the number of orbits is 
\[
|A:N\grp{y^{t_G}_0}| = \frac{n^2}{|N\grp{y^{t_G}_0}:N|} = \frac{n^2}{n'} = cn.
\]
Then if we consider the $N$-orbits on $\Omega \setminus ([\omega_1] \cup [\omega_2])$ passing through the block $[\omega_3]$, we see that their intersections with $[\omega_3]$ are the orbits of
\[
G(\omega_1,\omega_2,[\omega_3]) = \grp{y^{t_G}_0} \cap N = \grp{y^{t_Gn'}_0}.
\]
However, we have established that $\grp{y^{t_Gn'}_0}$ fixes $[\omega_3]$ pointwise, so there are $n$ orbits of $N$ passing through $[\omega_3]$.  The same situation will arise for any block $[\omega'] \in \Omega_1$, with some $A$-conjugate of $\grp{y^{t_Gn'}_0}$ playing the role of $\grp{y^{t_Gn'}_0}$.  Thus $N$ has $cn^2$ equally-sized orbits on $\Omega \setminus ([\omega_1] \cup [\omega_2])$.  Since $G$ is transitive on distant pairs, we deduce that $G$ has $cn^2$ equally-sized orbits on distant triples.
\end{proof}

\begin{ex}\label{ex:ldc}
Let $p$ be prime and let $n$ be coprime to $2p$, $n > 1$.  Let $m$ be the multiplicative order of $p$ modulo $n$ and suppose that $m$ is odd; for example, one can take $(p,n,m)$ to be $(2,7,3)$, or for examples where $n$ is not a prime power, take $(p,n,m)$ to be $(2,161,33)$ or $(3,143,15)$.  Write $q = p^{mn}$.  Then the field of order $q^2$ has an automorphism $y$ of order $2n$ given by $\lambda \mapsto \lambda^{p^{m}}$, which restricts to an automorphism of order $n$ of the field of order $q$.  Note that we have ensured $q \equiv 1 \mod n$; since $n$ is odd it follows that $q+1$ is a unit modulo $n$, whereas $q-1$ is a multiple of $n$.

Set $G = \mathrm{PGU}_3(q) \rtimes \grp{y}$, in its standard $2$-transitive action on a set $\Omega$.  Then $G$ has a point stabilizer $G(\omega) = P \rtimes G(\omega,\omega')$ where $P$ is a $p$-group acting regularly on $\Omega \setminus \{\omega\}$  and $G(\omega,\omega') = \grp{x,y}$, where $x$ generates a copy of $\Fb^*_{q^2}$.  We also have an involution $s \in G$ such that $s \not\in \grp{x,y}$, $sxs = x^{-q}$ and $sys = y$.  We have ensured that the order of $x$ is divisible by $n$; let $N = \grp{x^n,y^n}$.  Write $\ol{a}$ for the image of an element $a$ under the quotient map $\grp{x,y,s} \rightarrow \grp{x,y,s}/N$.  Since $p^m \equiv 1 \mod n$, the group $\grp{x,y}/N$ is abelian and takes the form of a direct product $\grp{\ol{x}} \times \grp{\ol{y}}$ with both factors being cyclic of order $n$.  Writing $\ol{z} = \ol{x}\ol{y}$, then $\ol{z}$ has order $n$ and $\ol{s}\ol{z}\ol{s} = \ol{x}^{-q}\ol{z}$.  We deduce from Proposition~\ref{prop:LDC:rank_one} that $\grp{x,y}/N = \grp{\ol{z}}\ol{s}\grp{\ol{z}}\ol{s}$.  We thus obtain a proper $2$-by-block-transitive action of $G$ with point stabilizer
\[
L = P \rtimes \grp{x^n,y^n,xy},
\]
which is normal in $G(\omega)$ of index $n$.

Similarly, we can take $G = \mathrm{PGL}_2(q) \rtimes \grp{y}$, again with the standard $2$-transitive action on a set $\Omega$, where now we take $y$ to have order $n$ (in order to make $G$ almost simple).  We obtain a proper $2$-by-block-transitive action of $G$ in the same manner as the previous paragraph, with only the following minor differences: this time $x$ generates a copy of $\Fb^*_{q}$, and the involution $s$ is such that $sxs = x\inv$.  If $p=2$ or $p=3$, the same construction applies to $G = {^2\mathrm{B}}_2(2^{mn}) \rtimes \grp{y}$ or $G = {^2\mathrm{G}}_2(3^{mn}) \rtimes \grp{y}$ respectively.

In the previous two paragraphs, we have arranged that $t_G  = 1$.  The number of equivalence classes of $2$-by-block-transitive actions of $G$ extending the standard $2$-transitive action with block size $n$ is therefore
\[
h_n = \varphi(\mathrm{gcd}(p^m-1,n)).
\]
So for instance if $(p,n,m) = (2,7,3)$, we have $\mathrm{gcd}(p^m-1,n) = 7$ and $h_n = 6$, and thus for the field automorphism $y: \lambda \mapsto \lambda^{2^{3}}$, we obtain $6$ permutationally inequivalent $2$-by-block-transitive actions for $G$, where $G$ is one of the groups
\[
 \mathrm{PGL}_2(2^{21}) \rtimes \grp{y}, \; {^2\mathrm{B}}_2(2^{21}) \rtimes \grp{y}, \; \mathrm{PGU}_3(2^{21}) \rtimes \grp{y},
\]
acting on $7(2^{21}+1),7(2^{42}+1),7(2^{63}+1)$ points respectively with blocks of size $7$.

The existence of these examples can be contrasted with the fact that the groups $\mathrm{PSU}_3(p^{mn})$, ${^2\mathrm{B}}_2(2^{mn})$ and ${^2\mathrm{G}}_2(3^{mn})$ do not occur as the socle of any imprimitive rank $3$ permutation group, see \cite[Propositions~4.7, 4.8, 4.9]{Rank3}. Indeed the rank of the imprimitive permutation groups produced in this example (in other words, the number of double cosets of $L$ in $G$) is $n+1$, and we needed to assume $n > 1$ and $n$ odd for the construction.
\end{ex}

\subsection{Quadratic-extended projective planes}

Before continuing with the classification of finite block-faithful $2$-by-block-transitive actions, we consider a geometric construction that produces both finite examples and infinite examples of sharply $2$-by-block transitive groups.  The construction is not strictly necessary for the proof of the main theorems, but it provides a better intuition for why case (b) of Theorem~\ref{thmintro:2bbtrans} arises than the direct analysis of subgroup structure. The idea of the construction is due to Hendrik Van Maldeghem (oral communication).

\begin{ex}\label{ex:qepp}
Let $K$ be a (finite or infinite) field and let $L = K(\alpha)$ be a quadratic extension of $K$.  We form the $L$-vector space $L^3 = Lv_0 \oplus Lv_1 \oplus Lv_2$; within $L^3$, we have a $K$-subspace $K^3 = Kv_0 \oplus Kv_1 \oplus Kv_2$.  The group $\PGL_3(L)$ acts on the projective plane $P_2(L)$ of $1$-dimensional subspaces of $L^3$; the projective plane $P_2(K)$ of $K^3$ then naturally embeds into $P_2(L)$ (henceforth we identify $P_2(K)$ with its image in $P_2(L)$).  By restricting matrix entries with respect to the given basis, we obtain a subgroup $\PGL_3(K) \le \PGL_3(L)$ stabilizing $P_2(K)$ inside $P_2(L)$.  Given a $2$-dimensional $L$-subspace $M$ of $L$, then $\dim_K(M)=4$, so $\dim_K(M \cap K^3) \in \{1,2\}$.  The two cases for $\dim_K(M \cap K^3)$ can be seen in the projective plane as follows: letting $l_M$ be the line of $P_2(L)$ given by $M$, if $\dim_K(M \cap K^3) = 2$ then $l_M$ contains a line of $P_2(K)$, whereas if $\dim_K(M \cap K^3) = 1$ then $l_M$ contains a unique point $p_K(l_M)$ of $P_2(K)$, and we say $l_M$ is \defbold{tangent} to $P_2(K)$.

We now define the \defbold{quadratic-extended projective plane} $\Pi =: P^L_2(K)$ to be the set of lines of $P_2(L)$ tangent to $P_2(K)$, and let $\PGL_3(K)$ act on $\Pi$ in the natural manner.  Consider a pair $\pi,\pi' \in \Pi$ with $\pi \neq \pi'$.  Then $\pi$ and $\pi'$ intersect in a unique point $p$ of $P_2(L)$.  There are two possibilities: either $p \in P_2(K)$, in which case $p = p_K(\pi) = p_K(\pi')$ and we say $\pi'$ is close to $\pi$, or $p \not\in P_2(K)$, in which case we say $\pi'$ is distant from $\pi$.  Closeness is then a $\PGL_3(K)$-invariant equivalence relation $\sim_K$ on $\Pi$.

\

\emph{Claim 1: $\PGL_3(K)$ is sharply transitive on triples $(p_0,p_1,p_2)$ with the following properties:
\begin{enumerate}[(i)]
\item $p_1,p_2 \in P_2(K)$;
\item $p_0 \in P_2(L) \setminus P_2(K)$;
\item $\{p_0,p_1,p_2\}$ is not collinear in $P_2(L)$;
\item The line through $p_0$ and $p_i$ is tangent to $P_2(K)$ for $i = \{1,2\}$.
\end{enumerate}}

It is clear that $\PGL_3(K)$ respects all the properties listed, so it acts on the specified set of triples.  Clearly also $\PGL_3(K)$ is transitive on pairs of distinct points in $P_2(K)$, so it is enough to consider triples $(p_0,p_1,p_2)$ where $p_1$ and $p_2$ are given: let us take $p_1 = Lv_1$ and $p_2 = Lv_2$, and let $H$ be the subgroup of $\GL_3(K)$ stabilizing $p_1$ and $p_2$; note that $H$ contains the diagonal matrices of $\GL_3(K)$.  A point $p_0 \in P_2(L) \setminus P_2(K)$ satisfying conditions (ii), (iii) and (iv) takes the form $Lw$ for $w = v_0 + \beta_1v_1 + \beta_2v_2$, such that $L = K(\beta_1) = K(\beta_2)$.  Since $L = K \oplus K\alpha$, we see that there is an element $h \in H$ sending $v_0$ to $v_0 + a_1v_1 + a_2v_2$ for some $a_1,a_2 \in K$ and fixing $v_1$ and $v_2$, chosen so that $hw = v_0 + b_1 \alpha v_1 + b_2 \alpha v_2$ for some $b_1,b_2 \in K^*$.  Then by applying a diagonal matrix, we can send $hw$ to the specific vector $x = v_0 + \alpha v_1 + \alpha v_2$.  Since $p_0 = Lw$ was given in general form, we deduce that $\PGL_3(K)$ acts transitively on the triples given in the claim.  Consider now an element $g$ of $\GL_3(K)$ stabilizing $Lx$, $Lv_1$ and $Lv_2$.  Then 
\[
\exists c_1,c_2,d_0,d_1,d_2 \in K: gv_0 = d_0v_0 + c_1v_1 + c_2v_2; \; gv_1 = d_1v_1; \; gv_2 = d_2v_2,
\]
and hence
\[
gx = d_0 v_0 + (c_1 + d_1\alpha)v_1 + (c_2 + d_2\alpha)v_2.
\]
Since $gx \in Lx$ we have
\[
d_0 = c_1\alpha^{-1} + d_1 = c_2\alpha^{-1} + d_2,
\]
from which we see that $0 = c_1 = c_2$ and $d_0 = d_1 = d_2$, so $g$ represents the trivial element of $\PGL_3(K)$.  Thus $\PGL_3(K)$ acts sharply transitively on the given triples, proving the claim.

\

\emph{Claim 2: The stabilizer $\Lambda = \PGL_3(K)(\pi)$ of a single tangent line $\pi$ is of the form $\AGL_1(L)$.}

Take for instance the tangent line $\pi$ spanned by $Lv_0$ and $Lw$, where $w = v_1+\alpha v_2$ and consider $g \in \GL_3(K)$ preserving $\pi$: say $g$ has matrix entries $a_{ij}$ with respect to the standard basis.  Then $g$ stabilizes $Lv_0$ and
\[
gv_0 = a_{00}v_0 + a_{10}v_1 + a_{20}v_2; \; gw = (a_{01}+a_{02}\alpha )v_0 + (a_{11}+ a_{12}\alpha)v_1 + (a_{21}\alpha^{-1}+ a_{22}) \alpha v_2,
\]
so $a_{11}+a_{12} \alpha  = a_{21} \alpha^{-1} + a_{22} \neq 0$ and hence $(a_{21},a_{22})$ is determined by $(a_{11},a_{12})$, and also $a_{10} = a_{20} = 0$.  With respect to the ordered basis $(v_0,w,v_2)$, then $g$ has matrix entries $b_{ij}$ such that $b_{10} = b_{20} = b_{21} = 0$ and $b_{ij} = a_{ij}$ for $ij \in \{00,02,12\}$; the remaining entries are
\[
b_{01} = a_{01}+a_{02} \alpha ; \quad b_{11} = a_{11}+ a_{12} \alpha  \neq 0; \quad b_{22} = a_{22} - a_{12}\alpha.
\]

As an element of $\GL_3(K)(\pi)$, we see that $g$ is determined by the values $a_{00},b_{01},b_{11}$, but we can choose independently $a_{00} \in K^*, b_{01} \in L, b_{11} \in L^*$.  Setting $b_{11}=1$ yields a normal subgroup $F$ of $\GL_3(K)(\pi)$ consisting of those elements that act trivially on $L^3/Lv_0$, whereas setting $a_{00} = 1$ and $b_{01} = 0$ yields a subgroup $F^*$ consisting of those elements that fix $v_0$ and stabilize $Lw$; we also have the centre of $\GL_3(K)$ inside $\GL_3(K)(\pi)$, which arises as those elements of $\GL_3(K)(\pi)$ with $b_{01}=0$ and $b_{11}=1$.  Given the degrees of freedom of the parameters, we see that $\GL_3(K)(\pi) = (Z \times F) \rtimes F^*$.  We see that the group $F \rtimes F^*$ acts faithfully on $Lv_0 + Lw$ as a copy of $\AGL_1(L)$, and hence
\[
\Lambda = \GL_3(K)(\pi)/Z \cong \AGL_1(L).
\]
completing the proof of Claim 2.

\

From Claim 1, we deduce that $\PGL_3(K)$ is sharply transitive on distant pairs in $\Pi$: distant pairs $(\pi,\pi')$ are in bijection with triples as in Claim 1 by sending $(\pi,\pi')$ to the triple $(p_K(\pi),p_K(\pi'),\pi \cap \pi')$.  Thus we have a sharply $2$-by-block transitive action of $\PGL_3(K)$ on $\Pi$ with respect to the equivalence relation $\sim_K$, with block stabilizer $\AGL_2(K)$ and point stabilizer $\AGL_1(L)$.

The action of $\PGL_3(K)$ can be extended by any group $R$ of field automorphisms of $L$ that normalize $K$.  Let $R$ act on $L^3$ by acting on the coefficients of the standard basis vectors; this yields a natural semidirect product $\PGL_3(L) \rtimes R$.  Then $R$ preserves $K^3$, so there is a subgroup of the form $\PGL_3(K) \rtimes R$; in the induced action on $P_2(L)$, the subplane $P_2(K)$ is preserved by $R$, as is the set of tangent lines to $P_2(K)$.  Thus $R$ acts on $\Pi$ and respects the block structure, yielding a $2$-by-block-transitive (but not necessarily block-faithful) action of $\PGL_3(K) \rtimes R$ on $\Pi$.  Since $\PGL_3(K)$ is already transitive on $\Pi$, the point stabilizers are likewise extended by a copy of $R$.

More variations on the space $\Pi$ are possible; one that is relevant to the finite groups setting is the following.  We suppose that $L/K$ is Galois, $\mathrm{Gal}(L/K) = \grp{\theta}$; note that in this case, if $R \ge \grp{\theta}$ is a group of field automorphisms of $L$ normalizing $K$, then $\theta$ is central in $R$.  There is then a natural action of $\grp{\theta}$ on $L^3$ by acting on the coordinates in the standard basis, and hence an action of $\grp{\theta}$ on $P_2(L)$; the latter action fixes $P_2(K)$ pointwise and preserves the set of tangent lines, so we obtain an action of $R$ on $\Pi$.  For each $\pi \in \Pi$ there is thus a conjugate $\ol{\pi} := \theta(\pi)$ that is close to $\pi$, with $\ol{\ol{\pi}} = \pi$.  Since $\PGL_3(K)$ acts by elements of the fixed field, it commutes with the action of $\theta$; we also have $R$ commuting with $\theta$.  Hence the pairs of conjugate tangent lines form a system of imprimitivity for the action of $\PGL_3(K) \rtimes R$, and we obtain another $2$-by-block-transitive $(\PGL_3(K) \rtimes R)$-set $\Omega/\theta$.  Indeed, since the action of $\theta$ has been factored out, we have an action of $\PGL_3(K) \rtimes (R/\theta)$ on $\Pi/\theta$.

This means, for example, if $e$ is a natural number and $|K| = p^{e}$, we have an action of $\PGL_3(K) \rtimes \grp{\phi}$ on $\Pi$, where $\phi$ is the Frobenius map on $L$, and hence has order $2e$, and the action of $\phi^e = \theta$ on $\Pi$ is nontrivial but stabilizes each block, so the action is not block-faithful.  The permutation group induced on the blocks in $\PGaL_3(K) \cong \PGL_3(K)\grp{\phi}/\grp{\theta}$, and the action of the latter extends to a $2$-by-block-transitive action on $\Pi/\theta$ with point stabilizer $\AGaL_1(L)$.
\end{ex}

\

In the rest of this subsection and the next, we will consider the remaining possibilities for $G$ to have $2$-by-block-transitive action on $G/L$, where $G$ satisfies Hypothesis~\ref{not:psl}.  The PD case and the exceptional case (c) of Proposition~\ref{prop:psl:nonstandard} have already been dealt with; we are thus left with case (b) of Proposition~\ref{prop:psl:nonstandard}, where $G/Z$ has socle $\PSL_3(q)$ and we have an action of $L$ on $(V/\alpha_0) \setminus \{0\}$ as a transitive subgroup $A$ of $\GaL_2(q)$ that does not contain $\SL_2(q)$.  We may thus assume $G$ satisfies Hypothesis~\ref{not:small_rank}.

We recall from Lemma~\ref{lem:psl_2pt} the structure of $\ol{G}(\alpha_0,\alpha_1)$, and the subgroups $W_0, W_1 = sW_1s$ and $Z^* =W_0 \times W_1 \times Z$, which is a subgroup of $G_\GL(\alpha_0,\alpha_1)$.  In the present context, we can write $\ol{G}(\alpha_0,\alpha_1)$ as $Z^*\ol{B}$, where
\[
\ol{B} \cap \GL_{n+1}(q) =: \ol{B}_\GL = (s\grp{x_0}s \times \grp{x_0}); \; \ol{B} = \ol{B}_\GL\grp{y_0}.
\]
Since $Z^* \le G$, we obtain a similar description of $G(\alpha_0,\alpha_1)$ as $Z^*B$, where $B:=G \cap \ol{B}$ is a subgroup of index $t_G$ in $\ol{B}$, and likewise $B_\GL:=G_\GL \cap \ol{B}$ has index $t_G$ in $\ol{B}_\GL$.  More precisely, recalling that $G$ has elements $x = x^{t_G}_0$ and $y = x^{r_G}_0y_0$ for $0 \le r_G < t_G$, we see that $B = B_\GL\grp{y}$ and $B_\GL$ consists of elements of the form $sx^a_0sx^b_0$ such that $a+b$ is a multiple of $t_G$.

Note that $G_\GL(\alpha_0,\alpha_1) \cap \grp{y} = \grp{y^{e_G}}$ and $|\grp{y^{e_G}}|=2$.  If $p=2$ then $y^{e_G} \in Z^*$, since $Z^*$ is normal and of odd index in $G_\GL(\alpha_0,\alpha_1)$.

The main outstanding case we need to deal with is that of a QP action (recall Definition~\ref{def:psl}), where $A$ is contained in a copy of $\GaL_1(q^2) \le \GaL_2(q)$.  As we saw in Example~\ref{ex:qepp}, there is indeed a $2$-by-block-transitive action of $\PGaL_3(q)$ on a set $\Omega$ with point stabilizer $\AGaL_1(q^2)$, where $\Omega$ takes the form of the set of $\mathrm{Gal}(q^2/q)$-orbits on the quadratic-extended projective plane.  In the next proposition we will learn that all QP actions are minor variations on this theme.

First consider $L$ as a subgroup of $\ol{G}(\alpha_0)$.  Up to conjugation, $L$ will be contained in a group of the form
\[
\ol{L}^{\GaL_1} = WZ \rtimes \grp{\hat{x}_0,y_0},
\]
where $\grp{\hat{x}_0,y_0}$ corresponds to the subgroup of $\GaL_1(q^2)$ of index $e/e_G$ that contains $\GL_1(q^2)$.

Taking into account that $L \le G$, in fact $L$ is conjugate in $\ol{G}(\alpha_0)$ to a subgroup of
\[
L^{\GaL_1} = WZ \rtimes \grp{\hat{x},y},
\]
which has index $t_G$ in $\ol{L}^{\GaL_1}$.  Now $\ol{G}(\alpha_0) = G(\alpha_0)\grp{\hat{x}_0}$, and for $t \in \Zb$ we see that
\[
G(\alpha_0) \cap \hat{x}^t_0L^{\GaL_1}\hat{x}^{-t}_0 \le L^{\GaL_1}.
\]
So in fact $L$ is $G(\alpha_0)$-conjugate to a subgroup of $L^{\GaL_1}$.  So to understand QP actions of $G$, it is enough to consider $G/L$ such that $L \le L^{\GaL_1}$.

For future reference, we name some subgroups of $L^{\GaL_1}$ of small index: for $d_x,d_y \in \{1,2\}$, write
\[
L_{d_x,d_y} = WZ \rtimes \grp{\hat{x}^{d_x},\hat{x}^{d_x-1} y^{d_y}}; \; L^*_{21} = WZ \rtimes \grp{\hat{x}^2,y}.
\]

\begin{prop}\label{prop:psl:plane-field}
Assume Hypothesis~\ref{not:psl} with $n=2$ and Hypothesis~\ref{not:small_rank}.  Let $L^{\GaL_1}$ be as above, and write $\mc{H}$ for the class of subgroups $L$ of $L^{\GaL_1}$ such that $Z \le L$ and such that $G$ has $2$-by-block-transitive action on $G/L$.    Then the groups $L_{11}, L_{12}, L_{21}, L^*_{21}$ as above are all in distinct $G$-conjugacy classes, except if $p=2$, in which case $L_{11} = L_{21} =  L^*_{21}$.  The set $\mc{H}$ consists of those groups of the form $L_{d_x,d_y}$, where $d_x,d_y \in \{1,2\}$ and $d_xd_y = |L^{\GaL_1}:L| \le 2$, such that the following additional conditions are satisfied:
\begin{enumerate}[(i)]
\item if $e_G$ is even, then $d_x=d_y=1$;
\item if $t_G=3$, then $r_G \neq 0$, $e_G$ is a multiple of $3$ and $p^{e/e_G} \equiv 1 \mod 3$.
\end{enumerate}
In particular, we have $0 \le |\mc{H}| \le 3$.

Moreover, if $L \in \mc{H}$, then
\[
|L \cap sLs:Z| = \frac{4e_G}{|L^{\GaL_1}:L|^2 t_G};
\]
the action of $G/Z$ on $G/L$ is sharply $2$-by-block-transitive if and only if $e_G = t_G$ (equivalently, $|G| = |\GL_3(q)|$) and either $d_y=2$, or $p>d_x=2$.
\end{prop}

\begin{proof}
By Proposition~\ref{prop:psl:nonstandard}, we only need to consider subgroups $L$ of $L^{\GaL_1}$ such that $WZ \le L$.  Thus we may assume that $L = WZ \rtimes (L \cap \grp{\hat{x},y})$.  Let $H_L = L(\alpha_1) \cap \grp{\hat{x},y}$ and note that $H_L \le B$.

Let $w \in W$, $z \in Z$ and $k,l \in \Zb$, and let $g = wz\hat{x}^ky^l$.  Then $y^l$ stabilizes $\alpha_1$ and $z$ acts trivially on $P_2(q)$, so $g\alpha_1 = \alpha_1$ if and only if $w\hat{x}^k\alpha_1 = \alpha_1$.  Since $w\hat{x}^k \in \GL_3(q)$ it is enough to consider the image of $v_1$.  Our choice of $\hat{x}$ ensures that $\hat{x}^kv_1 = \nu_1v_1+\nu_2v_2$ for some $\nu_1,\nu_2 \in \Fb_q$, and then $w\hat{x}^kv_1 = \nu_1v_1+\nu_2v_2 + \nu_0v_0$ for some $\nu_0 \in \Fb_q$.  Thus if $w\hat{x}^k\alpha_1 = \alpha_1$, then $\nu_2 = 0 = \nu_0$ and $\nu_1 \in \Fb^*_q$, in other words $\hat{x}^k$ stabilizes $\alpha_1$, and then also $w$ stabilizes $\alpha_1$, so $w \in W_0$.  Thus for all $L$ such that $WZ \le L \le L^{\GaL_1}$, we have
\[
L(\alpha_1) = W_0Z \rtimes H_L,
\]
and hence
\[
sL(\alpha_1)sL(\alpha_1) =  W_1ZsH_LsW_0ZH_L \subseteq Z^*sH_LsH_L.
\]
Write $G^+ = G(\alpha_0,\alpha_1)$ and $G^+_\GL = G_\GL(\alpha_0,\alpha_1)$.
Under our current assumptions we see that the equation (\ref{eq:linear_2pt}) from earlier is equivalent to
\begin{equation}\label{eq:2pt_reduced}
G^+ = Z^*sH_LsH_L.
\end{equation}

Note that $Z^*/Z$ is a $p$-Sylow subgroup of $G^+_\GL/Z$.  The intersection $\grp{x,y} \cap Z^*$ is therefore $\grp{y^{e_G}}$ in the case $p=2$, and trivial otherwise.

Let $WZ \le L \le L^{\GaL_1}$ and write $d= |L^{\GaL_1}:L|$.  Let $\mc{H}$ be the set of subgroups $L$ of $L^{\GaL_1}$ such that $G$ has $2$-by-block-transitive action on $G/L$.  We will proceed via a series of claims.

We have a restriction on $d$ using the size of the large double coset of $L$ in $G$.

\

\emph{Claim 1: If $L \in \mc{H}$, then
\[
|L \cap sLs:Z| = \frac{4e_G}{d^2 t_G}.
\]
In particular, $4e_G/(d^2 t_G)$ is an integer.
}

We have
\[
|L:Z| = \frac{|L^{\GaL_1}:Z|}{d} = \frac{2q^2(q^2-1)e_G}{dt_G};
\]
by comparison,
\[
|G(\alpha_0):Z| = \frac{q^3(q-1)(q^2-1)e_G}{t_G},
\]
so
\begin{align*}
|G:Z| - |G(\alpha_0):Z| &= |G(\alpha_0):Z|(|G:G(\alpha_0)|-1) \\
&= q^3(q-1)(q^2-1)(q^2+q)e_G/t_G = q^4(q^2-1)^2e_G/t_G.
\end{align*}
Applying Corollary~\ref{cor:double_coset} to $G/Z$, we deduce that
\[
|L \cap sLs:Z| = \frac{|L:Z|^2}{|G:Z| - |G(\alpha_0):Z|} = \frac{4q^4(q^2-1)^2e^2_Gt_G}{d^2t^2_Gq^4(q^2-1)^2e_G} = \frac{4e_G}{d^2 t_G}
\]
as claimed.

\

Since $L$ acts transitively on lines of $V$ other than $\alpha_0$, we can calculate $d$ as $d = |L^{\GaL_1}(\alpha_1):L(\alpha_1)|$; in turn, given the form taken by $L(\alpha_1)$, we see that $|L^{\GaL_1}(\alpha_1):L(\alpha_1)| = |H_{L^{\GaL_1}}:H_L|$, where $H_{L^{\GaL_1}} = \grp{x,y}$.  The next claim puts more restrictions on the value of $d$.

\

\emph{Claim 2: Suppose $L \in \mc{H}$.  If $e_G$ is even, then $d=1$; otherwise, $d \le 2$.}

We first consider integers $b > 0$ such that $x^ay^b \in H_L$ for some $a \in \Zb$.  We have $G^+ = G^+_\GL\grp{y}$ and $G^+_\GL \cap \grp{y} = \grp{y^{e_G}}$; since $s$ centralizes the quotient $G^+/G^+_\GL$, we see via (\ref{eq:2pt_reduced}) that $G^+ = G^+_\GL H_L$.  If $e_G$ is even then we deduce that $x^ay \in H_L$ for some $a \in \Zb$, and if $e_G$ is odd then $x^ay^2 \in H_L$ for some $a\in \Zb$.

Write $Q = H_L \cap B_\GL$ and $Q^{\GaL_1} = H_{L^{\GaL_1}} \cap B_\GL$.  We have $Q^{\GaL_1} = \grp{x}$ if $p=2$ and $Q^{\GaL_1} = \grp{x,y^{e_G}} \cong \grp{x} \times C_2$ if $p > 2$.

We claim next that $Q$ contains $x^l$ for $l$ coprime to $k$, for all odd primes $k$.  If $k=t_G=3$ and $q \not\equiv 1 \mod 9$, or if $k$ does not divide $q-1$, then $|\grp{x}|$ is coprime to $k$ and the conclusion is clear.  So assume that $k$ divides $q-1$, and if $k=t_G=3$, assume that $q \equiv 1 \mod 9$.  Let $R_k$ be group of $k$-th powers of $B_\GL$.  Then $B_\GL/R_k$ does not have a cyclic subgroup of index $\le 2$, so after applying Lemma~\ref{lem:metabelian_2pt} to the group $B/R_k$, we see that $Q$ is not contained in $R_k$.  We deduce that $Q$ contains $x^l$ where $l$ is coprime to $k$.

Since the previous paragraph applies to all odd primes $k$, we see that $Q$ contains $x^{2^a}$ for some $a \ge 0$; since $Q$ also contains $x^by$ or $x^by^2$ for some $b \in \Zb$, it follows that $d$ is a power of $2$.  If $p=2$ we deduce that in fact $x \in Q$, and the claim follows in this case.

We now assume $p>2$ and consider powers of $2$ dividing $d$.  If $p>2$ and $e_G$ is odd, we see from Claim 1 that $d$ cannot be a multiple of $4$, so $d \le 2$.

Now suppose $p>2$ and $e_G$ is even; in particular, $q$ is an even power of $p$, so $q-1$ is a multiple of $4$.  We consider the quotient $B/R_2$ of $B$, where $R_2$ is the set of squares in $B_\GL$; note that $R_2$ is normalized by $s$ and contains $y^{e_G}$, so we can write 
\[
B/R_2 = B_\GL/R_2 \times \grp{yR_2}; \; B_\GL/R_2 = \grp{sx^3_0sR_2} \times \grp{x^3_0R_2}.
\]
In order to achieve (\ref{eq:2pt_reduced}) we see that $H_LR_2/R_2$ must contain a nontrivial element of $B_\GL/R_2$ (since otherwise the intersection of $Z^*(sH_LsH_L)R_2/R_2$ with $B_\GL/R_2$ would be contained in a cyclic subgroup).  Thus $H_L$, and hence also $Q$, contains a nonsquare element of $B_\GL$.  Given the form of $Q^{\GaL_1}$, this can only be achieved if $L$ contains an odd power of $x$, and then by the previous paragraph we deduce that $x \in Q$.  This completes the proof of the claim.

\

From now on, we can assume that $d \le 2$ and that $L = WZ \rtimes \grp{\hat{x}^{d_x},\hat{x}^\epsilon y^{d_y}}$, such that $d_x,d_y \in \{1,2\}$; $d_xd_y = d$; and $0 \le \epsilon < d_x$.  Note that if $p=2$ then $\hat{x}$ has odd order, so $\grp{\hat{x}} = \grp{\hat{x}^2}$; in that case we will set $d_x=1$.  Since $L$ has index at most $2$ in $L^{\GaL_1}$, we have $y^2 \in L$.  We also have $x \in L$: if $p=2$ we know that $\hat{x} \in L$, while if $p>2$ then $\hat{x}^2 \in L$ and $x$ is an even power of $\hat{x}$.  Thus $\grp{x,y^2} \le H_L$.

\

\emph{Claim 3: $L$ acts transitively on $P_2(q) \setminus \{\alpha_0\}$ if and only if $\epsilon = d_x-1$.}

Suppose $\epsilon \neq d_x-1$; then $d_x=2$ and $\epsilon=0$.  Then $p > 2$, so $q+1$ is even, and since $y$ stabilizes $\alpha_1$, we see that the $L$-orbit of $\alpha_1$ is the same as the $WZ\grp{\hat{x}^2}$-orbit of $\alpha_1$.  The latter orbit has size $q(q+1)/2 < q(q+1)$, so $L$ is not transitive on $P_2(q) \setminus \{\alpha_0\}$, which is incompatible with $G$ having $2$-by-block-transitive action on $G/L$.

On the other hand, suppose that $\epsilon = d_x-1$; thus $(d_x,\epsilon)$ is either $(2,1)$ or $(1,0)$.  In either case, we see that the $L$-orbit of $\alpha_1$ contains the $W\grp{\hat{x}}$-orbit, which has size $q(q+1)$, so indeed $L$ acts transitively on $P_2(q) \setminus \{\alpha_0\}$.

\

From now on we assume that $\epsilon = d_x-1$, so $L = L_{d_x,d_y}$ where $(d_x,d_y) \in \{(1,1),(1,2),(2,1)\}$.  We also assume $d_x=d_y=1$ in the case that $e_G$ is even.  By Lemma~\ref{lem:LDC_2pt} we have $L \in \mc{H}$ if and only if the equation (\ref{eq:linear_mgml}) is satisfied; by Claim 3, $L$ is transitive on $P_2(q) \setminus \{\alpha_0\}$, so (\ref{eq:linear_mgml}) is equivalent to (\ref{eq:linear_2pt}); and as noted earlier, (\ref{eq:linear_2pt}) is equivalent to (\ref{eq:2pt_reduced}).  In fact we can refine (\ref{eq:2pt_reduced}) further.  Writing $Z^{**}:=Z^*\grp{sxs,x}$ it is clear that (\ref{eq:2pt_reduced}) implies
\begin{equation}\label{eq:2pt_reduced:bis}
G^+ = Z^{**}sH_LsH_L.
\end{equation}
On the other hand, we see that $Z^*\grp{x}$ and $Z^*\grp{sxs}$ are normal in $G^+$.  If $g \in G^+$ can be written as 
\[
g = z(sx^bs)x^a(sh_1s)h_2
\]
for $z \in Z^*$, $a,b \in \Zb$ and $h_1,h_2 \in H_L$, then since $x \in H_L$, we can rearrange to express $g$ as $z'sh'_1sh'_2$ for $z' \in Z^*$ and $h'_1,h'_2 \in H_L$.  Thus (\ref{eq:2pt_reduced}) is equivalent to (\ref{eq:2pt_reduced:bis}).

The next claim will determine which of the groups $L_{d_x,d_y}$ belong to $\mc{H}$, with one more restriction appearing for $t_G = 3$.  Write $f_G = e/e_G$.

\

\emph{Claim 4: We have $L \in \mc{H}$ if and only if one of the following holds:
\begin{enumerate}[(a)]
\item $t_G=1$;
\item $t_G=3$, $e_G$ is a multiple of $3$, $r_G \neq 0$ and $p^{f_G} \equiv 1 \mod 3$.
\end{enumerate}}

If $e_G$ is even then $y \in H_L$.  If $e_G$ is odd then $y^2 \in H_L$ and $y^{e_G} \in Z^{**}$, and we can write $y = y^{e_G+2a}$ for some $a \in \Zb$.  So in either case $y \in Z^{**}H_L$.  If $t_G=1$ then $G^+ = Z^{**}\grp{y}$, so $G^+ = Z^{**}H_L$ and the equation (\ref{eq:2pt_reduced}) is satisfied.

For the rest of the proof of the claim, we may assume $t_G=3$; given Claim 1, we may then assume $e_G$ is a multiple of $3$.  Let $x_* =  sx_0sx\inv_0$.  The group $B^{**} = G^+/Z^{**}$ is then a metacyclic group $\grp{x_*Z^{**},yZ^{**}}$ of order $3e_G$.  The image of $L^{\GaL_1}(\alpha_1)$, hence also $H_L$, in $B^{**}$ contains $yZ^{**}$ and has index $3$; thus
\[
Z^{**}L(\alpha_1) = Z^{**}\grp{y}.
\]
Note that $sx_*s = x^{k+1}_*$ where $k=-2$.  By Lemma~\ref{lem:metacyclic_2pt:conjugacy}, to achieve $L \in \mc{H}$ we must have $yx_*y\inv \in Z^{**}x^a_*$ where $a \equiv 1 \mod 3$, so in fact $yx_*y\inv \in Z^{**}x_*$.  In particular, we may assume from now on that $p^{f_G} \equiv 1 \mod 3$.

We have $y = x^{r_G}_0y_0$ and $s$ commutes with $y_0$ modulo $Z^{**}$; since
\[
sx^{r_G}_0s = sx^{r_G}_0s x^{-r_G}_0 x^{r_G}_0  = x^{r_G}_*x^{r_G}_0,
\]
we see that $sys \in Z^{**}x^{r_G}_*y$.  It is now easy to see that
\[
Z^{**}\grp{sys}\grp{y} = Z^{**}\grp{x_*,y}
\]
if and only if $r_G \neq 0$.  In turn, $Z^{**}\grp{x_*,y} = G(\alpha_0,\alpha_1)$; by the discussion before the claim, we deduce that (\ref{eq:2pt_reduced:bis}) is satisfied if and only if $r_G \neq 0$, completing the proof of the claim.

\

The next claim distinguishes the conjugacy classes of the groups in $\mc{H}$.

\

\emph{Claim 5: The groups of the form $L_{11}, L_{12}, L_{21}, L^*_{21}$ are not conjugate to one another in $G$, except for the case $L_{11} = L_{21} = L^*_{21}$, which happens when $p=2$.}

When $p=2$ we can ignore $L_{21}$ and $L^*_{21}$ as duplicates of $L_{11}$.  Since $L_{11}$ properly contains the others, it is not conjugate to them; Claim 3 distinguishes $L^*_{21}$ from the others.  Thus we may assume $p>2$ and we only need to show that $L_{12}$ is not conjugate to $L_{21}$.  It is enough to consider conjugacy in $G(\alpha_0)$; since $WZ$ is normal in $G(\alpha_0)$ and contained in every $L \in \mc{H}$, we can consider the images $\tilde{L}$ of the groups $L \in \mc{H}$ in $G(\alpha_0)/WZ$.  We argue that $|Z_{12}| > |Z_{21}|$, where $Z_{d_x,d_y}$ is the centre of of $\tilde{L}_{d_x,d_y}$.  The group $\tilde{L}_{d_x,d_y}$ takes the form 
\[
\tilde{L}_{d_x,d_y} = \grp{\tilde{x}^{d_xt_G}} \rtimes \grp{\tilde{y}^{d_y}}
\]
such that $\tilde{x}$ has order $(q^2-1)$ and $\tilde{y}\tilde{x}\tilde{y}\inv = \tilde{x}^{p^{f_G}}$.  (Here $\tilde{y}$ could be the image of $y$ or of $\hat{x}y$ in $\tilde{L}_{d_x,d_y}$; the distinction is not important for the structure of $\tilde{L}_{d_x,d_y}$.)  We see that 
\[
Z_{d_x,d_y} \cong C_{a_{d_x,d_y}} \times C_{b_{d_x,d_y}}
\]
for some natural numbers $a_{d_x,d_y}$ and $b_{d_x,d_y}$, where the first and second cyclic groups are generated by the powers of $\tilde{x}$ and $\tilde{y}$ respectively in $Z_{d_x,d_y}$.  We can estimate the orders of the cyclic groups as follows:
\begin{align*}
a_{12} = \mathrm{gcd}((q^2-1)/t_G,p^{2f_G}-1); &\; a_{21} = \mathrm{gcd}((q^2-1)/2t_G,p^{f_G}-1); \; 2b_{12}/b_{21} \in \Zb.
\end{align*}
In particular, $|Z_{12}|$ is a multiple of $|Z_{21}|(p^{f_G}+1)/2t_G$.  So to have $|Z_{12}| \le |Z_{21}|$, we need $p^{f_G}+1 \le 2t_G$.  Since $p$ is odd, this only leaves the cases $p \in \{3,5\}, f_G=1,t_G=3$.  However, these cases are incompatible with Claim 4.  Thus  $|Z_{12}| > |Z_{21}|$, so $L_{12}$ and $L_{21}$ are not conjugate in $G(\alpha_0)$, which proves the claim.

\

Putting the claims together, we have proved the characterization of $\mc{H}$ as stated in the proposition.  It remains to characterize when the action is sharply $2$-by-block-transitive.  By Claim 1, for $L \in \mc{H}$, this occurs exactly when $4e_G = d^2t_G$.  Given that $t_G \in \{1,3\}$ and $d \in \{1,2\}$, the only solutions are $d=2$ and $e_G = t_G \in \{1,3\}$.
\end{proof}

Writing $QP(d_x,d_y)$ for the $2$-by-block-transitive $G$-set obtained in Proposition~\ref{prop:psl:plane-field}, we see that $QP(1,2) \cong P^{\Fb_{q^2}}_2(\Fb_q)$ and $QP(1,1) \cong P^{\Fb_{q^2}}_2(\Fb_q)/\theta$, where $\theta$ is the field automorphism $x \mapsto x^q$ of $\Fb_{q^2}$.  It is not clear if the $G$-set $QP(2,1)$ has a natural geometric interpretation, or if its construction extends to groups defined over infinite fields.

\subsection{Exceptional actions with socle $\PSL_3(q)$}

To finish the classification of $2$-by-block-transitive actions extending the action on a finite projective plane, Proposition~\ref{prop:psl:nonstandard} and Lemma~\ref{lem:affine_2trans} ensure that we have only finitely many groups to consider, which will give us the exceptional actions in the sense of Definition~\ref{def:psl} (other than the action of $G = \PSL_5(2)$ on cosets of $C^4_2 \rtimes \Alt(7)$, which is also exceptional in the sense of Definition~\ref{def:psl}, but we dealt with it separately in Lemma~\ref{lem:pgl52}).  Specifically, we can assume $G$ and $L$ satisfy Hypotheses~\ref{not:psl} and~\ref{not:small_rank}, and that there is a subgroup $A$ of $\GaL_2(q)$, acting transitively on nonzero vectors, such that neither $A \ge \SL_2(q)$ nor $A \le \GaL_1(q^2)$.  Taking account of the possible field sizes $q$ and groups between $\PSL_3(q)$ and $\PGaL_3(q)$, there are a total of eleven groups $G/Z$ to consider:
\begin{equation}\label{eq:exceptional_list}
G/Z \in \{\PGaL_3(q) \mid q \in 5,7,9,11,19,23,29,59\} \cup \{\PSL_3(7),\PSL_3(9),\PSL_3(19)\}.
\end{equation}

In the next proposition we list all the $2$-by-block-transitive actions of $G/Z$ satisfying (\ref{eq:exceptional_list}) that properly extend the action on $P_2(q)$ and are not PD or QP.  We obtain a total of sixteen exceptional $2$-by-block-transitive actions this way.  All calculations of subgroups of small index and double coset enumerations are done using GAP.

\begin{prop}\label{prop:psl_exceptional}
Assume Hypotheses~\ref{not:psl} and~\ref{not:small_rank}.  Suppose that $L$ neither contains $WZ \rtimes \SL_2(q)$, nor is contained in a copy of $WZ \rtimes \GaL_1(q^2)$.  Then $G$ has $2$-by-block-transitive action on $G/L$ if and only if the action is as given in Table~\ref{table:exceptional}.
\end{prop}

\begin{proof}
By Proposition~\ref{prop:psl:nonstandard}, we may assume that $L \ge W$ and that acts transitively on the nonzero vectors in $V/\alpha_0$, so $L = WZ \rtimes A$ for some transitive subgroup $A$ of $\GaL_2(q)$ that is not contained in $\GaL_1(q^2)$.  The possibilities for $A$ are then limited by Lemma~\ref{lem:affine_2trans}, such that $G/Z$ satisfies (\ref{eq:exceptional_list}).  As in the proof of Proposition~\ref{prop:psl:plane-field}, we have 
\[
|G(\alpha_0):Z| = q^3(q-1)(q^2-1)e_G/t_G.
\]
We will refer to the index of $L$ in $G(\alpha_0)$ as the block size $b$, and write $a = q(q-1)/b \in \Qb$.  Then by Corollary~\ref{cor:double_coset}, given $g \in G \setminus G(\alpha_0)$ we have
\begin{align}\label{psl:exceptional_2pt}
|L \cap gLg\inv:Z| &= \frac{|L/Z|^2}{|G(\alpha_0)/Z|q(q+1)} = \frac{|G(\alpha_0)/Z|}{b^2q(q+1)} \nonumber \\ 
&= \frac{q^3(q-1)(q^2-1)e_G}{b^2q(q+1)t_G} = \frac{a^2e_G}{t_G}.
\end{align}
In particular, $a^2e_G/t_G$ must be an integer.  In the present situation, $e_G \in \{1,2\}$ and $t_G \in \{1,3\}$, so in fact $a$ is an integer, in other words, $b$ divides $q(q-1)$; indeed, $b$ divides $q(q-1)/t_G$.  If the action on $G/L$ is $2$-by-block-transitive, we see that it is sharply $2$-by-block-transitive if and only if $e_G = t_G = a=1$.

On the other hand, if $b \le q$, then in most cases $L$ contains $\SL_2(q)$ and hence gives rise to a PD action (assuming that $G$ acts $2$-by-block-transitively on $G/L$).  The exception is if there is some proper subgroup $R(q)$ of $\SL_2(q)$ of index $\le q$, which occurs only for $q=5,7,9,11$, with $|\SL_2(q):R(q)|=5,7,6,11$ respectively, see \cite[Satz 8.28]{Huppert}.  In that case, we need to consider multiples of $|\SL_2(q):R(q)|$ as possible values of $b$.

Write $\Gamma = \GaL_3(q)$ and $S = Z\SL_3(q)$.  We will take $G \in \{\Gamma,S\}$.

Write $\mc{L}_G$ for the set of proper groups $L'$ of $G(\alpha_0)$ satisfying the following conditions:
\begin{enumerate}[(i)]
\item $L' \ge Z$;
\item $|G(\alpha_0):L'|$ divides $q(q-1)/t_G$;
\item $|G(\alpha_0):L'| > q$ or $|G(\alpha_0):L'|$ is a multiple of $|\SL_2(q):R(q)|$ (or both).
\end{enumerate}
Calculations are performed on the quotient $G/Z$.

If $t_G=1$, write $\mc{L}^*_G$ for the groups in $\mc{L}_G$ conjugate to a subgroup of $L^{\GaL_1}$ as in Proposition~\ref{prop:psl:plane-field}.  Then $\mc{L}^*_G$ contains four classes, one with block size $q(q-1)/2$ and the others with block size $q(q-1)$; these correspond to the four conjugacy classes of subgroups $L_{11}, L_{12}, L_{21}, L^*_{12}$ that appeared in Proposition~\ref{prop:psl:plane-field}.  We can exclude these groups from consideration.

We split the rest of the proof according to the field size $q$.

\paragraph{\underline{$q=5$.}}  There are three conjugacy classes in $\mc{L}_G \setminus \mc{L}^*_G$, of block size $5,10,20$.  These have representatives $L$ of the form $L_5 \le L \le \N_G(L_5)$, where $L_5$ represents the unique class of subgroups of $G(\alpha_0)$ of the form $WZ \rtimes \SL_2(3)$; we note that $\N_G(L_5)$ takes the form $WZ \rtimes (\SL_2(3) \rtimes C_4)$.  A double coset enumeration shows that $G$ acts $2$-by-block-transitively on $G/L_5$, hence also on $G/L$ for any overgroup $L$ of $L_5$ in $G(\alpha_0)$.

By (\ref{psl:exceptional_2pt}) we have $|L \cap gLg\inv|/|Z| = a^2$.  More precisely, we find that $(L \cap gLg\inv)/Z \cong C^2_a$.

\paragraph{\underline{$q=9$.}} Excluding subgroups of $S$, there are three conjugacy classes in $\mc{L}_\Gamma \setminus \mc{L}^*_\Gamma$: one of block size $12$ and two of block size $24$.

Double coset enumerations with respect to representatives of $\mc{L}_\Gamma \setminus \mc{L}^*_\Gamma$ now show that there is one exceptional $2$-by-block-transitive action of $\Gamma$, of block size $12$.  The point stabilizer has the form
\[
L'_9 = WZ \rtimes (\SL_2(5).D_8),
\]
and the stabilizer of a distant pair takes the form $P'_9 = \Sym(3)^2 \times C_2$.

There are two conjugacy classes in $\mc{L}_S \setminus \mc{L}^*_S$, of block sizes $12$ and $24$.  By double coset enumeration with respect to the remaining candidates, we obtain one exceptional $2$-by-block-transitive action of $S$, of block size $12$, which is obtained by restricting the $2$-by-block-transitive action of $\Gamma$.  We have
\[
L'_9 \cap S = WZ \rtimes (\SL_2(5).C_4); \; P'_9 \cap S = \Sym(3)^2.
\]

\paragraph{\underline{$q=11$.}} There are four conjugacy classes in $\mc{L}_G \setminus \mc{L}^*_G$: one of block size $22$, one of block size $55$ and two of block size $110$.  The two classes of block size $110$ have representatives of the forms
\[
L_{11} = WZ \rtimes (\SL_2(3) \times C_5); \; L'_{11} = WZ \rtimes \SL_2(5).
\]
We find by double coset enumeration that $G$ has $2$-by-block-transitive action on both $G/L_{11}$ and $G/L'_{11}$; since the block size is exactly $q(q-1)$, the stabilizer of a distant pair is trivial in both cases.  The remainder of $\mc{L}_G$ is now accounted for by
\[
\N_G(L_{11}) = WZ \rtimes (\GL_2(3) \times C_5); \; \N_G(L'_{11}) = WZ \rtimes (\SL_2(5) \times C_5),
\]
of block sizes $55$ and $22$ respectively; clearly $G$ also has $2$-by-block-transitive action on both $G/\N_G(L_{11})$ and $G/\N_G(L'_{11})$.

As in the case $q=5$, we find in all cases that $(L \cap gLg\inv)/Z \cong C^2_a$.

\paragraph{\underline{$q \in \{7,19,23,29,59\}$.}}  There are five or six conjugacy classes in $\mc{L}_\Gamma$, of which four are accounted for by $\mc{L}^*_\Gamma$.  The remaining classes consist of a class of block size $q(q-1)$, represented by a group $L''_q = WZ \rtimes A_q$, and the other class (if there is one) is represented by $\N_\Gamma(L''_q) = WZ \rtimes N_q$.  The groups are as follows:
\begin{align*}
A_{7} = \SL_2(3).2&; \; N_7 = \SL_2(3).2 \times C_3\\
A_{19} = \SL_2(5) \times C_3&; \; N_{19} = \SL_2(5) \times C_9 \\
A_{23} = N_{23} &= \SL_2(3).C_2 \times C_{11}\\
A_{29} =  \SL_2(5) \times C_7&; \; N_{29} = (\SL_2(5) \rtimes C_2) \times C_7\\
A_{59} = N_{59} &= \SL_2(5) \times C_{29}.
\end{align*}
Suppose $L = WZ \rtimes A$, where $A$ is one of the groups listed above.  The groups $L = L''_7$ and $L = L''_{19}$ are contained in a proper normal subgroup $S$ of $\Gamma$, so in these cases, by Corollary~\ref{cor:LDC_normal}, $\Gamma$ does not have $2$-by-block-transitive action on $\Gamma/L$.  By contrast, for all the other candidates listed for $L$ (including $WZ \rtimes N_7$ and $WZ \rtimes N_{19}$), we find by double coset enumeration that $\Gamma$ has $2$-by-block-transitive action on $\Gamma/L$.

The remaining groups are $S = Z\SL_3(q)$ for $q \in \{7,19\}$.  For each of them, there is only one conjugacy class in $\mc{L}_S$, represented by $L''_q$.  By double coset enumeration, we find that $S$ has $2$-by-block-transitive action on $S/L''_q$ for $q=7$ but not for $q=19$.

It remains to describe $(L \cap gLg\inv)/Z$ for the possible $2$-by-block-transitive actions with $q \in \{7,19,23,29,59\}$.  In most cases, the answer is clear from (\ref{psl:exceptional_2pt}), as $(L \cap gLg\inv)/Z$ is trivial or has prime order.  The exceptions are for $L = \N_\Gamma(L''_{q})$ for $q \in \{7,19\}$, in which case we find that $(L \cap gLg\inv)/Z \cong C^2_3$, and for $L = \N_\Gamma(L''_{29})$, in which case we find that $(L \cap gLg\inv)/Z \cong C^2_2$. 
\end{proof}

\begin{ex}\label{ex:small_q}
It will be useful for later applications to note the set of $2$-by-block-transitive actions of $G$ for $\PSL_3(q) \le G \le \PGaL_3(q)$ and $2 \le q \le 5$ that properly extend the action on $P_2(q)$; these are listed in Table~\ref{table:small_q}.  For a QP action, the pair of numbers indicates the value of $(d_x,d_y)$.

\begin{table}[h!]
\begin{small}
\begin{center}

\begin{tabular}{ c | c | c | c }

$G$ & $G(\omega)$ & type & $|B|$  \\ \hline
$\PSL_3(2)$ & $W \rtimes C_3$ & QP$(1,2)$ & $2$ \\ \hline 
$\PSL_3(3)$ & $W \rtimes \SL_2(3)$ & PD & $2$  \\
 & $W \rtimes \GaL_1(9)$ & QP$(1,1)$ & $3$  \\
 & $W \rtimes C_8$ & QP$(1,2)$ & $6$  \\ 
 & $W \rtimes Q_8$ & QP$(2,1)$ & $6$  \\ \hline
$\PGL_3(4)$ & $W \rtimes (C_{15} \rtimes C_2)$ & QP$(1,1)$ & $6$  \\
 & $W \rtimes C_{15}$ & QP$(1,2)$ & $12$  \\ \hline
$\PGaL_3(4)$ & $W \rtimes \GaL_1(16)$ & QP$(1,1)$ & $6$ \\ \hline
$\PSL_3(5)$ & $W \rtimes \SL_2(5).C_2$ & PD & $2$  \\
 & $W \rtimes \SL_2(5)$ & PD & $4$  \\
 & $W \rtimes \GaL_1(25)$ & QP$(1,1)$ & $10$  \\
 & $W \rtimes C_{24}$ & QP$(1,2)$ & $20$  \\
 & $W \rtimes (C_3 \rtimes C_8)$ & QP$(2,1)$ & $20$  \\
 &  $W \rtimes (\SL_2(3) \rtimes C_4)$ & exceptional & $5$  \\
  &  $W \rtimes (\SL_2(3) \rtimes C_2)$ & exceptional & $10$  \\
  &  $W \rtimes \SL_2(3)$ & exceptional & $20$  \\
\end{tabular}
\captionof{table}{$2$-by-block-transitive actions with socle $\PSL_3(q)$ for $2 \le q \le 5$\label{table:small_q}}
\end{center}
\end{small}
\end{table}

Some remarks:
\begin{enumerate}[(1)]
\item There are no proper PD actions of $G = \PSL_3(2)$, and for a QP action, the parameters $(d_x,d_y)=(1,1)$ do not give a proper extension of $P_2(q)$ either, because 
\[
\GaL_1(4) = \GaL_2(2) = \SL_2(2).
\]
\item For $q=4$ there are no proper PD actions: this is because, in the notation of Proposition~\ref{prop:psl:standard}, we have $q-1 = n+1$, so $M$ is the kernel of the projective determinant map on $G_\GL(\alpha_0)$.  There are no QP actions of $G = \PSL_3(4)$ and $G = \mathrm{P}\Sigma\mathrm{L}_3(4)$ because $t_G=3$ and $e_G$ is not a multiple of $3$.
\item In the QP action of $\PGaL_3(4)$ or $\PGL_3(4)$ with block size $6$, each block can be identified with a copy of $P_1(5)$, with $G([\omega])$ acting as $\PGL_2(5)$ (if $G = \PGaL_3(4)$) or $\PSL_2(5)$ (if $G = \PGL_3(4)$) on the block.  In the exceptional action of $\PSL_3(5) = \PGaL_3(5)$ with block size $5$, each block can be identified with a copy of $P_1(4)$, with $G([\omega])$ acting as $\PGL_2(4) \cong \Sym(5)$.  These actions all appear as exceptional imprimitive rank $3$ actions in \cite{Rank3}.  The exceptional relationship between these $2$-by-block-transitive actions of $\PGaL_3(4)$ and $\PGaL_3(5)$ plays a role in the structure of boundary-$2$-transitive actions on the $(21,31)$-regular tree, see \cite{Reid}.
\end{enumerate}
\end{ex}

\subsection{Proofs of main theorems}\label{sec:proofs}

We can now prove the theorems from the introduction.

\begin{proof}[Proof of Theorem~\ref{thmintro:3bbtrans}]
Let $k \ge 3$ and suppose $G$ is a finite block-faithful $k$-by-block-transitive permutation group, such that the blocks are not singletons.  Then $G$ acts $k$-transitively on the set $X$ of blocks; let $G_0$ be the permutation group induced by $G$ on $X$.  By Corollary~\ref{cor:LDC_normal}, $G_0$ is not of affine type.  By Corollary~\ref{cor:no_ldc:symmetric}, $G_0$ is not $\Alt(X)$ or $\Sym(X)$.  By Proposition~\ref{prop:PSL:triples}, $G_0$ is not an action of $\PSL_2(q) \le G \le \PGaL_2(q)$ on the projective line.  By Lemma~\ref{lem:no_ldc:Mathieu}, $G_0$ is not a Mathieu group.  By the classification of finite $3$-transitive permutation groups, all possibilities are eliminated and we have a contradiction.
\end{proof}

Corollary~\ref{corintro:3bbtrans} now follows immediately from Theorem~\ref{thmintro:3bbtrans}, together with the explanation given in the introduction of the relationship between general $k$-by-block-transitive actions and block-faithful $k$-by-block-transitive actions.

\begin{proof}[Proof of Theorem~\ref{thmintro:2bbtrans}]
We have a finite block-faithful $2$-by-block-transitive permutation group $G$, acting on the set $\Omega$ and preserving some proper nontrivial system of imprimitivity; let $\omega \in \Omega$ and let $[\omega]$ be the block containing $\omega$.  Then $G$ has faithful $2$-transitive action on $G/G([\omega])$, and thus the pair $(G,G([\omega]))$ is subject to the classification of finite $2$-transitive permutation groups; in particular, either this action is affine type or $G$ is almost simple.  By Lemma~\ref{lem:unique_blocks}, $G([\omega])$ is the unique maximal subgroup of $G$ containing $G(\omega)$.  By Corollary~\ref{cor:LDC_normal}, $G$ has no nontrivial abelian normal subgroup, and hence $G$ is almost simple.  Let $G_0$ denote the permutation group given by the action of $G$ on $\Omega_0 = G/G([\omega])$.  We now go through the list of almost simple $2$-transitive permutation groups.

By Corollary~\ref{cor:no_ldc:symmetric}, $G_0$ cannot be a symmetric or alternating group in standard action.  By Lemma~\ref{lem:LDC_symplectic}, $G_0$ cannot belong to either of the families of $2$-transitive actions of symplectic groups.  By Lemma~\ref{lem:no_ldc:specific}, we exclude the cases where $(G_0,|\Omega_0|)$ is in the set
\[
\{(\PSL_2(11),11), (\Alt(7),15), (\PGaL_2(8),28), (\mathrm{HS},176), (\mathrm{Co}_3,276)\}.
\]
By Lemma~\ref{lem:no_ldc:Mathieu}, the only way $G_0$ can have socle a Mathieu group is if $G_0 = \mathrm{M}_{11}$ acting on $11$ points; $G([\omega]) = \mathrm{M}_{10}$; and $G(\omega) = \Alt(6)$.  This indeed produces a block-faithful $2$-by-block-transitive action of $\mathrm{M}_{11}$, which is line 1 in Table~\ref{table:exceptional}.

From now on we may assume the socle $S$ of $G_0$ is a group of Lie type in a standard $2$-transitive action.  If $S$ has Lie rank $1$, the possibilities are all accounted for by Proposition~\ref{prop:LDC:rank_one}, which also deals with case (c) of the theorem.  Thus we may assume that the Lie rank is at least $2$, which means that $\PSL_{n+1}(q) \le G \le \PGaL_{n+1}(q)$ and (up to isomorphism of permutation groups) we may assume $G([\omega])$ is a point stabilizer of the usual action of $G$ on the projective $n$-space $P_n(q)$.  Thus the action of $G$ on $\Omega$ is PD, QP or exceptional.

\paragraph{\underline{PD case.}}  All such actions are described by Proposition~\ref{prop:psl:standard}, which also accounts for case (a) of the theorem.

\paragraph{\underline{QP case.}} Here $n=2$, and all such actions are described by Proposition~\ref{prop:psl:plane-field}, which also accounts for case (b) of the theorem.

\paragraph{\underline{Exceptional case.}} If $G = \PSL_5(2)$, there is one exceptional action described by Lemma~\ref{lem:pgl52}, which is line 2 of Table~\ref{table:exceptional}, so let us assume $G \neq \PSL_5(2)$.  Then by Proposition~\ref{prop:psl:nonstandard} we have $n=2$, and $L$ does not contain $\SL_n(q)$.  We may thus assume Proposition~\ref{prop:psl_exceptional} applies; the resulting actions are displayed on lines 3 to 18 of Table~\ref{table:exceptional}, completing the argument for case (d) of the theorem.

One can easily check that cases (a)--(d) are mutually exclusive as claimed; see Example~\ref{ex:small_q} for more details on $\PSL_3(2)$ and $\PSL_3(3)$, which are to some extent special cases.
\end{proof}

\begin{proof}[Proof of Corollary~\ref{corintro:2bbtrans:socle}]
We suppose $G$ acts on $\Omega$, preserving an equivalence relation, such that $G$ acts transitively on $\Omega^{[2]}$.  Let $G([\omega]) \ge G(\omega)$ be the block and point stabilizers of the action, and let $K$ be the socle of $G([\omega])$; since $K$ is normal in $G([\omega])$ it is enough to show $K \le G(\omega)$.  If $G([\omega]) = G(\omega)$ the conclusion is clear, so assume $G([\omega]) > G(\omega)$.  Then we are in the setting of Theorem~\ref{thmintro:2bbtrans}.

If $\PSL_{n+1}(q) \le G \le \PGaL_{n+1}(q)$ for $n \ge 2$, we take $G([\omega])$ to be a point stabilizer of the usual action of $G$ on the projective $n$-space $P_n(q)$, then $K = W$, and we have $K \le G(\omega)$ by Proposition~\ref{prop:psl:nonstandard}.

If the socle of $G$ is of rank $1$ simple Lie type, then we are in case (c) of Theorem~\ref{thmintro:2bbtrans} and $K \le P$, where $P$ is as in Proposition~\ref{prop:LDC:rank_one}; again we have $K \le G(\omega)$.

The only remaining case is the exceptional $2$-by-block-transitive action of $G = \mathrm{M}_{11}$ on $22$ points.  In this case $G(\omega) = \Alt(6)$ is exactly the socle of $G([\omega]) = \mathrm{M}_{10}$.
\end{proof}

\begin{proof}[Proof of Theorem~\ref{thmintro:2bbtrans:sharp}]
We may assume that the action of $G$ on $\Omega$ is sharply $2$-by-block-transitive, but not sharply $2$-transitive.  This means the blocks are not singletons, and we are in the setting of Theorem~\ref{thmintro:2bbtrans}.

If case (a) or (c) of Theorem~\ref{thmintro:2bbtrans} holds, then by Lemma~\ref{lem:psl_2pt} or Corollary~\ref{cor:LDC:rank_one:sharp} respectively, the action is not sharply $2$-by-block-transitive.

If Theorem~\ref{thmintro:2bbtrans}(b) holds, then the characterization of sharply $2$-by-block-transitive action as in case (b) of the present theorem follows from Proposition~\ref{prop:psl:plane-field}.

If Theorem~\ref{thmintro:2bbtrans}(d) holds, then inspecting Table~\ref{table:exceptional}, we are left with six sharply $2$-by-block-transitive actions, with the field size $q$ as specified.  One can observe that all relevant values of $q$ are primes not congruent to $1$ modulo $3$, so indeed $\PSL_3(q) = \PGaL_3(q)$, and case (c) of the present theorem holds.
\end{proof}

\end{document}